\documentclass[twoside,a4paper,reqno,11pt]{amsart} 
\usepackage[top=28mm,right=30mm,bottom=28mm,left=30mm]{geometry}

\usepackage{amsfonts, amsmath, amssymb, mathrsfs, bm, latexsym, stmaryrd, array, hyperref, mathtools, bold-extra}

\usepackage{etaremune}
\usepackage{enumitem}

\renewcommand{\a}{\alpha}
\renewcommand{\b}{\beta}

\newcommand{\e}{\epsilon}
\renewcommand{\l}{\lambda}

\renewcommand{\O}{\Omega}

\newcommand{\la}{\langle}
\newcommand{\ra}{\rangle}
\newcommand{\C}{\mathcal{C}}

\newcommand{\leqs}{\leqslant}
\newcommand{\geqs}{\geqslant}

\newcommand{\normeq}{\trianglelefteqslant}

\newcommand{\vs}{\vspace{3mm}}

\newcommand{\FF}{\mathbb{F}}

\makeatletter
\newcommand{\imod}[1]{\allowbreak\mkern4mu({\operator@font mod}\,\,#1)}
\makeatother

\renewcommand{\leq}{\leqs}
\renewcommand{\geq}{\geqs}

\newtheorem{theorem}{Theorem} 
\newtheorem*{conj*}{Conjecture}

\newtheorem{corol}[theorem]{Corollary}

\newtheorem{thm}{Theorem}[section] 
\newtheorem{prop}[thm]{Proposition} 
\newtheorem{lem}[thm]{Lemma}
\newtheorem{cor}[thm]{Corollary}

\theoremstyle{definition}
\newtheorem{rem}[thm]{Remark}

\begin{document}

\title[Maximal overgroups of Sylow subgroups]{On the maximal overgroups of Sylow \\ subgroups of finite groups}

\author{Barbara Baumeister}
\address{B. Baumeister, Fakult\"{a}t f\"{u}r Mathematik, Universit\"{a}t Bielefeld, Postfach 10 01 31, 33501 Bielefeld, Germany}
\email{b.baumeister@math.uni-bielefeld.de}

\author{Timothy C. Burness}
\address{T.C. Burness, School of Mathematics, University of Bristol, Bristol BS8 1UG, UK}
\email{t.burness@bristol.ac.uk}

\author{Robert M. Guralnick}
\address{R.M. Guralnick, Department of Mathematics, University of Southern California, Los Angeles, CA 90089-2532, USA}
\email{guralnic@usc.edu}

\author{Hung P. Tong-Viet}
\address{H.P. Tong-Viet, Department of Mathematics and Statistics, Binghamton University, Binghamton, NY 13902-6000, USA}
\email{htongvie@binghamton.edu}

\renewcommand{\shortauthors}{Baumeister, Burness, Guralnick, Tong-Viet}

\begin{abstract}
In this paper, we determine the finite groups with a Sylow $r$-subgroup contained in a unique maximal subgroup. The proof involves a reduction to almost simple groups, and our main theorem extends earlier work of Aschbacher in the special case $r=2$. Several applications are presented. This includes some new results on weakly subnormal subgroups of finite groups, which can be used to study variations of the Baer-Suzuki theorem.
\end{abstract}

\date{\today}

\maketitle

\setcounter{tocdepth}{1}
\tableofcontents

\section{Introduction}\label{s:intro}

Let $G$ be a finite group, let $r$ be a prime divisor of $|G|$ and let $R$ be a Sylow $r$-subgroup of $G$. Let $\mathcal{M}(R)$ be the set of maximal subgroups of $G$ containing $R$. The main goal of this paper is to determine the triples $(G,r,H)$ with $\mathcal{M}(R) = \{H\}$.

In recent years, several problems of this flavour have been studied by various authors. For example, Aschbacher \cite[Theorem A]{asch80} has described the triples $(G,r,H)$ as above in the special case where $G$ is almost simple and $r=2$ (see \cite[Lemma 5.2]{asch80} for a more general result with $r=2$). Observe that if $r=2$ then $\mathcal{M}(R)$ includes a representative of every conjugacy class of odd-index maximal subgroups of $G$, and we note that the finite groups with a core-free maximal subgroup of odd index have been determined by Liebeck and Saxl \cite{LS85} (and independently Kantor \cite{Kan}). For certain odd primes $r$, the overgroups of Sylow $r$-subgroups of the general linear group ${\rm GL}_n(q)$ have been studied by Guralnick et al. \cite{GPPS}, finding a wide range of applications. There is also related work of Bamberg and Penttila \cite{BP08} in this direction. And in \cite{Ginsberg}, Ginsberg studies the quasisimple groups $G$ with a cyclic Sylow $r$-subgroup $R$ (for odd primes $r$) such that $\O_1(R) = \la x \, : x^r=1\ra$ is contained in a unique maximal subgroup of $G$, with applications to strongly $r$-embedded subgroups.

One of our main motivations stems from an extensive project of Guralnick and Tracey \cite{GT}, which seeks to describe the triples $(G,x,H)$, where $G$ is a finite group, $x \in G$ is an element and $H$ is the unique maximal subgroup of $G$ containing $x$, under the additional assumption that $H$ is core-free.  Of course, if $R = \la x \ra$ is a cyclic Sylow $r$-subgroup of $G$, then $\mathcal{M}(R) = \{H\}$ for some core-free subgroup $H$ if and only if $(G,x,H)$ is one of the triples arising in \cite{GT}. Our work in this paper is independent of \cite{GT}.

Our main results are also related to the study of weakly subnormal $r$-subgroups. Given a finite group $G$, we say that an $r$-subgroup $K$ of $G$ is \emph{weakly subnormal} in $G$ if it is subnormal in every proper overgroup of $K$ in $G$, but $K$ is not subnormal in $G$. For a non-subnormal $r$-subgroup $K$, Wielandt's Zipper Lemma (see \cite[Theorem 2.9]{Is}) implies that this is equivalent to saying that $\mathcal{M}(K) = \{H\}$ and $K \leqs O_r(H)$ for some maximal subgroup $H$ of $G$, where $O_r(H)$ denotes the $r$-core of $H$. In particular, if a Sylow $r$-subgroup $R$ of $G$ is not contained in a unique maximal subgroup, or if $\mathcal{M}(R) = \{H\}$ and $O_r(H) = 1$, then $G$ has no weakly subnormal $r$-subgroups. This observation allows us to use some of the results in this paper to establish various classification theorems on weakly subnormal $r$-subgroups. 

For example, suppose $G$ is an almost simple group (that is to say, the socle of $G$ is a nonabelian finite simple group $T$, which means that $T \normeq G \leqs {\rm Aut}(T)$). In Corollary \ref{c:NGR2}, we identify the groups of this form with a weakly subnormal Sylow $r$-subgroup. This is then extended in Corollary \ref{c:OrH}, which determines the triples $(G,r,K)$ such that $G$ is almost simple and $K$ is a maximal weakly subnormal $r$-subgroup. With some additional work, these results can be generalised to all finite groups (see \cite{GHT}) and this provides a general framework for studying Baer-Suzuki type problems for finite groups, building on earlier work in \cite{GM, GR}, for example. In particular, by combining some of the above results with the analysis in \cite{GT}, several new variations on the Baer-Suzuki theorem are proved in \cite{GHT}. 

Let us briefly outline this application. Let $G$ be a finite group, let $r$ be a prime and let $x \in G$ be an $r$-element. The Baer-Suzuki theorem states that if $\langle x,x^g\rangle$ is an $r$-group for all $g \in G$, then $x\in O_r(G)$. It is natural to consider variations, where the hypothesis is replaced by some other property $\mathcal{P}(x,g)$, which one assumes is satisfied for all elements $g$ in a specified nonempty subset 
$\mathcal{S}$ of $G$ (and we need the hypothesis to be valid in any subgroup of $G$ containing $x$). For instance, we could take $\mathcal{S}$ to be the set of all $r'$-elements in $G$ of prime power order, with $\mathcal{P}(x,g)$ the property that each commutator $[x,g]$ is an $r$-element. In order to explain the relevance of our main results to problems of this type, suppose $(G,\mathcal{S},x)$ is a counterexample with $|G|$ minimal, so our given property $\mathcal{P}(x,g)$ is satisfied for all $g \in \mathcal{S}$, but $x \not\in O_r(G)$. Then $x \in O_r(M)$ for every proper subgroup $M$ of $G$ containing $x$, so $\langle x\rangle$ is subnormal in every such subgroup $M$ and thus Wielandt's Zipper Lemma implies that $x$ is contained in a unique maximal subgroup $H$ of $G$. In particular, if $R$ is a Sylow $r$-subgroup of $G$ containing $x$, then $\mathcal{M}(R)=\{H\}$. In many cases, such a problem can be reduced to the almost simple case and this brings some of the results in this paper into play.

We now present our main results, focussing initially on the case where $G$ is almost simple with socle $T$. Our first main result handles the special case $r=2$, which extends \cite[Theorem A]{asch80} by presenting a precise description of the pairs $(G,H)$ with $\mathcal{M}(R) = \{H\}$. It is easy to see that $|\mathcal{M}(R)| = 1$ only if $G/T$ is a $2$-group (see Corollary \ref{c:easy}), which explains why this condition is included in the statement of Theorem \ref{t:prime2}. Referring to part (iii), Table \ref{tab:prime2} is presented at the end of the paper in Section \ref{s:tables}.

\begin{theorem}\label{t:prime2}
Let $G$ be an almost simple group with socle $T$, let $R$ be a Sylow $2$-subgroup of $G$ and assume $G/T$ is a $2$-group. Then $\mathcal{M}(R) = \{H\}$ if and only if one of the following holds:
\begin{itemize}\addtolength{\itemsep}{0.2\baselineskip}
\item[{\rm (i)}] $T = A_n$, $n = 2^k+1$, $k \geqs 2$ and $H = S_{n-1} \cap G$.
\item[{\rm (ii)}] $G = {}^2B_2(q)$ and $H = N_G(R) = q^{1+1}{:}(q-1)$ is a Borel subgroup.
\item[{\rm (iii)}] $T$ is a classical group and $(G,H)$ is recorded in Table \ref{tab:prime2}.
\end{itemize}
\end{theorem}

Note that in the statement of Theorem \ref{t:prime2} we assume $T \ne A_6, {\rm PSp}_4(2)'$ since both groups are isomorphic to ${\rm L}_2(9)$. We also exclude $T = {}^2G_2(3)'$ since it is isomorphic to ${\rm L}_2(8)$.

Next we turn to the case where $r$ is an odd prime and $G$ is almost simple with socle $T$. Once again, we note that $|\mathcal{M}(R)| = 1$ only if $G/T$ is an $r$-group. The tables referred to in parts (ii), (iii) and (iv) are presented in Section \ref{s:tables}, and we refer the reader to Remark \ref{r:main1} for further details on the cases that arise.
 
\begin{theorem}\label{t:main1}
Let $G$ be an almost simple group with socle $T$, let $R$ be a Sylow $r$-subgroup of $G$, where $r$ is an odd prime, and assume $G/T$ is an $r$-group. Then $\mathcal{M}(R) = \{H\}$ if and only if one of the following holds:
\begin{itemize}\addtolength{\itemsep}{0.2\baselineskip}
\item[{\rm (i)}] $G = T = A_n$ is an alternating group and one of the following holds:

\vspace{1mm}

\begin{itemize}\addtolength{\itemsep}{0.2\baselineskip}
\item[{\rm (a)}] $n=r^k+1$, $k \geqs 2$ and $H = A_{n-1}$.
\item[{\rm (b)}] $n=2r$ and $H = (S_r \wr S_2) \cap G$.
\item[{\rm (c)}] $n=r^2$ and $H = (S_r \wr S_r) \cap G$.
\item[{\rm (d)}] $n=r$, $r \neq 11,23$, $H = {\rm AGL}_1(r) \cap G$ and either $r=5$, or $r$ is not of the form $(q^d-1)/(q-1)$ with $q$ a prime power and $d\geqs 2$.
\end{itemize}
\item[{\rm (ii)}] $T$ is a classical group and $(G,r,H)$ is recorded in Table \ref{tab:main1}.
\item[{\rm (iii)}] $T$ is an exceptional group of Lie type and $(G,r,H)$ is recorded in Table \ref{tab:main2}.
\item[{\rm (iv)}] $G = T$ is a sporadic group and $(G,r,H)$ is one of the cases in Table \ref{tab:spor}.
\end{itemize}
\end{theorem}

We obtain the following result as a corollary to Theorems \ref{t:prime2} and \ref{t:main1} (see Section \ref{s:corols} for a proof).

\begin{corol}\label{c:NGR}
Let $G$ be an almost simple group with socle $T$, let $R$ be a Sylow $r$-subgroup of $G$ and assume $G/T$ is an $r$-group. Set $R_0 = R \cap T$. Then $\mathcal{M}(R) = \{N_G(R_0)\}$ if and only if $(G,r)$ is one of the cases recorded in Table \ref{tab:NGR}. 
\end{corol}  

Let $G,T,R$ and $R_0$ be defined as in Corollary \ref{c:NGR} and note that $R$ is weakly subnormal in $G$ if and only if $\mathcal{M}(R) = \{N_G(R)\}$. Since $N_G(R) \leqs N_G(R_0)$, the relevant groups $G$ with a weakly subnormal Sylow $r$-subgroup can be determined via Corollary \ref{c:NGR} since they coincide with the cases in Table \ref{tab:NGR} with $N_G(R)=N_G(R_0)$. In this way, we obtain  the following result (once again, the proof is given in Section \ref{s:corols}).

\begin{corol}\label{c:NGR2}
Let $G$ be an almost simple group with socle $T$ and let $R$ be a Sylow $r$-subgroup of $G$. Then $R$ is weakly subnormal in $G$ if and only if one of the following holds:
\begin{itemize}\addtolength{\itemsep}{0.2\baselineskip}
\item[{\rm (i)}] $G = T$ and $(G,r)$ is one of the cases in Table \ref{tab:NGR}.
\item[{\rm (ii)}] $r = 2$ and $G = {\rm PGL}_2(q)$, where $q \geqs 7$ is a Mersenne or Fermat prime.
\item[{\rm (iii)}] $r = 2$ and $G = {\rm L}_2(9).2$ or ${\rm L}_2(9).2^2$, with $G \ne {\rm P\Sigma L}_2(9)$.
\item[{\rm (iv)}] $r = 2$ and $G = {\rm L}_3(2).2$ or ${\rm L}_3(4).2_3$.
\item[{\rm (v)}] $r = 3$ and $G = {\rm L}_2(8).3$, ${\rm U}_3(8).3$ or ${\rm U}_3(8).3^2$.
\item[{\rm (vi)}] $r=5$ and $G = {}^2B_2(32).5$.
\end{itemize}
\end{corol}

Note that in part (iv) of Corollary \ref{c:NGR2}, the group ${\rm L}_3(4).2_3$ contains involutory graph automorphisms. And in part (v), we can take $G$ to be any of the three groups of the form ${\rm U}_3(8).3$.

\vs

Let $G$ be a finite group and let $K$ be a weakly subnormal $r$-subgroup of $G$  for some prime $r$. Then $\mathcal{M}(K)=\{H\}$ and $K\leq O_r(H)$ for some maximal subgroup $H$ of $G$. Now this implies that $\mathcal{M}(O_r(H))=\{H\}$ and so either $O_r(H)$ is a weakly subnormal $r$-subgroup of $G$ containing $K$, or $O_r(H)$ is subnormal in $G$. In the special case where $G$ is almost simple, it follows that $O_r(H)$ is a maximal weakly subnormal $r$-subgroup of $G$ and this occurs if and only if $\mathcal{M}(O_r(H))=\{H\}$. In particular, the latter property holds only if $\mathcal{M}(R)=\{H\}$ and $O_r(H)\neq 1$, where $R$ is a Sylow $r$-subgroup of $G$ containing $O_r(H)$. In the next corollary, we determine all such triples $(G,r,H)$.

\begin{corol}\label{c:Or}
Let $G$ be an almost simple group with socle $T$ and $\mathcal{M}(R) = \{H\}$, where $R$ is a Sylow $r$-subgroup of $G$. Then $O_r(H) \ne 1$ if and only if one of the following holds:
\begin{itemize}\addtolength{\itemsep}{0.2\baselineskip}
\item[{\rm (i)}] $H = N_G(R \cap T)$ and $(G,r,H)$ is one of the cases in Table \ref{tab:NGR}.
\item[{\rm (ii)}] $(G,r,H)$ is one of the cases in Table \ref{tab:Or}. 
\end{itemize}
\end{corol}

As a further corollary, we determine all the maximal weakly subnormal $r$-subgroups of almost simple groups. Note that the cases in part (i) are described in Corollary \ref{c:NGR2}; indeed, for the cases in Table \ref{tab:NGR}, we find that $\mathcal{M}(O_r(H))=\{H\}$ if and only if $H = N_G(R)$ (see Lemma \ref{l:equiv}). In (ii), $\phi$ is an involutory field automorphism and ${\rm L}_2(q).2_3$ is a nonsplit extension. 

\begin{corol}\label{c:OrH}
Let $G$ be an almost simple group with socle $T$ and $\mathcal{M}(R) = \{H\}$, where $R$ is a Sylow $r$-subgroup of $G$. Then $\mathcal{M}(O_r(H))=\{H\}$ if and only if one of the following holds:
\begin{itemize}\addtolength{\itemsep}{0.2\baselineskip}
\item[{\rm (i)}] $R$ is weakly subnormal in $G$.
\item[{\rm (ii)}]  $G= {\rm L}_2(q).2^2 = {\rm PGL}_2(q).\la \phi \ra$ or ${\rm L}_2(q).2_3$, $r= 2$ and $H$ is of type ${\rm GL}_1(q) \wr S_2$, where $q=81$ or $q=p^2$
with $p \geqs 5$ a Fermat prime. 
\item[{\rm (iii)}] $T = {\rm L}_2(q)$, $r=2$, $H$ is of type ${\rm GL}_1(q^2)$, $q \equiv 3 \imod{4}$ is a prime and $|R| \geqs 2^4$.
\item[{\rm (iv)}] $G = {\rm L}_3(3).2$, $r=2$ and $H$ is of type ${\rm GL}_2(3) \times {\rm GL}_1(3)$. 
\end{itemize}
\end{corol}

The above results on weakly subnormal $r$-subgroups of almost simple groups will play a key role in \cite{GHT}, where the general problem for arbitrary finite groups will be studied.

\vs

Let us now turn to the general case, where $G$ is an arbitrary finite group. Let us write $F^*(G) = F(G)E(G)$ for the \emph{generalised Fitting subgroup} of $G$, where $F(G)$ is the \emph{Fitting subgroup} (that is, the largest nilpotent normal subgroup of $G$) and $E(G)$ is the \emph{layer} of $G$, which is the product of the \emph{components} of $G$ (recall that a component is a subnormal \emph{quasisimple} subgroup, and a perfect group $L$ is quasisimple if $L/Z(L)$ is simple).

First note that $O_r(G)$ is contained in every Sylow $r$-subgroup of $G$ and so one can pass to $G/O_r(G)$. Write $\Phi(G/O_r(G))=D/O_r(G)$ and observe that 
$D$ is contained in every maximal subgroup of $G$ containing $R$. In Section \ref{s:red}, we will show that if $\mathcal{M}(R) = \{H\}$ and $R$ is not a normal subgroup of $G$, then $O_r(G/D) = \Phi(G/D) = 1$ (see Theorem \ref{t:main2gen}(i)). So this explains why we state the following result under the assumption $O_r(G)= \Phi(G)=1$  and we refer the reader to Theorem \ref{t:main2gen} for a more general statement.
  
\begin{theorem}\label{t:main2}  
Let $G$ be a finite group, let $r$ be a prime divisor of $|G|$ and let $R$ be a Sylow $r$-subgroup of $G$. Assume that $O_r(G)= \Phi(G)=1$. Then $R$ is contained in a 
unique maximal subgroup  $H$ of $G$ if and only if one of the following holds:
\begin{itemize}\addtolength{\itemsep}{0.2\baselineskip}
\item[{\rm (i)}] $G=P{:}R$ and $\mathcal{M}(R) = \{R\}$, where $P$ is an elementary abelian $p$-group for some prime $p \ne r$ and $R$ acts faithfully and irreducibly on $P$.
\item[{\rm (ii)}] $G$ is almost simple, $\mathcal{M}(R) = \{H\}$ and $(G,r,H)$ is one of the cases appearing in Theorem \ref{t:prime2} or \ref{t:main1}.
\item[{\rm (iii)}] $G=E(G)R$, where $E(G)$ is a direct product of two or more  nonabelian simple groups, and $R$ acts transitively on the set of components of $G$. Moreover, if $L$ is a component of $G$, then $N_R(L)/C_R(L)$ is a Sylow $r$-subgroup of the almost simple group $A = N_G(L)/C_G(L)$ and it is contained in a unique maximal subgroup $H$ of $A$, so $(A,r,H)$ is recorded in Theorem \ref{t:prime2} or \ref{t:main1}.
\end{itemize} 
\end{theorem} 

Note that the condition $O_r(G) = \Phi(G)=1$ in Theorem \ref{t:main2} could be replaced by the assumption that $H$ is core-free (since $H$ contains $O_r(G)\Phi(G)$). And conversely, the theorem implies that if $O_r(G) = \Phi(G)=1$ and $\mathcal{M}(R) = \{H\}$, then $H$ is core-free.

By a celebrated theorem of Thompson \cite{Tho}, a finite group with a nilpotent maximal subgroup of odd order is soluble. In addition, a nilpotent subgroup of a finite insoluble group is maximal only if it is the normaliser of a Sylow $2$-subgroup. For example, if $q = 2^m-1 \geqs 31$ is a Mersenne prime, then the Sylow $2$-subgroup $R = D_{q+1}$ is a maximal subgroup of $G = {\rm L}_2(q)$.
In \cite{Baum}, Baumann determines the structure of insoluble finite groups $G$ with a nilpotent maximal subgroup, showing that $O^2(G/F(G))$ is a direct product of nonabelian simple groups with dihedral Sylow $2$-subgroups (here the simple groups that can arise are of the form ${\rm L}_2(q)$, where $q=9$ or $q = 2^m \pm 1 \geqs 7$ is a prime). 

As an application of Theorem \ref{t:main2}, we can describe the structure of a finite group with a maximal Sylow $r$-subgroup. To do this, first note that a Sylow $r$-subgroup $R$ is maximal in $G$ if and only if $R/O_r(G)$ is maximal in $G/O_r(G)$, so we are free to assume that $O_r(G)=1$. In addition, if $R$ is maximal then it contains $\Phi(G)$, so $\Phi(G) \leqs O_r(G)$
and thus $O_r(G) = \Phi(G) = 1$ as in Theorem \ref{t:main2}.

\begin{corol}\label{c:max}
Let $G$ be a finite group and let $r$ be a prime divisor of $|G|$ with $O_r(G)=1$. Then a Sylow $r$-subgroup $R$ of $G$ is maximal if and only if one of the following holds:
\begin{itemize}\addtolength{\itemsep}{0.2\baselineskip}
\item[{\rm (i)}] $G = P{:}R$, where $P$ is an
elementary abelian $p$-group for some prime $p \ne r$ and $R$ acts faithfully and irreducibly on $P$.
\item[{\rm (ii)}] $r=2$ and $G = E(G)R$, where $E(G) = T^k$ is a minimal normal subgroup of $G$, $k \geqs 1$ and the almost simple group $N_G(T)/C_G(T)$ has a maximal Sylow $2$-subgroup. In particular, $N_G(T)/C_G(T)$ is one of the following:
\[
{\rm PGL}_2(7), \; {\rm PGL}_2(9), \; {\rm M}_{10}, \; {\rm L}_{2}(9).2^2, \; {\rm L}_2(q),\; {\rm PGL}_2(q),
\]
where $q >7$ is a Fermat or Mersenne prime.
\end{itemize}
\end{corol}

\vs

Let us briefly outline the structure of the paper. In Section \ref{s:prel} we present a number of preliminary results that will be needed in the proofs of our main theorems. In particular, we briefly discuss the subgroup structure of almost simple groups and we establish Corollary \ref{c:core}, which is a key tool in the almost simple setting. In Section \ref{s:red} we establish a strong form of Theorem \ref{t:main2}, which then allows us to focus on almost simple groups for the remainder of the paper. Theorems \ref{t:prime2} and \ref{t:main1} for groups with socle a sporadic or alternating group are established in Section \ref{s:t0}, and we complete the proof of Theorem \ref{t:prime2} in Section \ref{s:prime2}. We then focus on Theorem \ref{t:main1} in Sections \ref{s:t1} and \ref{s:t2}, treating the classical groups and exceptional groups in turn. Finally, we provide proofs of Corollaries \ref{c:NGR}-\ref{c:OrH} and \ref{c:max} in Section \ref{s:corols} and we finish by presenting the main tables for Theorems \ref{t:prime2} and \ref{t:main1} (and also Corollaries \ref{c:NGR}, \ref{c:Or} and \ref{c:OrH}) in Section \ref{s:tables}.

\vs

\noindent \textbf{Notation.} Our notation is standard. Let $G$ be a finite group and let $n$ be a positive integer. We will write $C_n$, or just $n$, for a cyclic group of order $n$ and $G^n$ will denote the direct product of $n$ copies of $G$. An unspecified extension of $G$ by a group $H$ will be denoted by $G.H$; if the extension splits then we may write $G{:}H$. We use $[n]$ for an unspecified soluble group of order $n$. We adopt the standard notation for simple groups of Lie type from \cite{KL}.

\vs

\noindent \textbf{Acknowledgements.} We thank David Craven for helpful discussions on the maximal subgroups of exceptional groups of Lie type and for providing us with some details from his forthcoming paper \cite{Craven4}. Guralnick was partially supported by the NSF grant DMS-1901595 and a Simons Foundation Fellowship 609771.  We also thank the referee for their careful reading of the paper
and helpful comments.  

\section{Preliminaries}\label{s:prel}

In this section we record some preliminary results that will be needed in the proof of Theorems \ref{t:prime2} and \ref{t:main1}.

\subsection{Prime divisors}\label{ss:pd}

Let $q$ be a prime power and let $d \geqs 2$ be a positive integer. Recall that a prime divisor $r$ of $q^d-1$ is a \emph{primitive prime divisor} if $q^e-1$ is indivisible by $r$ for all $1 \leqs e < d$. By a classical theorem of Zsigmondy \cite{Zsig}, such a prime $r$ exists unless $(d,q) = (6,2)$, or if $d = 2$ and $q$ is a Mersenne prime. Note that $r$ divides $q^{r-1}-1$ by Fermat's Little Theorem, so $r \equiv 1 \imod{d}$. Also note that $d$ is the order of $q$ modulo $r$, which we will sometimes denote by writing $d = d_q(r)$.

The following basic result will be a useful tool in our proof of Theorems \ref{t:prime2} and \ref{t:main1} for groups of Lie type. Here, and throughout the paper, we write $(m)_r$ for the largest power of $r$ dividing the positive integer $m$.  

\begin{lem}\label{l:gk}
Let $q$ be a prime power, let $d$ be a positive integer and let $r$ be a prime divisor of $q-\e$, where $\e=\pm$.
\begin{itemize}\addtolength{\itemsep}{0.2\baselineskip}
\item[{\rm (i)}] If $r=2$, then
\[
(q^d-\e)_r = \left\{\begin{array}{ll}
(q-\e)_r & \mbox{if $d$ is odd} \\
(q^2-1)_r(d/2)_r & \mbox{if $d$ is even and $\e=+$} \\
2 & \mbox{if $d$ is even and $\e=-$}
\end{array}\right.
\]
\item[{\rm (ii)}] If $r$ is odd, then
\begin{align*}
(q^d-\e)_r & = \left\{ \begin{array}{ll}
1 & \mbox{if $d$ is even and $\e=-$} \\
(q-\e)_r(d)_r & \mbox{otherwise}
\end{array}\right. \\
(q^d+\e)_r & = \left\{ \begin{array}{ll}
(q-\e)_r(d)_r & \mbox{if $d$ is even and $\e=-$} \\
1 & \mbox{otherwise.}
\end{array}\right.
\end{align*}
\end{itemize}
\end{lem}

\begin{proof}
The result for $(q^d-\e)_r$ is \cite[Lemma A.4]{BG}, so let us consider $(q^d+\e)_r$ with $r$ odd. 

If $r$ divides $q^d+1$ and $q-1$, then $r$ divides $q^{d-1}+1$ and we deduce that $r$ divides $(q+1,q-1) = (2,q-1)$, which is a contradiction. Therefore, $(q^d+\e)_r = 1$ if $\e=+$. Finally, suppose $\e=-$. If $d$ is odd and $r$ divides $q^d-1$, then $r$ divides $q^{d-1}+1$, but this is incompatible with the fact that $(q^{d-1}+1)_r = 1$ by the first part of the lemma. Now assume $d$ is even, say $d = 2^ab$ with $a \geqs 1$ and $b$ odd. If $a=1$ then
\[
(q^d-1)_r = (q^b-1)_r(q^b+1)_r = (q^b+1)_r = (q+1)_r(b)_r = (q+1)_r(d)_r  
\]
as required, and by induction we deduce that 
\[
(q^d-1)_r = (q^{2^{a-1}b}-1)_r(q^{2^{a-1}b}+1)_r = (q^{2^{a-1}b}-1)_r = (q+1)_r(d)_r 
\]
for $a \geqs 2$. The result follows.
\end{proof}

\subsection{Almost simple groups}\label{ss:as}

Next we present some initial observations in the special case where $G$ is an almost simple group with socle $T$. Let $R$ be a Sylow $r$-subgroup of $G$ and set $R_0 = R \cap T$, where $r$ is a prime divisor of $|G|$. 

First observe that $R_0$ may be trivial. For example, $R_0 = 1$ when $G = {}^2B_2(8){:}3$ and $r=3$ (recall that the order of a Suzuki group is indivisible by $3$), and if $G = {\rm L}_2(32){:}5$ and $r=5$. However, it is easy to show that $|\mathcal{M}(R)| \geqs 2$ in this situation. 

\begin{lem}\label{l:trivial}
Let $G$ be an almost simple group with socle $T$ and let $R$ be a Sylow $r$-subgroup of $G$ such that $R_0 = 1$. Then $|\mathcal{M}(R)| \geqs 2$.
\end{lem}

\begin{proof}
Seeking a contradiction, suppose $\mathcal{M}(R) = \{H\}$. Let $p_1, \ldots, p_k$ be the distinct prime divisors of $|T|$ and let $P_i$ be a Sylow $p_i$-subgroup of $T$. By the Frattini argument, we have $G = N_G(P_i)T$ and since $T$ is an $r'$-group we may assume that $R \leqs N_G(P_i) < G$. This implies that $H$ contains $N_G(P_i)$ and we deduce that $H = G$, which is a contradiction. We conclude that $|\mathcal{M}(R)| \geqs 2$ as required.
\end{proof}

Next we show that if $R_0 \ne 1$, then $\mathcal{M}(R)$ contains a core-free subgroup $H$ with the additional property that no other conjugate of $H$ contains $R$. 

\begin{lem}\label{l:sylow}
Let $G$ be an almost simple group with socle $T$ and let $R$ be a Sylow $r$-subgroup of $G$ with $R_0 \ne 1$. Then there exists a core-free subgroup $H \in \mathcal{M}(R)$ with $N_G(R) \leqs H$.
\end{lem}

\begin{proof}
First note that $R_0$ is a Sylow $r$-subgroup of $T$, so Frattini's argument yields $G = TN_G(R_0)$ and thus $G/T \cong N_G(R_0)/N_T(R_0)$. In particular, $R \leqs N_G(R_0) <G$ and thus  $R \leqs N_G(R_0) \leqs H$ for some maximal subgroup $H$ of $G$. Notice  that $H$ is core-free since it does not contain $T$. In addition, if $g \in N_G(R)$ then 
\[
R_0^g = (R \cap T)^g = R^g \cap T^g = R \cap T = R_0
\]
and we deduce that $R \leqs N_G(R) \leqs N_G(R_0) \leqs H$. The result follows.
\end{proof}

This has the following corollary, which is at the core of our proofs of Theorems \ref{t:prime2} and \ref{t:main1}. 

\begin{cor}\label{c:core}
Let $G$ be an almost simple group with socle $T$ and let $R$ be a Sylow $r$-subgroup of $G$ with $R_0 \ne 1$. Then $|\mathcal{M}(R)| = 1$ if and only if $G$ has a unique conjugacy class of maximal subgroups with $r'$-index.
\end{cor}

\begin{proof}
Let $H$ be a maximal subgroup of $G$ with $r'$-index and suppose $G$ has a unique conjugacy class of such subgroups. Without loss of generality, we may assume $R$ is contained in $H$. Then Lemma \ref{l:sylow} implies that $H$ is core-free and $N_G(R) \leqs H$, so $R$ is contained in a unique conjugate of $H$ and we conclude that $|\mathcal{M}(R)| = 1$. For the converse, suppose $H$ and $K$ represent distinct conjugacy classes of maximal subgroups with $r'$-index and let $R \leqs H$ and $S \leqs K$ be Sylow $r$-subgroups of $G$. Then $R$ and $S$ are $G$-conjugate, say $S = R^g$, and thus $R$ is contained in both $H$ and $ K^{g^{-1}}$. Therefore $|\mathcal{M}(R)| \geqs 2$ as required.
\end{proof}

The following corollary is another consequence.

\begin{cor}\label{c:easy}
Let $G$ be an almost simple group with socle $T$ and let $R$ be a Sylow $r$-subgroup of $G$ with $R_0 \ne 1$. Then $|\mathcal{M}(R)| = 1$ only if $G/T$ is an $r$-group. 
\end{cor}

\begin{proof}
Suppose $G/T$ is not an $r$-group. Then a Sylow $r$-subgroup of $G/T$ is contained in a maximal subgroup $H/T$ of $G/T$ and thus $H \in \mathcal{M}(R)$. By applying Lemma \ref{l:sylow} we conclude that $|\mathcal{M}(R)| \geqs 2$ and the result follows.
\end{proof}

\subsection{Subgroup structure}\label{ss:ss}

In order to prove Theorems \ref{t:prime2} and \ref{t:main1}, we need detailed information on the maximal subgroups of almost simple groups. Here we provide a very brief overview of some of the main results we will use. As before, let 
$G$ be an almost simple group with socle $T$. 

If $T$ is a sporadic simple group, then the maximal subgroups of $G$ have been determined up to conjugacy, with the exception of the Monster group 
$\mathbb{M}$. We refer the reader to Wilson's survey article \cite{Wil} for further details, noting that $\mathbb{M}$ has $44$ known conjugacy classes of maximal subgroups and any additional maximal subgroup must be almost simple with socle ${\rm L}_{2}(8)$,  ${\rm L}_{2}(13)$, ${\rm L}_{2}(16)$ or ${\rm U}_3(4)$. In a recent preprint \cite{DLP}, Dietrich et al. have used a novel computational approach to resolve this ambiguity: $\mathbb{M}$ has unique conjugacy classes of maximal subgroups isomorphic to ${\rm PGL}_{2}(13)$ or ${\rm Aut}({\rm U}_3(4))$, and none with socle ${\rm L}_2(8)$ or ${\rm L}_{2}(16)$. 

If $T$ is an alternating group, then the maximal subgroups of $G$ are described by the Aschbacher-O'Nan-Scott Theorem and our main reference is \cite{LPS}.

Next assume $T$ is an exceptional group of Lie type over $\mathbb{F}_q$, where $q = p^f$ and $p$ is a prime. Here the maximal subgroups of $G$ fall naturally into several collections (see \cite[Theorem 8]{LS03}) such as the parabolic subgroups, maximal rank subgroups (determined in \cite{LSS}), subfield subgroups and the so-called exotic local subgroups (see \cite{CLSS}). By combining  recent work of Craven \cite{Craven,Craven3} with earlier results by Borovik, Liebeck, Seitz and others, the maximal subgroups in each case have been determined up to conjugacy, with the exception of a collection of almost simple maximal subgroups, denoted $\mathcal{S}$ in this paper, that arises when $T = E_7(q)$ or $E_8(q)$. For $T = E_7(q)$, strong restrictions on the latter subgroups are obtained by Craven in \cite{Craven2}, with additional results for $T = E_8(q)$ forthcoming in \cite{Craven4}. These partial results will be sufficient for our purposes.

Finally, let us assume $T$ is a classical group over $\mathbb{F}_q$ with natural module $V$ of dimension $n$. Here the main theorem on the subgroup structure of classical groups is due to Aschbacher. In \cite{asch}, eight collections of subgroups of $G$ are defined, labelled $\mathcal{C}_i$ for $1 \leqs i \leqs 8$, and it is shown that every core-free maximal subgroup $H$ of $G$ is either  contained in one of these subgroup collections, or $H$ belongs to a family of almost simple subgroups which act irreducibly on $V$ (we will write $\mathcal{S}$ to denote the latter collection, which is described in more detail on p.3 of \cite{KL}). We refer the reader to Table \ref{tab0} for a rough description of the $\mathcal{C}_i$ collections, which comprise the \emph{geometric} maximal subgroups of $G$. A detailed analysis of these subgroups is given by Kleidman and Liebeck \cite{KL}, and throughout this paper we adopt the notation therein. In particular, we will refer repeatedly to the \emph{type} of a $\mathcal{C}_i$-subgroup $H$ (see \cite[p.58]{KL}), which provides an approximate description of the structure of $H$. For example, if $T = {\rm L}_n(q)$ and $H \in \mathcal{C}_2$ is the stabiliser in $G$ of a direct sum decomposition $V = V_1 \oplus V_2$ with $\dim V_i = m$, then we say that $H$ is a subgroup of type ${\rm GL}_m(q) \wr S_2$. The maximal subgroups of the low-dimensional classical groups with $n \leqs 12$ are determined up to conjugacy in \cite{BHR} and we will make extensive use of these results.

\begin{table}
\[
\begin{array}{ll} \hline
\C_1 & \mbox{Stabilisers of subspaces, or pairs of subspaces, of $V$} \\
\C_2 & \mbox{Stabilisers of direct sum decompositions $V=\bigoplus_{i=1}^{t}V_i$, where $\dim V_i  = m$} \\
\C_3 & \mbox{Stabilisers of prime index extension fields of $\mathbb{F}_q$} \\
\C_4 & \mbox{Stabilisers of tensor product decompositions $V=V_1 \otimes V_2$} \\
\C_5 & \mbox{Stabilisers of prime index subfields of $\mathbb{F}_q$} \\
\C_6 & \mbox{Normalisers of symplectic-type $r$-groups, $r \ne p$} \\
\C_7 & \mbox{Stabilisers of tensor product decompositions $V=\bigotimes_{i=1}^{t}V_i$, where $\dim V_i  = m$} \\
\C_8 & \mbox{Stabilisers of nondegenerate forms on $V$} \\ \hline
\end{array}
\]
\caption{The geometric subgroup collections of a classical group}
\label{tab0}
\end{table}

In order to prove Theorems \ref{t:prime2} and \ref{t:main1} for classical groups, we will divide the analysis into two cases. First we work closely with \cite{BHR,KL} and our initial goal is to identify two non-conjugate geometric maximal subgroups of $G$ with $r'$-index. Of course, this is not going to be possible in all cases and we will sometimes need to determine whether or not $\mathcal{M}(R)$ contains a subgroup in $\mathcal{S}$. If $n \leqs 12$ then we can do this by inspecting the relevant tables in \cite{BHR}. But a different approach is required for $n>12$. In this situation, the relevant prime $r$ will typically be a primitive prime divisor of $q^i-1$ and either $i \in \{1,2\}$ is small, or $i \in \{n,2n-2,2n\}$ is large.

If $i \in \{1,2\}$ then $T$ (and hence $R$) will contain an element $x$ of order $r$ with $\nu(x) \leqs 2$, where 
\begin{equation}\label{e:nu}
\nu(x) = \min\{\dim [\bar{V},\lambda \hat{x}]\,:\, \l \in \bar{\mathbb{F}}_q^{\times}\}
\end{equation}
is the codimension of the largest eigenspace of a lift $\hat{x} \in {\rm GL}(\bar{V})$ of $x$, where $\bar{V} = V \otimes \bar{\mathbb{F}}_q$ and $\bar{\mathbb{F}}_q$ is the algebraic closure of $\mathbb{F}_q$. We can then appeal to a theorem of Guralnick and Saxl \cite[Theorem 7.1]{GS} to severely restrict the  subgroups in $\mathcal{S}$ that contain such an element. Similarly, if $i \in \{n,2n-2,2n\}$ is large then we can work with the main theorem of \cite{GPPS} to study the relevant subgroups in $\mathcal{S}$. 

\section{The reduction theorem}\label{s:red}

In this section we prove Theorem \ref{t:main2}, which is our main result on the maximal overgroups of a Sylow subgroup of an arbitrary finite group. In particular, this result reduces the general problem to almost simple groups, which will be our main focus for the remainder of the paper.

First we recall some standard notation. For a finite group $G$, we write $\Phi(G)$ for the Frattini subgroup (the intersection of all the maximal subgroups of $G$) and $O_r(G)$ for the $r$-core of $G$, which is the largest normal $r$-subgroup of $G$. In addition, $F^*(G) = F(G)E(G)$ is the generalised Fitting subgroup of $G$, where $F(G)$ is the Fitting subgroup (the largest nilpotent normal subgroup of $G$) and $E(G)$ is the layer of $G$ (the product of the components of $G$, which coincide with the subnormal quasisimple subgroups of $G$).

We begin by presenting some general results. Our first lemma is due to Gasch\"utz \cite{Gas} (see \cite[Lemma 2.1(vii)]{AB}).

\begin{lem}\label{l:gas}
Let $G$ be a finite group and let $A$ be an abelian normal subgroup of $G$ with $A \cap \Phi(G) = 1$. Then there exists subgroup $B$ of $G$ such that $G = AB$ and $A \cap B = 1$. 
\end{lem}

The proof of the following result is an easy exercise. 

\begin{lem} \label{l:normal}   
Let $G=NK$ be a group, where $N$ is a normal subgroup of $G$. Then the map sending $H$ to $N \cap H$ defines a bijection from the set of maximal subgroups of $G$ containing $K$, to the set of maximal $K$-invariant subgroups of $N$ containing $N \cap K$.
\end{lem} 
 
Next we consider the case of coprime subgroups. The proof of the following result depends on the Schur-Zassenhaus Theorem (which in turn relies on the Feit-Thompson Theorem). However, it is worth noting that we only use
this in the proof of Lemma \ref{l:case1} when $K$ is an $r$-group and this only requires Sylow's theorem. 
  
\begin{lem} \label{l:coprime}   
Let $G = NK$ be a finite group, where $N$ and $K$ are subgroups of $G$ with $N \ne 1$ normal and $(|N|, |K|)=1$. Then $K$ is contained in a unique maximal subgroup of $G$ if and only if $N$ is a $p$-group for some prime $p$ and $K$ acts irreducibly on $N/\Phi(N)$.
\end{lem}
 
\begin{proof}  
First observe that $N$ is a $p$-group if and only if $N/\Phi(N)$ is a $p$-group. In addition, note that $K$ is contained in a unique maximal subgroup of $G$ if and only if $K\Phi(N)/\Phi(N)$ is contained in a unique maximal subgroup of $G/\Phi(N)$. Therefore, without any loss of generality, we may assume that $\Phi(N) = 1$. 

It is clear that if $N$ is a $p$-group and $K$ acts irreducibly on $N$, then $K$ is a maximal subgroup of $G$. So for the remainder, we may assume $K$ is contained in a unique maximal subgroup $H$ of $G$. 

Let $r_1, \ldots, r_t$ be the distinct prime divisors of $|N|$ and let
$R_i$ be a Sylow $r_i$-subgroup of $N$. Then for each $i$ we have
$G=NN_G(R_i)$ by the Frattini argument, which implies that 
$G/N = NN_G(R_i)/N \cong N_G(R_i)/N_N(R_i)$. By applying
the Schur-Zassenhaus Theorem we see that $N_G(R_i)$ contains a complement $L$ to $N_N(R_i)$, which is also a complement to
$N$ in $G$. Therefore, by another application of Schur-Zassenhaus, we may assume that $L = K$ and thus $K \leqs N_G(R_i)$. Note that $N = \la R_1, \ldots, R_t \ra$, so we have $G = \la KR_1, \ldots,  KR_t \ra$.

If $G=KR_i$ for some $i$,  then $N=R_i$ is an $r_i$-group. Otherwise, $t > 1$ and $K$ is contained in at least two maximal subgroups (since no proper subgroup of $G$ can contain all $KR_i$). Therefore, we may assume $N=R_1$, which is an elementary abelian $r_1$-group since $\Phi(N)=1$. In addition, since $|K|$ is indivisible by $r_1$,  $K$ acts completely reducibly on $N$.  Finally, we note that the maximal subgroups of $G$ containing $K$ are precisely 
the subgroups of the form $KJ$, where $J$ is a maximal $K$-invariant subgroup of $N$, whence $K$ acts irreducibly on $N$ and the proof is complete.
\end{proof}
 
In the next lemma we consider the case where $N$ is a direct product of simple groups. This result essentially follows from the proof of the Aschbacher-O'Nan-Scott Theorem \cite{AS}.  
 
\begin{lem}\label{l:simpleprod} 
Let $G=NK$ be a finite group, where $N=L_1 \times \cdots \times L_t$ is normal in $G$, $K$ is a subgroup of $G$ not containing $N$, and each $L_i$ is a nonabelian simple group. Assume that $K$ permutes the $L_i$ transitively and that $K_1 = L_1 \cap K$ is nontrivial. Then the map $T \mapsto L_1 \cap T$ defines a bijection from the set of maximal $K$-invariant subgroups of $N$ containing $N \cap K$ to the set of  maximal $N_K(L_1)$-invariant subgroups of $L_1$ containing $K_1$.
 \end{lem}
 
 \begin{proof}  
Since $K$ permutes the $L_i$ transitively, we have $N \cap K \geqs  K_1 \times \cdots \times K_t$, where each $K_i = L_i \cap K$ is a proper subgroup of $L_i$. Let $T$ be a maximal $K$-invariant subgroup of $N$ containing $N \cap K$. Let $\pi: N \to L_1$ be the projection map and set $T_1 = \pi(T) \geqs K_1$. If $T_1 = L_1$,  then $[T, K_1]=L_1 \leqs T$ and so the $K$-invariance of $T$ implies that $T=N$, a contradiction.  
Therefore, $K$ normalises $T_1 \times \cdots \times T_t$, where the $T_i$ are the $K$-conjugates of $T_1$, and so by maximality we have  
$T = T_1 \times \cdots \times T_t$. This shows that the map $T \mapsto L_1 \cap T$ is the desired bijection.
 \end{proof}  
 
In Theorem \ref{t:rfrattini} below we present a general result, which describes the intersection of all maximal subgroups with $r'$-index in an arbitrary finite group $G$. In particular, if $O_r(G)=1$, then this intersection coincides with the Frattini subgroup of $G$. To do this, we need a preliminary lemma.   This is well known in the case where $\pi$ consists of a single prime since we have $F(G/\Phi(G))=F(G)/\Phi(G)$.

\begin{lem} \label{l:aux}  
Let $G$ be a finite group and let $\pi$ be a set of prime numbers with $O_{\pi}(G)=1$. Then $O_{\pi}(G/\Phi(G))=1$. 
\end{lem}

\begin{proof} 
Set $N=\Phi(G)$ and note that $N$ is a $\pi'$-group. Seeking a contradiction, suppose that $O_{\pi}(G/N) \ne 1$. Then by the Schur-Zassenhaus Theorem, $G$ has a normal subgroup $NS$ for some nontrivial $\pi$-subgroup $S$ of $G$. By the Frattini argument and a second application of Schur-Zassenhaus, we deduce that $G=NN_G(S)$. But since $N=\Phi(G)$, we must have $G=N_G(S)$ and thus $S$ is a normal $\pi$-subgroup of $G$, which forces $S=1$ as $O_\pi(G)=1$. 
\end{proof}  

Let $D = D(G,r)$ be the subgroup of $G$ containing $O_r(G)$ such that $\Phi(G/O_r(G)) = D/O_r(G)$. We now show that $D$ is the intersection of all the maximal subgroups of $G$ with $r'$-index. 

\begin{thm}\label{t:rfrattini}
Let $G$ be a finite group and let $r$ be a prime divisor of $|G|$.  Then
\begin{itemize}\addtolength{\itemsep}{0.2\baselineskip}
\item[{\rm (i)}]  $O_r(G/D)=\Phi(G/D)=1$; and 
\item[{\rm (ii)}]  $D$ is the 
intersection of all   maximal subgroups of $G$ with $r'$-index.
\end{itemize} 
\end{thm}

\begin{proof}   
By Lemma \ref{l:aux}, $O_r(G/D)=1$.    Since $D/O_r(G)$ is the Frattini subgroup of $G/O_r(G)$, it follows that $\Phi(G/D) = 1$ and thus (i) holds.

Now let us turn to (ii). Let $A$ be the intersection of all the maximal subgroups of $G$ with $r'$-index. Clearly, $D$ is contained in every maximal subgroup with $r'$-index and so by (i), it is enough to show that $A=1$ if $D=1$. Let $R$ be a Sylow $r$-subgroup of $G$. Seeking a contradiction, assume that $D=1$ and $A \ne 1$. Let $B$ be a minimal normal subgroup of $G$ contained in $A$.

Since $O_r(G)=1$,  $B$ is not an $r$-group.  If $B$ is an $r'$-group, then $G=BM$ for some maximal subgroup $M$ (as $\Phi(G)=1$).  Then $M$ has $r'$-index but $M$ does not contain $B$,
a contradiction.    Thus, $B$ is a direct product of nonabelian simple groups of order divisible by $r$.   
 By the Frattini argument, $G=BN_G(R \cap B)$.  Since $R \cap B$ is not normal in $G$, it follows that $N_G(R \cap B) \leqs M$ for some maximal  subgroup $M$ of $G$. But $M$ has $r'$-index and it does not contain $B$, and so once again we have reached a contradiction and (ii) follows.  
 \end{proof} 

For the remainder of this section, let $G$ be a finite group, let $r$ be a prime divisor of $|G|$ and let $R$ be a Sylow $r$-subgroup of $G$. 
Let $\mathcal{M}(R)$ be the set of maximal subgroups of $G$ containing $R$.

Let us assume $R$ is contained in a unique maximal subgroup $H$ of $G$ (at this point, we are making no further assumptions on $G$, other than finiteness). The next result is a noteworthy observation.

\begin{lem}\label{l:gen}  
If $R$ is not normal in $G$, then $G = \langle R^g \,:\,  g \in G \rangle$.   
\end{lem}

\begin{proof}
Let $A= \langle R^g \,:\, g \in G \rangle$ and note that $A$ is normal in $G$, so $G = N_G(R)A$ by the Frattini argument. If $A \ne G$ then $R$ is contained in every maximal subgroup of $G$ containing $A$, so we must have $A \leqs H$. But we also have $N_G(R) \leqs H$ (since $N_G(R) \ne G$) and thus $H = G$, which is a contradiction.
\end{proof}

We are now in a position to prove Theorem \ref{t:main2}, under the assumption that we have already established Theorems \ref{t:prime2} and \ref{t:main1}. Indeed, the theorem follows by combining Lemmas \ref{l:case1} and \ref{l:case2} below, which consider separately the cases where $O_{r'}(G) \not\leqs \Phi(G)$ and $O_{r'}(G) \leqs \Phi(G)$. 

Note that if $R$ is normal in $G$, then $R$ is contained in a unique maximal subgroup of $G$ if and only if $G/R$ is a cyclic $p$-group for some prime $p \ne r$, and $R$ itself is maximal if and only if it has index $p$. So for the remainder of the section, we will assume that our Sylow $r$-subgroup $R$ is not normal in $G$.   

\begin{lem}\label{l:case1}
Suppose $O_r(G)=1$ and $O_{r'}(G) \not\leqs \Phi(G)$. Then $R$ is contained in a unique maximal subgroup $H$ of $G$ if and only if $G = P{:}R$, where $P$ is a $p$-group for some prime $p \ne r$ and $R$ acts faithfully and irreducibly on $P/\Phi(P)$. Moreover, $H=\Phi(P)R$ and the core of $H$ is $\Phi(G)=\Phi(P)$. 
\end{lem} 

\begin{proof}    
If $G=P{:}R$ and $R$ acts faithfully and irreducibly on $P/\Phi(P)$, then it is clear that $H=\Phi(P)R$ is the unique maximal subgroup of $G$ containing $R$ and the core of $H$ is $\Phi(P) = \Phi(G)$.  

For the converse, let us assume $R$ is contained in a unique maximal subgroup of $G$. Set $N=O_{r'}(G)$.  Since $N$ is not contained in $\Phi(G)$, there exists a maximal subgroup $H$ of $G$ that does not contain $N$. Then $G = NH$ and we may assume that $H$ contains $R$. Since $|\mathcal{M}(R)| = 1$, it follows that 
$G=N{:}R$ and thus Lemma \ref{l:coprime} tells us that $N$ is a $p$-group for some prime $p \ne r$ and $R$ acts irreducibly on $N/\Phi(N)$. It just remains to show that $R$ acts faithfully on $N/\Phi(N)$. To see this, set $C=C_R(N/\Phi(N))$ and note that $C$ centralises $N$ (indeed, we have $N=[C,N]C_N(C)$ and $[C,N] \leqs \Phi(N)$). Therefore, $C$ is trivial and the result follows.
\end{proof}

\begin{lem}\label{l:case2}
Suppose $O_r(G)=1$ and $O_{r'}(G) \leqs \Phi(G)$.  Let  $S=R\Phi(G)/\Phi(G)$ and define $J=G/\Phi(G)$ and $E = E(J)$. 
Then $R$ is contained in a unique maximal subgroup $H$ of $G$ if and only if the following hold:
\begin{itemize}\addtolength{\itemsep}{0.2\baselineskip}
\item[{\rm (i)}] $J=ES$ and $S$ acts transitively on the set of components of $J$. 
\item[{\rm (ii)}] Moreover, if $L$ is a component of $J$, then 
\[
N_S(L)C_J(L)/C_J(L) \cong N_S(L)/C_S(L)
\]
is a Sylow $r$-subgroup of the almost simple group $A = N_J(L)/C_J(L)$ and it is contained in a unique maximal subgroup of $A$.    
\end{itemize}
In addition, the core of $H$ is $\Phi(G)$.  
\end{lem} 

\begin{proof}    
First note that the final assertion follows from Theorem \ref{t:rfrattini}. Let us also note that the hypotheses imply that $F(G)=\Phi(G)=O_{r'}(G)$, so $F(J)=1$ and  $E \ne 1$.

Next observe that $R$ is contained in a unique maximal subgroup of $G$ if and only if $S$ is contained in a unique maximal subgroup of $J$. Therefore, without loss of generality, we may assume that $\Phi(G)=1$, so $G=J$ and $R = S$.  

Suppose that (i) and (ii) hold. Then $E$ is a direct product of isomorphic nonabelian simple groups and thus $R$ is contained in a unique maximal subgroup $H$ of $G$ by Lemmas \ref{l:normal} and \ref{l:simpleprod}. 

For the converse (still assuming that $G=J$), let us assume $\mathcal{M}(R) = \{H\}$. Since $F(G) = 1$ we have $E \ne 1$ and we claim that $G=ER$. To see this, first note that $R \leqs ER$ and the Frattini argument gives  $G=EN_G(E \cap R)$. Since $R \leqs N_G(E \cap R)$ and $R$ is contained in a unique maximal subgroup of $G$, it follows that $G=ER$.  

Let $X = \{L_1, \ldots, L_s\}$ be the set of components of $G$ and suppose $R$ acts intransitively on $X$. Then we can write $E = E_1E_2$, where the $E_i$ are proper normal subgroups of $G$. Then $R \leqs RE_i < G$, so $RE_1,RE_2 \leqs H$ and thus $G = \la R, E_1, E_2 \ra \leqs H$, which is a contradiction. Therefore (i) holds. Finally, in order to show that (ii) holds, we can pass to $G/Z(E)$, which means that we may assume $O_{r'}(G)=1$ and each $L_i$ is simple. In addition, $R \cap L_i$ is nontrivial for each $i$ and we conclude by applying  Lemma \ref{l:simpleprod}.    
\end{proof}

This completes the proof of Theorem \ref{t:main2}, modulo the proofs of Theorems \ref{t:prime2} and \ref{t:main1} in Sections \ref{s:t0} - \ref{s:t2}. We can also describe the general case.  

\begin{thm} \label{t:main2gen} 
Let $G$ be a finite group,  let $r$ be a prime divisor of $|G|$, let $R$ be a Sylow $r$-subgroup of $G$ and write $\Phi(G/O_r(G)) = D/O_r(G)$. Assume that $\mathcal{M}(R) = \{H\}$ and that $R$ is not normal in $G$. Then the following hold:  
\begin{itemize}\addtolength{\itemsep}{0.2\baselineskip}
\item[{\rm (i)}] $D$ is the core of $H$ in $G$ and we have $O_r(G/D) = \Phi(G/D)=1$.
\item[{\rm (ii)}] $G/D$ and $H/D$ are as in Theorem \ref{t:main2}.
\item[{\rm (iii)}] If $G$ is $r$-soluble, then $|G|$ is divisible by precisely two primes.
\item[{\rm (iv)}] If $E(G) \ne 1$, then $G=E(G)R$.  
\end{itemize}
\end{thm}  

\begin{proof}    
First note that $D$ is the core of $H$ by Theorem \ref{t:rfrattini}, so $H/D$ is core-free and thus $O_r(G/D) = \Phi(G/D)=1$ as in (i). Then Theorem \ref{t:main2} applies, giving (ii). If $G$ is $r$-soluble, then Lemma \ref{l:case1} yields (iii). Finally, we note that (iv) holds as in the proof of Lemma \ref{l:case2}.   
\end{proof}  

\begin{rem}
If we make the assumption $\Phi(G) \cap O_r(G)=1$, then Lemma \ref{l:gas} implies that $G = O_r(G){:}L$ is a semidirect product, where $O_r(L)=1$ and $L$ is described in Lemmas \ref{l:case1} and \ref{l:case2}.  
\end{rem}

\section{Sporadic and alternating groups}\label{s:t0}

In this section we prove Theorems \ref{t:prime2} and \ref{t:main1} for almost simple groups whose socle is a sporadic simple group or an alternating group.

\begin{prop}\label{p:spor}
The conclusion to Theorems \ref{t:prime2} and \ref{t:main1} hold when $T$ is a sporadic simple group.
\end{prop}

\begin{proof} 
In view of \cite[Theorem A]{asch80}, we may assume $r$ is odd, in which case $G = T$ since $|{\rm Out}(T)| \leqs 2$. If $G \ne \mathbb{M}$, then the character table of both $G$ and a representative of each conjugacy class of maximal subgroups of $G$ is available in the \textsf{GAP} Character Table Library \cite{GAPCTL}. So in these cases, noting  Corollary \ref{c:core}, it is an entirely straightforward exercise to determine the relevant triples $(G,r,H)$.

Finally, let us assume $G = \mathbb{M}$ is the Monster group. As discussed in \cite{Wil}, $G$ has $44$ known conjugacy classes of maximal subgroups and any additional maximal subgroup is almost simple with socle $S$ in the following list: 
\begin{equation}\label{e:list}
{\rm L}_{2}(8), \; {\rm L}_{2}(13), \; {\rm L}_{2}(16), \; {\rm U}_{3}(4).
\end{equation}
(As noted in Section \ref{ss:ss}, this ambiguity has very recently been resolved in \cite{DLP}.) By inspecting the list of known maximal subgroups, it is easy to check that $|\mathcal{M}(R)| \geqs 2$ if $r \not\in \{47,59,71\}$. And for each of the remaining primes we deduce that $|\mathcal{M}(R)| = 1$ since $|{\rm Aut}(S)|$ is indivisible by $r$ for each simple group $S$ in \eqref{e:list}. The result follows. 
\end{proof}

\begin{prop}\label{p:alt}
The conclusion to Theorems \ref{t:prime2} and \ref{t:main1} hold when $T$ is an alternating group.
\end{prop}

\begin{proof}
Here $T = A_n$ with $n \geqs 5$ and we note that $G = T$ if $r$ is odd. The groups with $n \leqs 24$ can be handled very easily using {\sc Magma} \cite{magma}, so we will assume $n \geqs 25$ for the remainder of the proof. Set $\O = \{1, \ldots, n\}$ and $R_0 = R \cap T$. We refer to \cite{LPS} for detailed information on the maximal subgroups of $G$. 

\vs

\noindent \emph{Case 1. $R$ is intransitive.}

\vs

To begin with, let us assume $R$ has $m \geqs 2$ orbits on $\O$ of length $\ell_1, \ldots, \ell_m$, where $\ell_i \geqs \ell_{i+1}$ for all $i$. Note that each $\ell_i$ is an $r$-power.

First assume $m \geqs 3$. Then $\mathcal{M}(R)$ contains $(S_{\ell_1} \times S_{n-\ell_1}) \cap G$ (or $(S_{\ell_1}\wr S_2) \cap G$ if $n = 2\ell_1$). In addition, $\mathcal{M}(R)$ also contains $(S_{\ell_1} \wr S_m) \cap G$ if $\ell_1 = \ell_m$, and $(S_{\ell_m} \times S_{n-\ell_m}) \cap G$ if $\ell_1 > \ell_m$, whence $|\mathcal{M}(R)| \geqs 2$ in this case.

Now suppose $m=2$, so $n=\ell_1+\ell_2$. As above, $\mathcal{M}(R)$ contains $(S_{\ell_1} \times S_{n-\ell_1}) \cap G$ (or $(S_{\ell_1}\wr S_2) \cap G$ if $n = 2\ell_1$). If $\ell_2>1$ then $n=tr$ for some $t \geqs 2$ and thus $(S_r \wr S_t) \cap G$ is a maximal subgroup with $r'$-index. Therefore, to complete the analysis of Case 1 we may assume that either $\ell_2 =1$ or $\ell_1 = \ell_2 = r$ (with $r$ odd in the latter case).

First assume $\ell_2=1$. Write $n = r^k+1$ for some $k \geqs 1$ and note that 
\[
|R_0| = \frac{r^{\sum_{i} \left\lfloor n/r^i \right\rfloor}}{1+\delta_{2,r}}  = r^{\frac{n-2-\delta_{2,r}}{r-1}},
\]
where $\delta_{2,r} = 1$ if $r=2$, otherwise $\delta_{2,r} = 0$. Plainly, $\mathcal{M}(R)$ contains an intransitive subgroup $H = S_{n-1} \cap G$ and we need to determine whether or not there are any additional subgroups in $\mathcal{M}(R)$. To this end, suppose $\mathcal{M}(R)$ contains a subgroup $L$ that is not conjugate to $H$. Then $L$ must be transitive on $\O$ and so there are two cases to consider. 

Suppose $L$ acts imprimitively on $\O$, say $L = (S_a \wr S_b) \cap G$ for some $a,b \geqs 2$ with $n = ab = r^k+1$. Suppose that $R$ fixes $n \in \O$ and let $\Delta \subseteq \Omega$ be an imprimitive block for $L$ of size $a$. Since $\{\Delta^x \, : \, x \in L\}$ is a partition of $\Omega$, we may assume that $n \in \Delta$. If $y \in R$ then $n^y = n \in \Delta^y \cap \Delta \ne \emptyset$ and thus  $\Delta^y = \Delta$. Therefore, $\Delta$ is a union of $R$-orbits. Finally, since  $R$ contains a cycle of length $n-1>a$, we deduce that $|\Delta| = a = 1$ and we have reached a contradiction.

We have now reduced to the case where $L$ acts primitively on $\O$ (we continue to assume that $\ell_2=1$). Since $R$, and hence $L$, contains a cycle of length $r^k = n-1$, we can use \cite[Theorem 1.2]{Jones} to determine the possibilities for $L$. As a consequence, either 

\begin{itemize}\addtolength{\itemsep}{0.2\baselineskip}
\item[(a)] ${\rm AGL}_d(q) \leqs L \leqs {\rm A\Gamma L}_d(q)$, where $n = q^d$, $d \geqs 1$ and $q = s^a$ is a prime power; or  
\item[(b)] $L = {\rm L}_2(s)$ or ${\rm PGL}_2(s)$, where $n=s+1$ and $s \geqs 5$ is a prime.
\end{itemize}
If (a) holds then $n = r^k+1 = s^{ad}$ and \cite[Lemma 2.6]{BTV} implies that $s=2$, $k=1$ and $ad$ is a prime. Similarly, if (b) holds then $n=r^k+1=s+1$, so $r^k=s$ and thus $r=s$ and $k=1$. So in both cases we have $n = r+1$, hence $r$ is odd, $G=T$ and $R$ is contained in a maximal subgroup ${\rm L}_2(r)$ of $G$ (this is maximal by the main theorem of \cite{LPS}, noting that $r \ne 23$ since $n \geqs 25$). It follows that $|\mathcal{M}(R)| \geqs 2$. 

In conclusion, if $m=2$ and $\ell_2 = 1$ then $n = r^k+1$ and $|\mathcal{M}(R)| = 1$ if and only if $k \geqs 2$ (see Theorem \ref{t:prime2}(i) and Theorem \ref{t:main1}(i)(a)).

To complete the analysis of Case 1, let us assume $m=2$ and $\ell_1 = \ell_2 = r$. Here $n=2r$, $|R| = r^2$, $r$ is odd and $\mathcal{M}(R)$ contains an imprimitive subgroup $H = (S_r \wr S_2) \cap G$. Suppose $L$ is another subgroup in $\mathcal{M}(R)$, which is not conjugate to $H$. Then $L$ must act primitively on $\O$ and it contains an $r$-cycle, so a classical theorem of Jordan (see \cite[Theorem 1.1]{Jones}, for example) implies that $L = G$, which is a contradiction. Therefore, $|\mathcal{M}(R)| = 1$ whenever $n = 2r$ (see case (i)(b) in Theorem \ref{t:main1}).

\vs

\noindent \emph{Case 2. $R$ is transitive and imprimitive.}

\vs

Here $n = r^k$ with $k \geqs 2$. If $k\geqs 3$, then $\mathcal{M}(R)$ contains $(S_r \wr S_{n/r}) \cap G$ and $(S_{n/r} \wr S_r) \cap G$, so we may assume $k=2$ and $r \geqs 5$. Note that $\mathcal{M}(R)$ contains the imprimitive maximal subgroup $H = (S_r \wr S_r) \cap G$. Suppose $L \in \mathcal{M}(R)$ and note that $L$ acts transitively on $\O$. Since $n=r^2>r+2$, it follows that $L$ contains an $r$-cycle fixing at least $n-r \geqs 3$ points. So if $L$ is primitive on $\O$, then Jordan's theorem implies that $L = G$, which is a contradiction. Therefore, $L$ acts imprimitively on $\Omega$ and the blocks in a nontrivial block system have size $r$. It follows that $L$ is conjugate to $H$ and by applying Corollary \ref{c:core} we conclude that $\mathcal{M}(R) = \{H\}$.

\vs

\noindent \emph{Case 3. $R$ is primitive.}

\vs

Finally, let us assume $R$ acts primitively on $\Omega$, in which case $|R| = n = r$ and $N_G(R) = {\rm AGL}_1(r) \cap G$ is a maximal subgroup of $G$. By a classical theorem of Burnside, it follows that either $\mathcal{M}(R) = \{ N_G(R) \}$, or $R< L < G$ for some $2$-transitive almost simple group $L$ with socle $S$. By inspecting \cite{GPP}, using the fact that $n \geqs 25$, we deduce that $S = {\rm L}_d(q)$ with $r = (q^d-1)/(q-1)$ is the only possibility. Such a subgroup $L$ does not contain $N_G(R)$ and thus $|\mathcal{M}(R)| = 1$ if and only if $n=r$ is not of this form. This completes the proof of the proposition.
\end{proof}

\section{Proof of Theorem \ref{t:prime2}}\label{s:prime2}

In this section we complete the proof of Theorem \ref{t:prime2}. In view of Propositions \ref{p:spor} and \ref{p:alt}, we may assume $T$ is a simple group of Lie type over $\mathbb{F}_q$, where $q=p^f$ and $p$ is a prime. Let $R$ be a Sylow $2$-subgroup of $G$ and set $R_0 = R \cap T$. In view of Corollary \ref{c:easy}, we may assume throughout that $G/T$ is a $2$-group. 

Recall that \cite[Theorem A]{asch80} describes the possibilities for $(G,H)$ with $\mathcal{M}(R) = \{H\}$ and our goal is to obtain a precise classification of the pairs $(G,H)$ that arise.

\begin{prop}\label{p:prime2_ex}
The conclusion to Theorem \ref{t:prime2} holds when $T$ is an exceptional group of Lie type.
\end{prop}

\begin{proof}
From \cite[Theorem A]{asch80} we see that $|\mathcal{M}(R)| = 1$ only if $T = {}^2B_2(q)$ and $H$ is a Borel subgroup. Here $G = T$ since we are assuming $G/T$ is a $2$-group, and it is clear that $\mathcal{M}(R) = \{H\}$ in this case.
\end{proof}

\begin{rem}
In the statement of Proposition \ref{p:prime2_ex} we are excluding the groups with socle $T = {}^2G_2(3)' \cong {\rm L}_2(8)$. Here it is easy to show that $|\mathcal{M}(R)| = 1$ in this case if and only if $G=T$.
\end{rem}

For the remainder of this section we may assume $T$ is a classical group and we divide the proof into several cases.

\begin{lem}\label{l:2}
The conclusion to Theorem \ref{t:prime2} holds when $T = {\rm L}_2(q)$.
\end{lem}

\begin{proof}
If $p=2$ then $\mathcal{M}(R) = \{H\}$ with $H = N_G(R)$, so we can assume $q$ is odd. Here the groups with $q \leqs 11$ can be handled using {\sc Magma}, so we will assume $q \geqs 13$. 

Suppose $q \equiv 1 \imod{4}$ and note that $|R_0| = (q-1)_2$. Since $q \geqs 13$, we observe that a $\C_2$-subgroup of type ${\rm GL}_1(q) \wr S_2$ is maximal and has odd index. We now divide the analysis into two cases according to $f$.

Suppose $f=1$, so $G=T$ or $T.2 = {\rm PGL}_2(q)$. By inspecting \cite[Tables 8.1, 8.2]{BHR} we deduce that $|\mathcal{M}(R)| = 1$ if $|R_0| \geqs 2^4$. On the other hand, if $|R_0| = 2^3$ then $S_4 \in \mathcal{M}(R)$ if $G = T$, whereas $|\mathcal{M}(R)| = 1$ if $G = T.2$. Finally, suppose $q \equiv 5 \imod{8}$, so $q \equiv \pm 1, \pm 3 \imod{10}$. If $G=T$ then either $q \equiv \pm 1 \imod{10}$ and $A_5 \in \mathcal{M}(R)$, or $q \equiv \pm 3 \imod{10}$ and $q \equiv 13,37 \imod{40}$, which means that $A_4 \in \mathcal{M}(R)$. Similarly, if $G = T.2$ then $S_4 \in \mathcal{M}(R)$ since $q \equiv 13,21,29,37 \imod{40}$.

Now assume $f>1$. Here we just need to determine whether or not $\mathcal{M}(R)$ contains a subfield subgroup $H$ of type ${\rm GL}_2(q_0)$, where $q = q_0^k$ with $k$ a prime. If $k$ is odd then $q_0 \equiv 1 \imod{4}$ and by applying Lemma \ref{l:gk} we deduce that $H$ has odd index. Similarly, if $k=2$ then $H_0 = {\rm PGL}_2(q_0)$ and thus $|H_0|_2 = (q_0^2-1)_2 = |R_0|$. Here $H$ is maximal in $G$ if and only if $G \leqs {\rm P\Sigma L}_2(q) = T.\la \phi \ra$, where $\phi$ is a field automorphism of order $f$. In summary, if $q \equiv 1 \imod{4}$ and $f>1$, then $|\mathcal{M}(R)| = 1$ if and only if $f$ is a $2$-power and $G \not\leqs {\rm P\Sigma L}_2(q)$.

To complete the proof, we may assume $q \equiv 3 \imod{4}$. Here $p \equiv 3 \imod{4}$, $f$ is odd, $|R_0| = (q+1)_2$ and a $\C_3$-subgroup of type ${\rm GL}_1(q^2)$ has odd index. If $f >1$ then it is easy to see that every maximal subfield subgroup has odd index, so $|\mathcal{M}(R)| = 1$ only if $f=1$. Then by  arguing as in the case $q \equiv 1 \imod{4}$ we can show that $|\mathcal{M}(R)| = 1$ if and only if $|R_0| \geqs 2^4$, or if $|R_0| = 2^3$ and $G = {\rm PGL}_2(q)$.
\end{proof}

\begin{lem}\label{l:3}
The conclusion to Theorem \ref{t:prime2} holds when $T = {\rm L}_3^{\e}(q)$.
\end{lem}

\begin{proof}
First assume $q$ is even. If $\e=+$ then by inspecting \cite[Tables 8.3, 8.4]{BHR} we see that $|\mathcal{M}(R)| = 1$ if and only if $G \not\leqs {\rm P\Gamma L}_3(q)$, in which case a Borel subgroup $H = P_{1,2}$ has odd index.  Similarly, if $\e=-$ then $\mathcal{M}(R) = \{H\}$ with $H = P_1$ (the stabiliser in $G$ of a totally singular $1$-space).

For the remainder, we may assume $q \equiv -\e \imod{4}$ (see \cite[Theorem A]{asch80}), which means that $|R_0| = 4(q+\e)_2$. Suppose $\e=+$, so $p \equiv 3 \imod{4}$ and $f$ is odd. By \cite[Theorem A]{asch80}, $|\mathcal{M}(R)| = 1$ only if $G \not\leqs {\rm P\Gamma L}_3(q)$ and we see that a $\C_1$-subgroup of type ${\rm GL}_2(q) \times {\rm GL}_1(q)$ has odd index. In addition, every maximal subfield subgroup also has odd index and we quickly deduce that $|\mathcal{M}(R)| = 1$ if and only if $f=1$ and $G \not\leqs {\rm P\Gamma L}_3(q)$. Finally, suppose $\e=-$ and $q \equiv 1 \imod{4}$, in which case $N_1$ has odd index (this is the stabiliser of a nondegenerate $1$-space). If $q=5$ then $|R_0| = 2^4$ and $\mathcal{M}(R)$ contains a subgroup in $\mathcal{S}$ with socle $A_6$, so we can assume $q \geqs 9$. Now every maximal subfield subgroup has odd index, so $|\mathcal{M}(R)| = 1$ only if $f$ is a $2$-power and by inspecting \cite[Tables 8.5, 8.6]{BHR} we conclude that $|\mathcal{M}(R)|$ is indeed equal to $1$ in this case.
\end{proof}

\begin{lem}\label{l:4}
The conclusion to Theorem \ref{t:prime2} holds when $T = {\rm L}_n^{\e}(q)$ with $n \geqs 4$.
\end{lem}

\begin{proof}
Here \cite[Theorem A]{asch80} states that $|\mathcal{M}(R)| = 1$ only if $n = 2^k+1$ and $q \equiv -\e \imod{4}$ (with the additional condition $G \not\leqs {\rm P\Gamma L}_n(q)$ if $\e=+$), in which case a $\C_1$-subgroup of type ${\rm GL}_{n-1}^{\e}(q) \times {\rm GL}_1^{\e}(q)$ has odd index. Using Lemma \ref{l:gk} we compute   
\[
|R_0| = 2^{n-1}(q+\e)_2^{(n-1)/2}\left((n-1)/2\right)!_2 = 2^{(3n-5)/2}(q+\e)_2^{(n-1)/2}.
\]

First assume $\e=+$, so $p \equiv 3 \imod{4}$ and $f$ is odd. Since every maximal subfield subgroup of $G$ has odd index, we may assume $f=1$. By inspecting \cite{BHR,KL}, it is straightforward to check that every geometric maximal subgroup has even index (apart from the $\C_1$-subgroups of type ${\rm GL}_{n-1}(q) \times {\rm GL}_1(q)$, as noted above), so it remains to determine if there are any odd-index maximal subgroups in the collection $\mathcal{S}$. By inspecting \cite[Table 8.19]{BHR} we see that there are no such subgroups when $n=5$. And we reach the same conclusion for $n \geqs 9$ by applying \cite[Theorem 7.1]{GS}, noting that $R_0$ contains an involution of the form $x = (-I_{n-1},I_1)$ with $\nu(x) = 1$ (see \eqref{e:nu}). In conclusion, if $T = {\rm L}_n(q)$ and $n \geqs 4$, then $|\mathcal{M}(R)| = 1$ if and only if $n=2^k+1$, $q=p \equiv 3 \imod{4}$ and $G \not\leqs {\rm P\Gamma L}_n(q)$.

Finally, suppose $\e=-$. Once again we find that there are no $\mathcal{S}$ collection subgroups in $\mathcal{M}(R)$ and every maximal subfield subgroup has odd index. From here, it is straightforward to deduce that 
$|\mathcal{M}(R)| = 1$ if and only if $n=2^k+1$, $q \equiv 1 \imod{4}$ and $f$ is a $2$-power.
\end{proof}

\begin{lem}\label{l:6}
The conclusion to Theorem \ref{t:prime2} holds when $T = {\rm PSp}_n(q)$ with $n \geqs 4$.
\end{lem}

\begin{proof}
Since ${\rm PSp}_4(2)' \cong {\rm L}_2(9)$, we may assume $(n,q) \ne (4,2)$. 
By \cite[Theorem A]{asch80} we have $|\mathcal{M}(R)| = 1$ only if $n=4$, $q$ is even and $G \not\leqs {\rm P\Gamma Sp}_4(q)$. Indeed, by inspecting \cite[Table 8.14]{BHR} we see that a Borel subgroup is the unique maximal overgroup of $R$.
\end{proof}

\begin{lem}\label{l:7}
The conclusion to Theorem \ref{t:prime2} holds when $T = \O_n(q)$ with $n \geqs 7$ odd.
\end{lem}

\begin{proof}
By \cite[Theorem A]{asch80} we have $|\mathcal{M}(R)| = 1$ only if $n = 2^k+1$, in which case $\mathcal{M}(R)$ contains a $\C_1$-subgroup of type ${\rm O}_{n-1}(q) \perp {\rm O}_1(q)$. Note that 
\[
|R_0| = \frac{1}{2}(q^2-1)_{2}^{(n-1)/2}((n-1)/2)!_2 = 2^{(n-5)/2}(q^2-1)_{2}^{(n-1)/2}.
\]
If $q = q_0^k$ and $k$ is a prime, then it is easy to see that a subfield subgroup of type ${\rm O}_{n}(q_0)$ has odd index if and only if $k$ is odd, so we may assume $f$ is a $2$-power. Now, if $f=1$ and $q \equiv \pm 3 \imod{8}$ then $(q^2-1)_2 = 2^3$, $|R_0| = 2^{2n-4}$ and one can check that $\mathcal{M}(R)$ contains a $\C_2$-subgroup of type ${\rm O}_1(q) \wr S_n$. Therefore, if $f=1$ then $|\mathcal{M}(R)| = 1$ only if $q \equiv \pm 1 \imod{8}$. 

By inspecting \cite{BHR,KL}, it remains to determine whether or not $\mathcal{M}(R)$ contains a subgroup in $\mathcal{S}$. But this possibility can be ruled out by applying \cite[Theorem 7.1]{GS}, using the fact that $T$ contains an involution of the form $(-I_{n-1},I_1)$.
\end{proof}

Finally, we handle the even-dimensional orthogonal groups.

\begin{lem}\label{l:8}
The conclusion to Theorem \ref{t:prime2} holds when $T = {\rm P\O}_n^{\e}(q)$ with $n \geqs 8$ even.
\end{lem}

\begin{proof}
Here \cite[Theorem A]{asch80} gives $n = 2^k+2$ and $q \equiv -\e \imod{4}$, so $T = \O_n^{\e}(q) = {\rm PSO}_n^{\e}(q)$ and we compute
\[
|R_0| = 2^{n/2-1}(q^2-1)_2)^{n/2-1}.
\]

First assume $\e=+$, so $p \equiv 3 \imod{4}$ and $f$ is odd. If $q = q_0^k$ for some prime $k$, then it is straightforward to show that $\mathcal{M}(R)$ contains a subfield subgroup of type ${\rm O}_{n}^{+}(q_0)$. Therefore, we may assume $f=1$ and thus ${\rm Aut}(T) = {\rm PGO}_{n}^{+}(q) = T.\la \delta, \gamma\ra$. If $G$ is contained in $T.\la \gamma \ra = {\rm PO}_{n}^{+}(q)$, then it has two conjugacy classes of odd index maximal subgroups of type ${\rm O}_{n-1}(q) \perp {\rm O}_1(q)$ and thus $|\mathcal{M}(R)| \geqs 2$. On the other hand, if $G \not\leqs T.\la \gamma \ra$ then $\mathcal{M}(R)$ contains a $\C_1$-subgroup of type ${\rm O}_{n-2}^{+}(q) \perp {\rm O}_2^{+}(q)$, and by combining \cite{BHR,KL} with \cite[Theorem 7.1]{GS} we deduce that $|\mathcal{M}(R)|=1$.

Now assume $\e=-$, so $q \equiv 1 \imod{4}$ and ${\rm Aut}(T) = T.\la \delta,\varphi \ra$, where $|\delta|=2$ and $|\varphi| = 2f$. We note that $\mathcal{M}(R)$ contains a $\C_1$-subgroup of type ${\rm O}_{n-2}^{+}(q) \perp {\rm O}_{2}^{-}(q)$. If $G \leqs T.\la \varphi \ra$ then $G$ has two classes of odd index maximal subgroups of type ${\rm O}_{n-1}(q) \perp {\rm O}_1(q)$, so we may assume $G \not\leqs T.\la \varphi \ra$. If $q=q_0^k$ and $k$ is an odd prime, then a subfield subgroup of type ${\rm O}_{n}^{-}(q_0)$ has odd index, so we may assume $f$ is a $2$-power. Finally, by inspecting \cite{BHR,KL} and applying \cite[Theorem 7.1]{GS} we deduce that $|\mathcal{M}(R)|=1$ if and only if $G \not\leqs T.\la \varphi \ra$ and $f$ is a $2$-power. 
\end{proof}

This completes the proof of Theorem \ref{t:prime2}.

\section{Proof of Theorem \ref{t:main1}: Classical groups}\label{s:t1}

In this section we prove Theorem \ref{t:main1} in the case where $T$ is a finite simple classical group over $\mathbb{F}_q$ with natural module $V$. As before, write $q=p^f$ with $p$ a prime and set $n = \dim V$. We use $P_m$ to denote the stabiliser in $G$ of a totally singular $m$-dimensional subspace of $V$ (if $T = {\rm L}_n(q)$, then we adopt the convention that every subspace of $V$ is totally singular). In addition, we set $H_0 = H \cap T$ for any subgroup $H$ of $G$. 

Throughout this section, $r$ is an odd prime divisor of $|G|$ and $R$ is a Sylow $r$-subgroup of $G$. In view of Lemma \ref{l:trivial} and Corollary \ref{c:easy}, we may assume $r$ divides $|T|$ and $G/T$ is an $r$-group. If $r \ne p$, then we write $d_q(r)$ for the order of $q$ modulo $r$ (in which case, $r$ is a primitive prime divisor of $q^i-1$ for $i=d_q(r)$). We will also adopt the following notation (see Remark \ref{r:main1}(b)): if $m$ is a positive integer and $\e \in \{\pm 1\}$, then we will use $\a(m,\e)$ and $\b(m,\e)$ to denote the following conditions on $q$:
\begin{align}\label{e:ab}
\begin{split}
\a(m,\e)\!:  & \;\; \mbox{$(q^{1/k})^m \not\equiv \e \imod{r}$ for all $k \in \pi(f) \setminus \{r\}$} \\
\b(m,\e)\!: & \;\; \mbox{$(q^{1/k})^m \not\equiv \e \imod{r}$ for all $k \in \pi(f) \setminus \{2,r\}$} 
\end{split}
\end{align}
where $\pi(f)$ is the set of prime divisors of $f$. For an integer $s$, we also write $s = \square \imod{p}$ and $s = \boxtimes \imod{p}$ to denote that $s$ is a square or nonsquare modulo $p$, respectively. 

We now divide the proof into several cases according to the possibilities for $T$. 
As discussed in Section \ref{ss:ss}, we will refer to subgroups in the $\mathcal{C}_i$ and $\mathcal{S}$ collections, working with the definitions presented in \cite{KL}, which differ slightly from Aschbacher's set up in \cite{asch}.

\subsection{Linear groups}

\begin{prop}\label{p:psl2}
The conclusion to Theorem \ref{t:main1} holds when $T = {\rm L}_2(q)$.
\end{prop}

\begin{proof}
The groups with $q \leqs 11$ can be handled using {\sc Magma}, so we will assume $q \geqs 13$. If $r=p$ then $\mathcal{M}(R) = \{H\}$, where $H = N_G(R_0) = P_1$ is a Borel subgroup of $G$. For the remainder, we may assume $r \ne p$, so $r$ divides $q-1$ or $q+1$.

\vs

\noindent \emph{Case 1. $r$ divides $q-1$.}

\vs

Here $R$ is contained in a Borel subgroup and a $\C_2$-subgroup of type ${\rm GL}_1(q) \wr S_2$, both of which are maximal since $q \geqs 13$ (see \cite[Table 8.1]{BHR}). Therefore $|\mathcal{M}(R)| \geqs 2$.

\vs

\noindent \emph{Case 2. $r$ divides $q+1$.}

\vs

In this case $R$ is contained in a maximal $\C_3$-subgroup $H = N_G(R)$ of type ${\rm GL}_1(q^2)$. We now need to inspect \cite[Tables 8.1, 8.2]{BHR} in order to determine whether or not $G$ has another maximal subgroup with $r'$-index. There are several cases to consider.

First assume $f=1$, in which case $G = T$ since $G/T$ is an $r$-group. If $r \geqs 7$ then it is easy to see that $H$ is the only maximal subgroup with $r'$-index and thus $|\mathcal{M}(R)| = 1$. The same conclusion holds if $r \in \{3,5\}$ and $|R| >r$, so we may assume $r \in \{3,5\}$ and $|R| = r$. If $r=5$ then $q = p \equiv -1 \imod{10}$ and thus $\mathcal{M}(R)$ contains a subgroup $A_5$ in the collection $\mathcal{S}$. Similarly, if $r=3$ then we find that $\mathcal{M}(R)$ contains one of $A_4$, $S_4$ (both in $\mathcal{C}_6$) or $A_5$ (in $\mathcal{S}$).

Next assume $f=2$, so $G=T$ once again. If $r=3$ then $r$ divides $q-1 = (p-1)(p+1)$, so we must have $r \geqs 5$. The order of a subfield subgroup of type ${\rm GL}_2(p)$ is indivisible by $r$, so $|\mathcal{M}(R)| = 1$ if $r \geqs 7$, or if $r = 5$ and $|R| >r$. On the other hand, if $|R| = r = 5$, then $p \equiv \pm 3 \imod{10}$ and $\mathcal{M}(R)$ contains a subgroup $A_5$ in the collection $\mathcal{S}$.

Finally, suppose $f>2$. We need to determine whether or not 
$\mathcal{M}(R)$ contains a subfield subgroup $L$ of type ${\rm GL}_2(q_0)$, where $q=q_0^k$ for some prime divisor $k$ of $f$ (with the additional condition $k \ne f$ when $p=2$; see \cite[Table 8.1]{BHR}). If $k=2$ then $|L_0|$ is indivisible by $r$, so we may assume $k$ is odd and thus $L$ has $r'$-index if and only if $(q_0+1)_r = (q+1)_r$. By Lemma \ref{l:gk}, this condition holds if and only if $r \ne k$ and $r$ divides $q_0+1$. Therefore, $|\mathcal{M}(R)|=1$ if and only if $\a(1,-1)$ holds (see \eqref{e:ab}), or if $(r,p) = (3,2)$ and $q^{1/k} \not\equiv -1 \imod{3}$ for all $k \in \pi(f) \setminus \{3,f\}$.
\end{proof}

\begin{prop}\label{p:psln}
The conclusion to Theorem \ref{t:main1} holds when $T = {\rm L}_n(q)$ and $n \geqs 3$.
\end{prop}

\begin{proof}
Recall that $G/T$ is an $r$-group, so we have $G \leqs {\rm P\Gamma L}_n(q)$. In particular, $G$ does not contain any  graph or graph-field automorphisms of $T$.

If $r=p$ then every parabolic subgroup of $G$ has $r'$-index and hence $|\mathcal{M}(R)| \geqs 2$. For the remainder, we may assume $r \ne p$ divides $|T|$. Let $i = d_q(r)$ be the order of $q$ modulo $r$ and note that $i$ divides $r-1$.

\vs

\noindent \emph{Case 1. $i=1$}

\vs

Here $q \geqs 4$ and we note that 
\[
|R_0| = \frac{1}{(d)_r}(q^2-1)_r(q^3-1)_r \cdots (q^n-1)_r = \frac{1}{(d)_r}((q-1)_r)^{n-1}(n!)_r
\]
by Lemma \ref{l:gk}, where $d=(n,q-1)$. If $n$ is indivisible by $r$ then $(q^n-1)_r=(q-1)_r$ and it is easy to see that the parabolic subgroups $P_1$ and $P_{n-1}$ have $r'$-index, whence $|\mathcal{M}(R)| \geqs 2$. So for the remainder of Case 1 we may assume that $r$ divides $n$. 

Let $L$ be a $\C_2$-subgroup of type ${\rm GL}_{1}(q) \wr S_{n}$, which is maximal if $q \geqs 5$ (see \cite{BHR,KL}). The order of $L_0$ can be computed via \cite[Proposition 4.2.9]{KL} and we quickly deduce that $L$ has $r'$-index. In addition, if $r<n$ then a $\C_2$-subgroup of type ${\rm GL}_{r}(q) \wr S_{n/r}$ also has $r'$-index (and such a subgroup is always maximal). Therefore, we may assume that either $q=4$  or $n=r$. 

Suppose $q=4$, so $r=3$ and $n=3m$ for some $m \geqs 1$. Here it is easy to check that a $\C_8$-subgroup $H$ of type ${\rm GU}_{n}(2)$ has $3'$-index. Indeed, we have 
\begin{align*}
|R_0| & = \frac{1}{3}(4^2-1)_3(4^3-1)_3 \cdots (4^{3m}-1)_3 \\
& = \frac{1}{3}(2^2-1)_3(2^3+1)_3 \cdots (2^{3m}-(-1)^{3m})_3 = |H_0|_3.
\end{align*}
In addition, if $m \geqs 2$ then $\mathcal{M}(R)$ also contains a $\C_2$-subgroup of type ${\rm GL}_{3}(q) \wr S_{m}$. Finally, if $m=1$ then a {\sc Magma} computation shows that $|\mathcal{M}(R)| = 1$ if and only if $G = {\rm PGL}_3(4)$ (for example, if $G = T$ then $R = 3^2$ is contained in maximal subgroups of type ${\rm GU}_{3}(2)$ and $A_6$, where the latter is in the collection $\mathcal{S}$).

To complete the analysis of the case $i=1$, we may assume that $n=r$ and $q \geqs 7$, so $|R_0| = ((q-1)_r)^{r-1}$. As noted above, $\mathcal{M}(R)$ contains a $\C_2$-subgroup of type ${\rm GL}_{1}(q) \wr S_n$. There are now several cases to consider.

First assume $q=q_0^2$ is a square. Here $(q-1)_r = (q_0-\e)_r$ for some $\e=\pm$ and it is easy to check that the corresponding $\C_5$ or $\C_8$ subgroup of type ${\rm GL}_{n}^{\e}(q_0)$ has $r'$-index (both subgroups are always maximal). Therefore, for the remainder we may assume $f$ is odd. We will handle the cases $n \in \{3,5\}$ separately.

Suppose $n=r=3$. If $q=p \equiv 4,7\imod{9}$ then $\mathcal{M}(R)$ contains a $\C_6$-subgroup of type $3^{1+2}.{\rm Sp}_2(3)$. On the other hand, if $q=p \equiv 1\imod{9}$ then $|R_0| \geqs 3^4$ and by inspecting \cite[Tables 8.3, 8.4]{BHR} we deduce that $|\mathcal{M}(R)|=1$. Now assume $q=q_0^k$ for some odd prime $k$ and note that $q_0 \equiv 1 \imod{3}$ since $q \equiv 1 \imod{3}$. If $k \geqs 5$ then $(q-1)_3=(q_0-1)_3$ and $\mathcal{M}(R)$ contains a subfield subgroup of type ${\rm GL}_{3}(q_0)$. However, if $f = 3^a$ is a $3$-power and $q=q_0^3$, then $|R_0| = 3^2((q_0-1)_3)^2$ and we deduce that $|\mathcal{M}(R)|=1$. 

Next assume $n=r=5$, so $|R_0|=((q-1)_5)^4$. By inspecting \cite[Tables 8.18, 8.19]{BHR}, we see that we only need to consider $\C_5$-subgroups $H$ of type ${\rm GL}_5(q_0)$, where $q=q_0^k$ and $k$ is an odd prime. Now $q_0 \equiv 1 \imod{5}$ and it is easy to check that $H$ has $5'$-index if and only if $k \ne 5$. Therefore, $|\mathcal{M}(R)|=1$ if and only if $f = 5^a$ for some $a \geqs 0$. That is to say, $|\mathcal{M}(R)|=1$ if and only if $\a(1,1)$ holds.

Now assume $n=r \geqs 7$ and suppose $\mathcal{M}(R)$ contains a subgroup $H$, in addition to the $\C_2$-subgroup of type ${\rm GL}_1(q) \wr S_n$. By inspecting  \cite{BHR,KL}, we deduce that $H \in \C_5 \cup \mathcal{S}$ and so there are two possibilities to consider.

First assume $H \in \C_5$ is a subfield subgroup of type ${\rm GL}_n(q_0)$, where $q=q_0^k$ for some odd prime $k$ (recall that $f$ is odd). Then $H$ has $r'$-index if and only if $(q-1)_r = (q_0-1)_r$, and by Lemma \ref{l:gk} this is equivalent to the conditions $k \ne r$ and $q_0 \equiv 1 \imod{r}$. In conclusion, there are no $\C_5$-subgroups in $\mathcal{M}(R)$ if and only if $\a(1,1)$ holds.

Finally, suppose $H \in \mathcal{S}$ and let $S$ denote the socle of $H$. Up to conjugacy, we observe that $R_0$ contains elements of the form ${\rm diag}(\lambda, \lambda^{-1}, I_{n-2})$ (modulo scalars), where $\l \in \mathbb{F}_{q}^{\times}$ has order $r$. Therefore, $H_0$ is an irreducible subgroup of $T$ containing an element $x$ with $\nu(x) = 2$ (see \eqref{e:nu}) and this allows us to apply \cite[Theorem 7.1]{GS} in order to determine the possibilities for $H$ (also see \cite[Theorem 2.12]{BGS} for a convenient version of this result). By inspecting \cite[Table 5]{BGS}, we see that $(T,S) = ({\rm L}_7(p), {\rm U}_3(3))$ is the only possibility with $n \geqs 7$ prime, where $p$ is a prime such that $p \equiv 1 \imod{3}$. Here $|H|_7 = 7$ and the condition $i=1$ implies that $H$ does not have $7'$-index. Therefore, there are no $\mathcal{S}$ collection subgroups in $\mathcal{M}(R)$ and our analysis of the case $i=1$ is complete.

\vs

\noindent \emph{Case 2. $1 < i < n$}

\vs

If $i$ does not divide $n$, then the maximal parabolic subgroups $P_1$ and $P_{n-1}$ have $r'$-index and thus $|\mathcal{M}(R)| \geqs 2$. Now assume $i$ divides $n$, say $n = \ell i$. Suppose $\ell$ is indivisible by $r$. Then the maximal parabolic subgroups $P_i$ and $P_{n-i}$ have $r'$-index, so we may assume $n$ is even and $i=n/2$. Here we can take $P_{n/2}$ and a $\C_2$-subgroup of type ${\rm GL}_{n/2}(q) \wr S_2$, noting that the $\C_2$-subgroup is maximal unless $n=4$ and $q \leqs 3$ (see \cite{BHR,KL}). In the latter case, we have $G = {\rm L}_4(2) \cong A_8$, $r=3$ and it is easy to check that $|\mathcal{M}(R)| \geqs 2$. 

So to complete the analysis of this case, we may assume $n = \ell i$ and $\ell = mr$ for some $m \geqs 1$. Here we observe that a $\C_2$-subgroup of type ${\rm GL}_i(q) \wr S_{\ell}$ has $r'$-index, and this is maximal unless $i = q=2$. In the latter case we have $G=T$, $r=3$ and one can check that a $\C_8$-subgroup of type ${\rm Sp}_n(q)$ is maximal with $r'$-index. Now, if $m \geqs 2$ then a $\C_2$-subgroup of type ${\rm GL}_{ir}(q) \wr S_m$ is maximal and contains a Sylow $r$-subgroup, so we have reduced to the case where $n = ri$. Let $k$ be the smallest prime divisor of $i$. Then a $\C_3$-subgroup of type ${\rm GL}_{n/k}(q^k)$ is maximal and has $r'$-index, which implies that $|\mathcal{M}(R)| \geqs 2$.

\vs

\noindent \emph{Case 3. $i = n$}

\vs

Here $|R_0| = (q^n-1)_r$ and $n$ divides $r-1$. In particular, $r$ does not divide $n$ and so our assumption that $G/T$ is an $r$-group implies that $G \leqs {\rm P\Sigma L}_{n}(q) = T.\la \phi \ra$. If $t$ is a prime divisor of $n$, then a $\C_3$-subgroup of type ${\rm GL}_{n/t}(q^t)$ is maximal and has $r'$-index, so we may assume $n = t^a$ is a prime power. If $t=2$ then $\mathcal{M}(R)$ also contains a $\C_8$-subgroup of type ${\rm Sp}_n(q)$, so we may assume $t$ is odd. Note that $r \geqs 2n+1$.

First assume $q=q_0^2$ is a square and fix $\e=\pm$ so that $(q^n-1)_r = (q_0^n-\e)_r$. If $\e=+$ then $\mathcal{M}(R)$ contains a $\C_5$-subgroup of type ${\rm GL}_{n}(q_0)$, and similarly a $\C_8$-subgroup of type ${\rm GU}_{n}(q_0)$ if $\e=-$. 

For the remainder, we may assume $f$ is odd. Since $r \geqs 2n+1$, it is easy to see that the order of a $\C_6$-subgroup of type $t^{1+2a}.{\rm Sp}_{2a}(t)$ is indivisible by $r$. Therefore, it remains to determine if there are any subgroups $H \in \mathcal{M}(R)$ belonging to the $\C_5$ or $\mathcal{S}$ collections.

Suppose $H \in \C_5$ is a subfield subgroup of type ${\rm GL}_{n}(q_0)$, where $q=q_0^k$ and $k$ is an odd prime. Then $H$ has $r'$-index if and only if $(q_0^n-1)_r = (q^n-1)_r$, which is equivalent to the conditions $q_0^n \equiv 1 \imod{r}$ and $k \ne r$. In other words, if $f$ is odd then there are no subfield subgroups in $\mathcal{M}(R)$ if and only if $\a(n,1)$ holds.

Finally, let us assume $H \in \mathcal{S}$ has socle $S$. Here our key tool is the main theorem of \cite{GPPS}, which we can use to determine the possibilities for $H$ that contain an element of order $r$. Indeed, by carefully inspecting \cite[Tables 2-8]{GPPS}, we deduce that $S = {\rm L}_{2}(r)$ and $n=(r-1)/2$ is the only possibility, where $r \geqs 7$ is a prime with $r \equiv 3 \imod{4}$ (see \cite[Table 8]{GPPS} and \cite[Table 2(c)]{HM}). Note that $r \in \{7,11,19,23,47,59, \ldots\}$ since $n = t^a$ is an odd prime power. By considering the irrationalities of the corresponding Brauer character, we see that $q=p$ (hence $G = T$) and we require $\mathbb{F}_{p}$ to contain the ${\rm b}_r$ irrationality in the notation of \cite[Section 4.2]{BHR}. Since $r \equiv 3 \imod{4}$ we have 
${\rm b}_r = \frac{1}{2}(-1+\sqrt{-r})$, whence ${\rm b}_r \in \mathbb{F}_p$ if and only if $-r = \square \imod{p}$ (that is, $-r$ is a square modulo $p$). Clearly, $H$ has $r'$-index if and only if $|R| = r$. 

So let us assume $H$ has $r'$-index. We claim that $H$ is a maximal subgroup of $G$. Suppose otherwise. Then by our above analysis, we must have $H<K<G$, where $K$ is a $\C_3$-subgroup of type ${\rm GL}_{n/t}(q^t)$. But this implies that some quasisimple cover $L < {\rm GL}_n(q)$ of $S$ is contained in ${\rm GL}_{n/t}(q^t)$, which is not possible because $L$ acts absolutely irreducibly on the natural module (as a condition of the membership of $H$ in $\mathcal{S}$), whereas ${\rm GL}_{n/t}(q^t)$ does not. Alternatively, we can rule out the containment of $H$ in $K$ by observing that any nontrivial representation of $S = {\rm L}_2(r)$ in non-defining characteristic has dimension at least $n = (r-1)/2$. 

We conclude that $H$ is maximal as claimed and so in this situation we deduce that $|\mathcal{M}(R)| \geqs 2$.
\end{proof}

\subsection{Unitary groups}

\begin{prop}\label{p:psu3}
The conclusion to Theorem \ref{t:main1} holds when $T = {\rm U}_3(q)$.
\end{prop}

\begin{proof}
The groups with $q \in \{3,4,5\}$ can be handled using {\sc Magma}, so we will assume $q \geqs 7$. If $r=p$ then $H = N_G(R_0) = P_1$ is a maximal parabolic subgroup and $\mathcal{M}(R) = \{H\}$. For the remainder, we may assume $r \ne p$. As before, set $i = d_q(r)$ and note that $i \in \{1,2,6\}$. 

If $i=1$ then $|R_0| = (q-1)_r$ and we observe that the $\C_1$-subgroups of type $P_1$ and $N_1$ both have $r'$-index and thus $|\mathcal{M}(R)| \geqs 2$ (here $N_1$ denotes the stabiliser in $G$ of a nondegenerate $1$-space).

Next assume $i=2$, so $|R_0| = ((q+1)_r)^2$. Since $q \geqs 7$, we first observe that $\mathcal{M}(R)$ contains a $\C_2$-subgroup of type ${\rm GU}_1(q) \wr S_3$. If $r \geqs 5$ then $N_1$ also has $r'$-index, so we may assume $r=3$. If $f=1$ then by inspecting \cite[Tables 8.5, 8.6]{BHR} we see that either $|R_0| = 3^2$ and $\mathcal{M}(R)$ contains a $\C_6$-subgroup of type $3^{1+2}.{\rm Sp}_{2}(3)$, or $|R_0| >  3^2$ and $|\mathcal{M}(R)| = 1$. Now assume $f>1$ and note that $f$ is odd since $q \equiv -1 \imod{3}$. Then $|\mathcal{M}(R)| \geqs 2$ if and only if $\mathcal{M}(R)$ contains a subfield subgroup $H$ of type ${\rm GU}_3(q_0)$ with $q=q_0^k$ for some prime divisor $k$ of $f$. Now $q_0 \equiv -1 \imod{3}$ and we calculate that $H$ has $3'$-index if and only if $k \ne 3$. In conclusion, if $f>1$ then $|\mathcal{M}(R)| = 1$ if and only if $f = 3^a$ is a $3$-power (for example, $|\mathcal{M}(R)| = 1$ if $G = {\rm U}_3(8)$ and $r=3$).

To complete the proof, let us assume $i=6$, so $r$ divides $q^3+1$. Note that $|R_0| = (q^3+1)_r$ and $r \geqs 7$. Since $q \geqs 7$, we see that $\mathcal{M}(R)$ contains a $\C_3$-subgroup of type ${\rm GU}_1(q^3)$. If $f=1$ then either $|\mathcal{M}(R)| = 1$, or $|R| = r=7$, $q=p \equiv 3,5 \imod{7}$ and $\mathcal{M}(R)$ contains an $\mathcal{S}$ collection subgroup with socle ${\rm L}_2(7)$. Note that if $r=7$ and $i=6$ then $q \equiv 3,5 \imod{7}$ (for example, if $q \equiv 2,4 \imod{7}$, then $7$ divides $q^3-1$). Therefore, $|\mathcal{M}(R)| \geqs 2$ if $f=1$ and $|R| = r=7$. 

Finally, let us assume $f>1$. Here we need to determine whether or not $G$ has a subfield subgroup with $r'$-index, and in the usual manner we deduce that $|\mathcal{M}(R)|=1$ if and only if $\b(3,-1)$ holds (see \eqref{e:ab}).
\end{proof}

\begin{prop}\label{p:psun}
The conclusion to Theorem \ref{t:main1} holds when $T = {\rm U}_n(q)$ and $n \geqs 4$.
\end{prop}

\begin{proof}
If $r=p$ then 
every parabolic subgroup of $G$ has $r'$-index and hence $|\mathcal{M}(R)| \geqs 2$. For the remainder, we may assume $r \ne p$. Let $i = d_q(r)$ and note that $i$ divides $r-1$. The groups with socle
\[
T \in \{ {\rm U}_4(2), {\rm U}_4(3), {\rm U}_5(2), {\rm U}_6(2) \}
\]
can be handled using {\sc Magma} and we find that $|\mathcal{M}(R)|=1$ if and only if $G = {\rm U}_5(2)$ and $r = r_{10} = 11$, in which case $\mathcal{M}(R) = \{{\rm L}_{2}(11)\}$. We will exclude these cases for the remainder of the proof.

\vs

\noindent \emph{Case 1. $i=2$}

\vs

Here $r$ divides $q+1$ and in view of Lemma \ref{l:gk} we observe that
\[
(q^d - (-1)^d)_r = (q+1)_r(d)_r
\]
for every positive integer $d$. As a consequence, if $n$ is indivisible by $r$ then 
\[
|R_0| = ((q+1)_r)^{n-1}(n!)_r
\]
and we deduce that $\mathcal{M}(R)$ contains subgroups of type $N_1$ and ${\rm GU}_1(q) \wr S_n$ (by inspecting \cite{BHR,KL}, recalling that $T \ne {\rm U}_4(3), {\rm U}_6(2)$, we see that both subgroups are maximal in $G$). Therefore, we may assume $n = mr$. If $m \geqs 2$, then $\mathcal{M}(R)$ contains $\C_2$-subgroups of type ${\rm GU}_1(q) \wr S_n$ and ${\rm GU}_r(q) \wr S_{m}$, so we have reduced to the case where $n=r$. Here $\mathcal{M}(R)$ contains a subgroup of type ${\rm GU}_1(q) \wr S_n$ and we need to determine if $G$ has any additional maximal subgroups with $r'$-index.

Suppose $H \in \mathcal{M}(R)$ and $H$ is not of type ${\rm GU}_1(q) \wr S_n$. Then by carefully considering the geometric maximal subgroups of $G$ we deduce that $H \in \C_5 \cup \mathcal{S}$, so there are two possibilities to consider.

Suppose $H$ is a subfield subgroup of type ${\rm GU}_n(q_0)$, where $q=q_0^k$ and $k$ is an odd prime divisor of $f$. Then $H$ has $r'$-index if and only if $k \ne r$ and $q_0 \equiv -1 \imod{r}$. It follows that there are no subfield subgroups in $\mathcal{M}(R)$ if and only if $\b(1,-1)$ holds.

Finally, let us assume $H \in \mathcal{S}$. The case $n=5$ can be eliminated by inspecting \cite[Table 8.21]{BHR} (here we find that the socle of $H$ is either ${\rm L}_2(11)$ or ${\rm U}_4(2)$). For $n \geqs 7$ we can appeal to \cite[Theorem 7.1]{GS}, noting that $T$ contains an element $x$ with $\nu(x) = 2$. By inspecting \cite[Table 5]{BGS}, it is straightforward to rule out the existence of any $\mathcal{S}$ collection subgroups in $\mathcal{M}(R)$. 

\vs

\noindent \emph{Case 2. $i=1$}

\vs

Next suppose $r$ divides $q-1$, so $q \geqs 4$ and $(n,q+1)$ is indivisible by $r$. Note that $(q^d+1)_r = 1$ if $d$ is odd (see Lemma \ref{l:gk}). If $n$ is odd, then $\mathcal{M}(R)$ contains $\C_1$-subgroups of type $P_{(n-1)/2}$ and $N_1$, so $|\mathcal{M}(R)| \geqs 2$. Similarly, if $n$ is even then $\mathcal{M}(R)$ contains a $\C_2$-subgroup of type ${\rm GL}_{n/2}(q^2)$ and a $\C_8$-subgroup of type ${\rm Sp}_n(q)$. 

\vs

\noindent \emph{Case 3. $i=2n$}

\vs

Here $n$ is odd and $r$ divides $q^n+1$. In particular, $2n$ divides $r-1$ and we note that $|R_0| = (q^n+1)_r$. If $t$ is a prime divisor of $n$, then $\mathcal{M}(R)$ contains a $\C_3$-subgroup of type ${\rm GU}_{n/t}(q^t)$, so we may assume $n = t^a$ is a prime power. It is easy to see that $r$ does not divide the order of a $\C_6$-subgroup of type $t^{1+2a}.{\rm Sp}_{2a}(t)$ and so it remains to determine whether or not $\mathcal{M}(R)$ contains a subgroup $H \in \C_5 \cup \mathcal{S}$. 

In the usual manner, we deduce that there are no subfield subgroups in $\mathcal{M}(R)$ if and only if $\b(n,-1)$ holds, 
so let us turn to the subgroups in $\mathcal{S}$. As in the proof of Proposition \ref{p:psln} (see Case 3), we can use \cite{GPPS} to determine the subgroups $H \in \mathcal{S}$ containing an element of order $r$. Here the only possibility arises when $n=(r-1)/2$ and $H$ has socle $S = {\rm L}_2(r)$, where $r \geqs 7$ is a prime and $r \equiv 3 \imod{4}$. By considering the irrationalities of the corresponding Brauer character, we see that $q=p$ (so $G=T$) and $-r = \boxtimes \imod{p}$ (note that if $-r$ is a square modulo $p$, then $S$ embeds in ${\rm L}_n(q)$, rather than ${\rm U}_n(q)$). Clearly, $H$ has $r'$-index if and only if $|R|=r$. Moreover, by arguing as in the final paragraph of the proof of Proposition \ref{p:psln}, we deduce that such a subgroup $H$ is maximal (this relies on the observation that the extension field subgroup ${\rm GU}_{n/t}(q^t)<{\rm GU}_n(q)$ is not absolutely irreducible on the natural module, or we can use the fact that the dimension of any nontrivial representation of ${\rm L}_2(r)$ in non-defining characteristic is at least $n = (r-1)/2$).

\vs

\noindent \emph{Case 4. $i = 2n-2$}

\vs

Here $n$ is even, $r$ divides $q^{n-1}+1$ and we have $|R_0| = (q^{n-1}+1)_r$. Clearly, $\mathcal{M}(R)$ contains a $\C_1$-subgroup of type $N_1$, and we see that any additional subgroup $H \in \mathcal{M}(R)$ must be in one of the collections $\C_5$ or $\mathcal{S}$. There are no subfield subgroups in $\mathcal{M}(R)$ if and only if $\b(n-1,-1)$ holds.

Finally, suppose $H \in \mathcal{S}$ has $r'$-index. By inspecting \cite{GPPS}, we deduce that $n = (r+1)/2$ and $H$ has socle $S = {\rm L}_2(r)$, where $f=1$,   $r \equiv 3 \imod{4}$ and $-r = \boxtimes \imod{p}$. If all of these conditions are satisfied, then such a subgroup $H$ is maximal because it cannot be contained in a reducible $\C_1$-subgroup of type $N_1$.

\vs

\noindent \emph{Case 5. $2 < i < 2n-2$}

\vs

If $i>n$ then $i \equiv 2 \imod{4}$, $|R_0| = (q^{i/2}+1)_r$ and it is easy to see that the $\C_1$-subgroups $P_1$ and $N_1$ have $r'$-index. Now assume $2 < i \leqs n$, noting that $i \leqs n/2$ if $i$ is odd. If $n$ is even, then $P_{n/2}$ and a $\C_2$-subgroup of type ${\rm GL}_{n/2}(q^2)$ have $r'$-index. Similarly, if $n$ is odd then $P_{(n-1)/2}$ and $N_1$ have $r'$-index.
\end{proof}

\subsection{Symplectic groups}

\begin{prop}\label{p:psp}
The conclusion to Theorem \ref{t:main1} holds when $T = {\rm PSp}_n(q)'$ and $n \geqs 4$.
\end{prop}

\begin{proof}
The groups with $n=4$ and $q \in \{2,3\}$ can be handled using {\sc Magma} and we find that $|\mathcal{M}(R)| = 1$ if and only if $G = {\rm PSp}_4(2)' \cong A_6$ and $r = r_2 = 3$ (here $|R|=3^2$ and $\mathcal{M}(R) = \{H\}$ with $H$ of type ${\rm Sp}_2(2) \wr S_2$). For the remainder, we will assume $(n,q) \ne (4,2), (4,3)$. If $r=p$ then the maximal parabolic subgroups $P_1$ and $P_2$ have $r'$-index. Now assume $r \ne p$ and set $i = d_q(r)$.

\vs

\noindent \emph{Case 1. $i \in \{1,2\}$}

\vs

Here $q \geqs 4$ and a maximal parabolic subgroup $P_{n/2}$ has $r'$-index. In addition, $\mathcal{M}(R)$ contains a subgroup of type ${\rm GL}_{n/2}(q)$ if $q$ is odd, and one of type ${\rm O}_{n}^{+}(q)$ if $q$ is even. Therefore, $|\mathcal{M}(R)| \geqs 2$ when $i=1$.

Now assume $i=2$, so $r$ divides $q+1$. If $n=4$ then $q \geqs 4$, $|R_0| = ((q+1)_r)^2$ and $\mathcal{M}(R)$ contains a $\C_2$-subgroup of type ${\rm Sp}_2(q) \wr S_2$, together with a $\C_3$-subgroup of type ${\rm GU}_2(q)$ if $q$ is odd and a $\C_8$-subgroup of type ${\rm O}_4^{+}(q)$ if $q$ is even. Now assume $n \geqs 6$. If $q$ is odd, then $\mathcal{M}(R)$ contains a $\C_2$-subgroup of type ${\rm Sp}_2(q) \wr S_{n/2}$ and a $\C_3$-subgroup of type ${\rm GU}_{n/2}(q)$. For $q$ even we find that $\mathcal{M}(R)$ contains a $\C_8$-subgroup of type ${\rm O}_n^{\e}(q)$, where $\e = (-)^{n/2}$, and it also contains a $\C_2$-subgroup of type ${\rm Sp}_2(q) \wr S_{n/2}$ if $q \geqs 4$. Finally, suppose $q=2$ and $r=3$. If $n=6$ then one can check that $|\mathcal{M}(R)|=1$, so we may assume $n \geqs 8$. If $n$ is indivisible by $3$, then a $\C_1$-subgroup $N_2 = {\rm Sp}_{n-2}(2) \times {\rm Sp}_2(2)$ has $3'$-index. Similarly, if $n \geqs 12$ is divisible by $3$ then a $\C_2$-subgroup ${\rm Sp}_6(2) \wr S_{n/6}$ has $3'$-index.

\vs

\noindent \emph{Case 2. $i=n$}

\vs

Here $r \equiv 1 \imod{n}$ and $|R_0| = (q^n-1)_r$. First assume $n \equiv 2 \imod{4}$ and let $t$ be an odd prime divisor of $n$. Then a $\C_3$-subgroup of type ${\rm Sp}_{n/t}(q^t)$ has $r'$-index. In addition, $\mathcal{M}(R)$ contains a $\C_3$-subgroup of type ${\rm GU}_{n/2}(q)$ if $q$ is odd, and a $\C_8$-subgroup of type ${\rm O}_{n}^{-}(q)$ if $q$ is even.

Now assume $n \equiv 0 \imod{4}$ and observe that $\mathcal{M}(R)$ contains a $\C_3$-subgroup of type ${\rm Sp}_{n/2}(q^2)$. If $q$ is even, then $\mathcal{M}(R)$ also contains a $\C_8$-subgroup of type ${\rm O}_{n}^{-}(q)$, so we may assume $q$ is odd. As above, if $n$ is divisible by an odd prime $t$, then a $\C_3$-subgroup of type ${\rm Sp}_{n/t}(q^t)$ has $r'$-index, so we have reduced to the case where $n = 2^a$ is a $2$-power and $q$ is odd. To complete the analysis of this case, we need to consider the existence of a subgroup $H \in \C_5 \cup \C_6 \cup \mathcal{S}$ with $r'$-index.

Suppose $H \in \C_5$ is a subfield subgroup of type ${\rm Sp}_n(q_0)$, where $q=q_0^k$ for some prime $k$. Then $|H_0|_r = (q_0^n-1)_r$ and thus $H$ has $r'$-index if and only if $(q_0^n-1)_r = (q^n-1)_r$. We deduce that there are no subfield subgroups in $\mathcal{M}(R)$ if and only if $\a(n,1)$ holds.

Next suppose $H$ is a $\C_6$-subgroup of type $2^{1+2a}.{\rm O}_{2a}^{-}(2)$, so $f=1$ and $G=T$. If $r > n+1$, then $|H|$ is indivisible by $r$, so let us assume $r = n+1 = 2^a+1$, in which case $r$ is a Fermat prime. Here $|H|_r = r$ and thus $H$ has $r'$-index if and only if $|R| = r$. Let us also observe that if $|R| = r = n+1$ and $f \geqs 2$, then $r$ divides $p^n-1$ and thus every subfield subgroup of $G$ has $r'$-index. This shows that $|\mathcal{M}(R)| \geqs 2$ if $|R| = r=n+1$.

Finally, let us consider the subgroups in $\mathcal{S}$. First assume $n=4$, so $r,q \geqs 5$ and the relevant subgroups are recorded in \cite[Table 8.13]{BHR} (recall that $q$ is odd). The groups with $q \in \{5,7\}$ can be handled using {\sc Magma} (we find that $|\mathcal{M}(R)| = 1$ if and only if $(q,r) = (5,13)$ or $(7,5)$), so let us assume $q \geqs 9$. Then by inspecting \cite[Table 8.13]{BHR} we see that $q=p \equiv \pm 1, \pm 5 \imod{12}$ and there is a subgroup $H \in \mathcal{S}$ with socle $A_6$. Therefore, we must have $G=T$, $r=5$ and we deduce that $H$ has $r'$-index if and only if $|R| = 5$. As noted above, $\mathcal{M}(R)$ also contains a $\C_6$-subgroup in this situation.

For the remainder we may assume $n = 2^a \geqs 8$. Here \cite{GPPS} describes the subgroups $H \in \mathcal{S}$ that contain an element of order $r$; by carefully studying \cite[Tables 2-8]{GPPS} we deduce that $n \in \{s,s-1,(s-1)/2\}$ and $H$ has socle $S = {\rm L}_2(s)$ with $s \geqs 7$. Now $H$ has $r'$-index if and only if $|R_0| = r$, so we may as well assume $|R_0| = r$. Therefore, in view of our above observations, we may assume that $r \geqs 2n+1$, which in turn implies that $n = (s-1)/2$ and $r=s = 2n+1$. By inspecting \cite[Table 2(b)]{HM} we see that this situation does arise, subject to the irrationality $\sqrt{r}$ of the corresponding Brauer character. This means that either $q=p$ and $r = \square \imod{p}$, or $q=p^2$ and $r = \boxtimes \imod{p}$, so in both cases we have $G=T$. Note that such a subgroup $H$ is maximal since it is not contained in a $\C_3$-subgroup of type ${\rm Sp}_{n/2}(q^2)$ (for example, because ${\rm Sp}_{n/2}(q^2) < {\rm Sp}_{n}(q)$ is not absolutely irreducible on the natural module) nor a subfield subgroup if $q=p^2$ (by definition, a subgroup in $\mathcal{S}$ cannot be realised over a proper subfield of $\mathbb{F}_q$). 

In conclusion, if $n = 2^a \geqs 8$, $q$ is odd, $r = 2n+1$ and $\a(n,1)$ holds, then $|\mathcal{M}(R)|=1$ if and only if $|R|>r$, or $f>2$, or $f=1$ and $r = \boxtimes \imod{p}$, or $f=2$ and $r = \square \imod{p}$.

\vs

\noindent \emph{Case 3. $2<i<n$}

\vs

Finally, let us assume $2<i<n$ and note that $n \geqs 6$. If $n$ is indivisible by $i$, then $\mathcal{M}(R)$ contains $\C_1$-subgroups of type $P_1$ and $N_2$. Now assume $n = \ell i$ with $\ell \geqs 2$. If $i$ is odd, then $\ell$ is even and it is easy to see that $P_{n/2}$ has $r'$-index, together with a $\C_3$-subgroup of type ${\rm Sp}_{n/t}(q^t)$, where $t$ is any prime divisor of $i$. Similarly, if $i \equiv 0 \imod{4}$ then $\mathcal{M}(R)$ contains subgroups of type ${\rm Sp}_{i}(q) \wr S_{n/i}$ and ${\rm Sp}_{n/2}(q^2)$. Finally, if $i \equiv 2 \imod{4}$ then a subgroup of type ${\rm Sp}_{i}(q) \wr S_{n/i}$ has $r'$-index, as does one of type ${\rm GU}_{n/2}(q)$ (if $q$ is odd) and ${\rm O}_n^{\e}(q)$ (if $q$ is even), where $\e = (-)^{n/i}$.
\end{proof}

\subsection{Odd dimensional orthogonal groups}

In this section we prove Theorem \ref{t:main1} when $T = \O_n(q)$ and $n \geqs 7$ is odd. In order to handle this case, we require a couple of preliminary lemmas.

\begin{lem}\label{l:new1}   
Let $r \geqs 5$ be a prime and suppose $x \in {\rm GL}_n(q)={\rm GL}(V)$ acts irreducibly on a $d$-dimensional subspace of $V$, where $|x| = r^a$ and $a \geqs 2$. Then $|x| \geqs d + r^{a-1} \geqs d + 5$.
\end{lem}

\begin{proof}  
Since $x$ acts irreducibly on a $d$-dimensional space, we see that $q$ has order $d$ modulo $r^a$ and thus $d$ divides $|(\mathbb{Z}/r^a \mathbb{Z})^*|=r^a - r^{a-1}$. The result follows. 
\end{proof}

Recall that $F^*(L) = F(L)E(L)$ is the generalised Fitting subgroup of the finite group $L$, where $F(L)$ and $E(L)$ denote the Fitting subgroup and layer of $L$, respectively. Also recall that a perfect group $Q$ is quasisimple if $Q/Z(Q)$ is simple.

\begin{lem}\label{l:new2} 
Let $L$ be a finite group such that $Q=F^*(L)$ is quasisimple, and let $V$ be a nontrivial faithful irreducible $\mathbb{F}_qL$-module. Let $r \geqs 5$ be a prime and let $a \geqs 1$ be an integer.    Assume that $L$ contains an element $x$ of order $r^a$. 
\begin{itemize}\addtolength{\itemsep}{0.2\baselineskip}
\item[{\rm (i)}] If $Q/Z(Q) = A_n$ and $n \geqs \max\{9,r^a\}$, then $\dim V \geqs r^a-2$.
\item[{\rm (ii)}] If $Q/Z(Q)$ is a sporadic group and $a \geqs 2$, then $\dim V \geqs r^a+1$.
\item[{\rm (iii)}] Suppose $Q/Z(Q)$ is a simple group of Lie type in characteristic $s$ with $(q,s) = 1$. Then either $L/Z(L) = {\rm L}_2(r)$ and $a=1$, or $\dim V  \geqs r^a -2$. 
\end{itemize}
\end{lem}

\begin{proof}  
Without loss of generality, we may assume that $L = \langle Q, x \rangle$. In addition, since we are seeking a lower bound on $\dim V$, we may work over the algebraic closure of $\mathbb{F}_q$. We may also assume that $Q$ acts homogeneously on $V$ (otherwise, we can pass to an irreducible $Q$-submodule and take the stabiliser of that module). Since $L/Q$ is cyclic, we can also assume that $Q$ acts irreducibly on $V$. Finally, since $r \geqs 5$, we may assume that $L=Q$ in cases (i) and (ii).  

For (i), it is well known that if $n \geqs 9$ then the dimension of any nontrivial irreducible $\mathbb{F}_qL$-module (in any characteristic) is at least $n-2$ (see \cite[Section 2]{T}, for example).  

Turning to (ii), we first observe that there are very few cases where $L$ contains an element of order $r^a$ with $r \geqs 5$ a prime and $a \geqs 2$. Indeed, by inspection we find that $r^a=25$ is the only possibility and
$L/Z(L) = {\rm HN}$, ${\rm Ly}$, $\mathbb{B}$ or $\mathbb{M}$. In each case, it is easy to check that $\dim V \geqs 26$. 

Now consider (iii). First assume $Q/Z(Q) = {\rm L}_2(s^e)$, where $e \geqs 1$ and $s$ is a prime with $(q,s) = 1$. Suppose that $s$ is odd, which means that $\dim V \geqs (s^e-1)/2$. If $r=s$ then every $r$-element in $L$ has order $r$ and we note that $(r^e-1)/2 \geqs r-2$ if and only if $e \geqs 2$. And if $e=1$ and $\dim V < r-2$, then $\dim V  = (r \pm 1)/2$ and $V$ is not a module for ${\rm PGL}_2(r)$, which leads to the special case $L/Z(L) = {\rm L}_2(r)$, $a=1$ in the statement of part (iii). On the other hand, if $r \ne s$ then every $r$-element in $L$ has order at most $(s^e+1)/2$ and the result follows. Finally, suppose $s=2$. Here $\dim V \geqs s^e-1$ and the maximal order of an $r$-element in $L$ is at most $s^e+1$. Therefore $r^a \leqs s^e+1$ and the result follows.

For the remainder, we may assume $Q/Z(Q) \ne {\rm L}_2(s^e)$. Here the desired bound $\dim V \geqs r^a-2$ essentially follows from \cite[Theorem 1.1]{DiMuro}, where a stronger bound is established (with several exceptions, which one has to check). Alternatively, one can combine upper bounds on the orders of elements of prime power order in $L$ (see \cite{GMPS}) with lower bounds from \cite{T} on the minimal dimension of faithful irreducible cross-characteristic representations of $L$. 
Note that we also have to consider central extensions of groups of the form ${\rm SL}_d(t)$ and ${\rm SU}_d(t)$, but in this situation we may assume the prime $r$ divides $t \pm 1$ and the corresponding bounds are much better.  

For example, suppose $Q/Z(Q)$ is a classical group defined over $\mathbb{F}_{t}$, where $t = s^e$ for a prime $s$. Let $d \geqs 3$ be the dimension of the natural module $W$ and suppose $x \in L$ has order $r^a$.  If $x$ is semisimple then it induces an inner-diagonal
automorphism on $Q/Z(Q)$ and it acts faithfully on a subspace of $W$, which is either totally singular or nondegenerate if $L$ preserves a form on $W$. This implies that $|x| \leqs (t^{d} -1)/(t-1)$. If $x$ is unipotent, then $|x| \leqs s^b$, where $b$ is minimal such that $s^b \geqs d$. And if $x$ is in the coset of a field automorphism, then we can obtain a much better bound on $|x|$ using Shintani descent. The desired result then follows by applying the relevant bound on $\dim V$ from  \cite{T}.  For example, if $Q={\rm SL}_d(t)$, then we immediately deduce that the minimal dimension of a nontrivial faithful irreducible representation of $L$ is at least $(t^d-1)/(t-1) - 2$.  

We can use a very similar argument to handle the case where $Q/Z(Q)$ is an exceptional group over $\mathbb{F}_t$, where $t = s^e$ as above. For example, if $x$ is semisimple then we can bound $|x|$ by considering the orders of maximal tori (or we can just use the trivial bound $(t+1)^m$, where $m$ is the rank of the group). And we have $|x| \leqs s^b$ if $x$ is unipotent, where $b$ is minimal such that $s^b \geqs h$ and $h$ is the Coxeter number of $Q$. As noted above, even better bounds hold when $x$ is in the coset of a field automorphism, and in all cases we can play these bounds off against the lower bounds on $\dim V$ in \cite{T}. We leave the reader to check the details.  
\end{proof}

We now use the previous lemmas to establish the following technical result, which will be needed in the proof of Proposition \ref{p:pso_odd} below. Recall that $R_0 = R \cap T$.

\begin{prop}\label{p:new}
Let $G$ be an almost simple group with socle $T = \O_{n}(q)$, where $n \geqs 7$ is odd and $q$ is odd. Let $R$ be a Sylow $r$-subgroup of $G$, where $r$ is a primitive prime divisor of $q^i-1$ with $2 < i < n-1$. Assume $|R_0|>r$ and $R_0$ acts irreducibly on a nondegenerate $(n-1)$-dimensional subspace of the natural module for $T$. Then $R$ is not contained in a maximal subgroup in the collection $\mathcal{S}$. 
\end{prop}

\begin{proof} 
Let $V$ be the natural module for $T$ and let $W$ be a nondegenerate hyperplane of type $\e = \pm$ on which $R_0$ acts irreducibly. Write $q=p^f$, where $p$ is a prime, and note that $r \equiv 1 \imod{i}$. In particular, $r \geqs 5$ since $i>2$. Also note that any irreducible subgroup of $T$ containing $R_0$ must be absolutely irreducible, since the fixed space of $R_0$ on $V$ is $1$-dimensional. 

Fix an element $y \in Z(R_0)$ of order $r$ and let $U$ be a minimal $\la y \ra$-invariant subspace of $W$ on which $y$ acts nontrivially. Note that $\dim U = i$. Since $R_0$ acts irreducibly on $W$, it follows that any two irreducible $\la y\ra$-submodules of $W$ are isomorphic and thus $i$ is even (since the action of $y$ on $W$, and hence $U$, is self-dual). In particular, this implies that $C_V(y) = W^{\perp}$ is $1$-dimensional.

Note that we can decompose $V = V_0 \perp V_1 \perp \ldots \perp V_s$ as a direct sum of pairwise orthogonal nondegenerate spaces, where $\dim V_0<i$ and each $V_j$ for $j \geqs 1$ is a minus-type space. For each $j \in \{1, \ldots, s\}$, there exists an element $x_j \in T$ of order $r$ acting irreducibly on $V_j$ and trivially on $V_j^{\perp}$. Then $R_1 = \langle x_1, \ldots, x_s \rangle$ is an elementary abelian $r$-subgroup, which we may assume is contained in $R_0$.
Since $y$ centralises $R_1$, it follows that each $V_j$ with $j \geqs 0$ is $y$-invariant and thus $y$ is trivial on $V_0$. Therefore, $V_0 = W^{\perp}$ and 
$W = V_1 \perp \ldots \perp V_s$.  

Now $R_0$ acts irreducibly on $W$, which we may view as a space over $\FF_{q^i}$. In this setting, $R_0$ contains an element acting on $W$ as a pseudoreflection, which implies that $R_0$ acts absolutely irreducibly on $W$. Therefore, $\dim_{\FF_{q^i}} W = r^a$ with $a \geqs 1$ and thus $n-1 = i r^a$. It follows that $s = r^a$ and $W$ is a minus-type space. Since $R_0$ acts irreducibly on $W$, there is an element $x \in R_0$ cyclically permuting the summands $\{V_1, \ldots, V_s\}$, with $x^{r^a}$ acting irreducibly on each one. Therefore, $x$ has order $r^b$ with $b>a$, and $x$ acts irreducibly on $W$. We also note that $R_0$ is nonabelian. Indeed, we have $R_0 \cong C_{r^c} \wr C_r$ for some $c \geqs 1$.

We can now complete the proof. Seeking a contradiction, suppose $\mathcal{M}(R)$ contains a subgroup $H \in \mathcal{S}$. Since $T = \O_n(q) < {\rm SO}(V)$ acts on $V$ with trivial centre, we see that $H_0= H \cap T$ is an almost simple group acting absolutely irreducibly, primitively and tensor indecomposably on $V$. In particular, we may assume that $G=T$.  Let $S$ denote the socle of $H$.

Since the above element $x \in R_0$ acts irreducibly on $W$, Lemma \ref{l:new1} implies that 
\[
|x| = r^b \geqs \dim W + 5 = n+4
\]
and thus $n \leqs r^b-4$. However, if $S$ is a sporadic group, or a simple group of Lie type in cross-characteristic, then Lemma \ref{l:new2} implies that $n \geqs r^b-2$ and we have reached a contradiction. Similarly, if $S = A_m$ is an alternating group, then $m \geqs r^b \geqs 25$ and once again this possibility is ruled out by Lemma \ref{l:new2}.  

To complete the proof, we may assume $H$ is a group of Lie type over $\mathbb{F}_{q'}$, where $q' = p^e$ for some $e \geqs 1$. Since $R_0$ is nonabelian, the Weyl group of $S$ must contain an element of order $r$.

Recall that $n = i r^a + 1$ and let $z \in H$ be an element of order $r$ such that $\nu(z) = i$. If $a \geqs 2$ then $n \geqs 101$ and
\[
\nu(z) = i < \frac{1}{2}\sqrt{4(i+1)^2+1} \leqs \frac{1}{2}\sqrt{ir^2+1} \leqs \frac{1}{2}\sqrt{n},
\]
which means that the possibilities for $H$ are recorded in \cite[Theorem 7.1]{GS}. It is easy to see that no examples arise with $S$ a group of Lie type in the defining characteristic, so we may assume that $a=1$ and $n = ir+1 \geqs 21$. Note that $i \leqs r-1$ and thus $n \leqs (r-1)r+1 < r^2$.

If $S$ is an exceptional group of Lie type, then the fact that the Weyl group of $S$  contains an element of order $r \geqs 5$ implies that $S = E_6(q')$, $E_7(q')$ or $E_8(q')$, with $r =5$ in the first two cases and $r \in \{5,7\}$ in the latter. Now $n \geqs 27$ if $S = E_6(q')$ or $E_7(q')$, and we have $n \geqs 248$ if $S = E_8(q')$. So in every case we see that $n > r^2$, which is a contradiction.

Finally, suppose that $S$ is a classical group over $\mathbb{F}_{q'}$. First assume $q' > q$. Here \cite[Proposition 5.4.6]{KL} implies that $n = \ell^t$, where $\ell \geqs 3$, $t \geqs 2$ and $S$ is contained in an orthogonal group $J$ of dimension $\ell$.    
As noted above,  we see that $r$ divides the order of the Weyl group of $J$, so $\ell \geqs 2r$ and thus $n  = \ell^t \geqs 4r^2$, which is a contradiction.  Finally, let us assume $q' \leqs q$.  The irreducibility of $x$ on $W$ implies that $r^b$ is a primitive prime power divisor of $q^{n-1}-1$ and this means that the rank of $S$ is at least $(n-1)/2 \geqs 10$. But by inspecting \cite{Lu}, we conclude that $G$ does not have such a subgroup $H \in \mathcal{S}$, so this final possibility is also eliminated.
\end{proof}

\begin{prop}\label{p:pso_odd}
The conclusion to Theorem \ref{t:main1} holds when $T = \O_n(q)$ with $nq$ odd and $n \geqs 7$. 
\end{prop}

\begin{proof}
If $r=p$ then the maximal parabolic subgroups $P_1$ and $P_2$ have $r'$-index. Now assume $r \ne p$ and set $i = d_q(r)$. Note that $i \leqs n-1$ since we may assume $r$ divides $|T|$ (see Lemma \ref{l:trivial}). There are several cases to consider. 

\vs

\noindent \emph{Case 1. $i \in \{1,2\}$}

\vs

First assume $i=1$, so $q \geqs 7$. Here it is easy to check that the $\mathcal{C}_1$-subgroups of type $P_{(n-1)/2}$ and ${\rm O}_{n-1}^{+}(q)$ have $r'$-index and thus $|\mathcal{M}(R)| \geqs 2$.

Now suppose $i=2$, so $q \geqs 5$. If $n \equiv 1 \imod{4}$ then $(q^{(n-1)/2}-1)_r = (q^{n-1}-1)_r$ and thus a $\C_1$-subgroup of type ${\rm O}_{n-1}^{+}(q)$ has $r'$-index. Similarly, if $n \equiv 3 \imod{4}$ then $\C_1$-subgroups of type ${\rm O}_{n-1}^{-}(q)$ have $r'$-index. We need to determine whether or not $G$ has any additional maximal subgroups of $r'$-index.

If $n-1$ is indivisible by $r$, then $(q^{n-1}-1)_r = (q+1)_r$ and we see that $\mathcal{M}(R)$ contains $\C_1$-subgroups of type ${\rm O}_2^{-}(q) \times {\rm O}_{n-2}(q)$. Therefore, we may assume $r$ divides $n-1$. If $n \equiv 1 \imod{4}$ then subgroups of type ${\rm O}_{(n-1)/2}^{-}(q) \times {\rm O}_{(n+1)/2}(q)$ have $r'$-index, so we can assume $n \equiv 3 \imod{4}$. If $n \ne 2r+1$ then subgroups of type ${\rm O}_{2r}^{-}(q) \times {\rm O}_{n-2r}(q)$ have $r'$-index, so we have reduced the problem to $n = 2r+1$, in which case $|R_0| = r((q+1)_r)^r$. In the usual manner, we can show that there are no subfield subgroups in $\mathcal{M}(R)$ if and only if $\a(1,-1)$ holds. It remains to determine whether or not there are any subgroups in $\mathcal{S}$ with $r'$-index.

Since $i=2$, it follows that $T$ contains a semisimple element $x$ with $\nu(x)=2$, so we may use \cite[Theorem 7.1]{GS} to study the maximal overgroups of $R$ in $\mathcal{S}$. Suppose $H \in \mathcal{S}$ has $r'$-index and let $S$ be the socle of $H$. First assume $S = A_d$ is irreducibly embedded in $T$ via the fully deleted permutation module, so $f=1$ and either $d = n+1$ and $(d,p) = 1$, or $d = n+2$ and $p$ divides $d$. Here it is easy to see that $H$ does not have $r'$-index. The remaining possibilities for $H$ are recorded in \cite[Table 5]{BGS} and by inspection we see that $H$ has $r'$-index if and only if $G = \O_7(p)$, $H = {\rm Sp}_6(2)$, $r=3$ and $p \not\equiv -1 \imod{9}$. 

\vs

\noindent \emph{Case 2. $2<i<n-1$}

\vs

If $n-1$ is indivisible by $i$, then maximal subgroups of type $P_1$ and ${\rm O}_{n-1}^{+}(q)$ have $r'$-index. Now assume $n-1 = \ell i$ for some $\ell \geqs 2$. If $i$ is odd then we can take subgroups of type $P_{(n-1)/2}$ and ${\rm O}_{n-1}^{+}(q)$, so we may assume $i$ is even. If $\ell$ is even, then subgroups of type ${\rm O}_{n-1}^{+}(q)$ have $r'$-index, and the same conclusion holds for subgroups of type ${\rm O}_{n-1}^{-}(q)$ when $\ell$ is odd. In addition, if $r$ does not divide $\ell$ then subgroups of type ${\rm O}_i^{-}(q) \times {\rm O}_{n-i}(q)$ have $r'$-index. 

So we have now reduced to the case where $4 \leqs i \leqs n-3$ is even and $n-1 = mri$ for some $m \geqs 1$. Here $r \geqs 5$ and 
\[
|R_0| = (q^i-1)_r(q^{2i}-1)_r \cdots (q^{mri}-1)_r > r.
\]
Suppose $m$ is not an $r$-power, so we may write $n-1 = r^abi$, where $a \geqs 1$ and $b \geqs 2$ is indivisible by $r$. Here one can check that a $\C_1$-subgroup of type ${\rm O}_{r^ai}^{-}(q) \times {\rm O}_{n-r^ai}(q)$ has $r'$-index  and thus $|\mathcal{M}(R)| \geqs 2$. Therefore, we may assume $m$ is an $r$-power. In particular, $\ell$ is odd and thus $\mathcal{M}(R)$ contains a $\C_1$-subgroup $H$ of type ${\rm O}_{n-1}^{-}(q)$. 

As usual, there are no subfield subgroups in $\mathcal{M}(R)$ if and only if $\a(i,1)$ holds. Moreover, if the latter condition holds, then by inspecting \cite{BHR,KL} it is straightforward to check that $H$ is the unique geometric subgroup in $\mathcal{M}(R)$. Here $H$ is the stabiliser in $G$ of a nondegenerate $(n-1)$-dimensional subspace $W$ of $V$ and we note that $R_0$ acts irreducibly on $W$. Finally, by applying Proposition \ref{p:new} we conclude that there are no $\mathcal{S}$ collection subgroups in $\mathcal{M}(R)$ and hence $|\mathcal{M}(R)| = 1$. 

\vs

\noindent \emph{Case 3. $i=n-1$}

\vs

Here $|R_0| = (q^{n-1}-1)_r = (q^{(n-1)/2}+1)_r$ and we see that $\C_1$-subgroups of type ${\rm O}_{n-1}^{-}(q)$ have $r'$-index. Let us also observe that if $|R_0| = r = n$ then either $f=1$ and a $\C_2$-subgroup of type ${\rm O}_1(q) \wr S_n$ has $r'$-index, or $f>1$ and $\mathcal{M}(R)$ contains a subfield subgroup (since $r$ divides $p^{n-1}-1$). In particular, if $|R_0| = r = n$ then $|\mathcal{M}(R)| \geqs 2$. 

By inspecting \cite{BHR,KL}, we see that any additional subgroup in $\mathcal{M}(R)$ must be in one of the collections $\C_5$ or $\mathcal{S}$. In the usual manner, there are no subfield subgroups in $\mathcal{M}(R)$ if and only if $\a(n-1,1)$ holds. 
So we have reduced the problem to determining whether or not there are any maximal subgroups in $\mathcal{S}$ with $r'$-index.

First assume $n \in \{7,9,11\}$. We can use \cite{BHR} to handle these cases. If $n=7$ then there is a maximal subgroup $H \in \mathcal{S}$ with socle $S = G_2(q)$ and $r'$-index, so $|\mathcal{M}(R)| \geqs 2$. Next assume $n=11$. By inspecting \cite[Table 8.75]{BHR} we see that there are no examples if $r>11$ or $|R_0|>r$, while we noted above that $|\mathcal{M}(R)| \geqs 2$ if $|R_0| = r = 11$. Now assume $n=9$, so $r \geqs 17$. From \cite[Table 8.59]{BHR}, we deduce that there are no examples if $r>17$ or $|R| > r$, so let us assume $|R| = r=17$. If $f=1$ and $r = \square \imod{p}$, or if $f=2$ and $r = \boxtimes \imod{p}$, then there is a maximal subgroup $H \in \mathcal{S}$ with socle ${\rm L}_2(17)$, which clearly has $r'$-index. So for $n=9$ we deduce that there are no subgroups in $\mathcal{S}$ with $r'$-index if and only if $r>17$, or $|R|>r$, or $f>2$, or $f=2$ and $r = \square \imod{p}$, or $f=1$ and $r = \boxtimes \imod{p}$.
 
To complete the proof, we may assume $i=n-1$ and $n \geqs 13$. By inspecting \cite{GPPS}, we find that there are no subgroups in $\mathcal{S}$ with $r'$-index when $r > 2n-1$, so we may assume $r \in \{n,2n-1\}$. For $r=2n-1$, the only possibility is a subgroup with socle ${\rm L}_2(s)$, where $r = s$ is a prime. Here $r \equiv 1 \imod{4}$ (since $n$ is odd) and by inspecting \cite[Table 2(b)]{HM} we see that either $f=1$ and $r = \square \imod{p}$, or $f=2$ and $r = \boxtimes \imod{p}$. In addition, we have $G=T$ and such a subgroup has $r'$-index if and only if $|R| = r$. It is also maximal because it is not contained in a reducible $\C_1$-subgroup of type ${\rm O}_{n-1}^{-}(q)$.
Finally, if $r=n \geqs 13$ then by inspecting \cite{GPPS} it is straightforward to check that there are no subgroups in $\mathcal{S}$ with $r'$-index when $|R_0|>r$. For example, if $H$ has socle $S = {\rm L}_d(s)$ and $n = r = (s^d-1)/(s-1)$ (see \cite[Table 8]{GPPS}), then $|H_0|$ is clearly indivisible by $r^2$. And finally, if $|R_0| = r = n$ then $|\mathcal{M}(R)| \geqs 2$, as noted above. 
\end{proof}

\subsection{Even dimensional orthogonal groups}

To complete the proof of Theorem \ref{t:main1} for classical groups, we may assume $T$ is an even dimensional orthogonal group. We will handle the plus-type and minus-type groups separately.

\begin{prop}\label{p:pso_plus}
The conclusion to Theorem \ref{t:main1} holds when $T = {\rm P\O}_n^{+}(q)$ and $n \geqs 8$. 
\end{prop}

\begin{proof}
We will show that $|\mathcal{M}(R)| \geqs 2$ in all cases. If $r=p$ then every maximal parabolic subgroup has $r'$-index and there are at least two conjugacy classes of such subgroups in $G$ (including the special case where $n=8$, $r=3$  and $G$ contains triality automorphisms). Now assume $r \ne p$ and let $i = d_q(r)$. We may assume $r$ divides $|T|$, in which case $i \leqs n-2$. 

\vs

\noindent \emph{Case 1. $i \in \{1,2\}$}

\vs

First assume $G$ does not contain any triality graph or graph-field automorphisms when $n=8$. If $i=1$ then $G$ contains two conjugacy classes of maximal parabolic subgroups of type $P_{n/2}$ (corresponding to the two $T$-orbits on maximal totally singular subspaces of the natural module). These subgroups have $r'$-index and thus $|\mathcal{M}(R)| \geqs 2$. 

Now assume $i=2$, so $q \ne 3$. If $n \equiv 2 \imod{4}$ then $\C_1$-subgroups of type $P_1$ and $N_1$ have $r'$-index and thus $|\mathcal{M}(R)| \geqs 2$. Now assume $n \equiv 0 \imod{4}$. Here a $\C_2$-subgroup of type ${\rm O}_2^{-}(q) \wr S_{n/2}$ is maximal (see \cite{BHR,KL}) and has $r'$-index. In addition, we can take a $\C_3$-subgroup of type ${\rm GU}_{n/2}(q)$ (in fact, there are two conjugacy classes of such subgroups).

To complete the analysis of this case, let us assume $n=8$, $r=3$ and $G$ contains triality graph or graph-field automorphisms. Set $\e = +$ if $i=1$, otherwise $\e=-$, so $|R_0| = 3((q-\e)_3)^4$ and by inspecting \cite[Table 8.50]{BHR} we see that $G$ has a maximal subgroup of type ${\rm GL}_3^{\e}(q) \times {\rm GL}_1^{\e}(q)$ with $3'$-index. 
If $i=2$ or $q \geqs 7$, then $\mathcal{M}(R)$ also contains a $\C_2$-subgroup of type ${\rm O}_2^{\e}(q) \wr S_4$. Finally, if $(i,q)=(1,4)$ then a subfield subgroup of type ${\rm O}_8^{+}(2)$ has $3'$-index, and we conclude that $|\mathcal{M}(R)| \geqs 2$ in all cases.

\vs

\noindent \emph{Case 2. $2<i \leqs n-2$}

\vs

Here $r \geqs 5$ and we will first assume $i$ is odd. If $i$ does not divide $n/2$, then $(q^{n/2}-1)_r = 1$ and one can check that the $\C_1$-subgroups $P_{n/2}$ and $N_1$ have $r'$-index. On the other hand, if $i$ divides $n/2$ then a $\C_2$-subgroup of type ${\rm O}_{2i}^{+}(q) \wr S_{n/2i}$ has $r'$-index, and so does a $\C_3$-subgroup of type ${\rm O}_{n/k}^{+}(q^k)$, where $k$ is a prime divisor of $i$. 

Now assume $i$ is even. If $i$ does not divide $n/2$ then $N_1$ has $r'$-index, and so does a $\C_1$-subgroup of type ${\rm O}_{2}^{\e}(q) \times {\rm O}_{n-2}^{\e}(q)$, where $\e=+$ if $i$ divides $(n-2)/2$, otherwise $\e=-$. Finally, suppose $i$ divides $n/2$. Here a $\C_2$-subgroup of type ${\rm O}_{i}^{-}(q) \wr S_{n/i}$ has $r'$-index. In addition, a $\C_3$-subgroup of type ${\rm O}_{n/k}^{+}(q^k)$ also has $r'$-index, where $k=2$ if $i \equiv 0 \imod{4}$ and $k$ is an odd prime divisor of $i$ if $i \equiv 2 \imod{4}$. 
\end{proof}

\begin{prop}\label{p:pso_minus}
The conclusion to Theorem \ref{t:main1} holds when $T = {\rm P\O}_n^{-}(q)$ and $n \geqs 8$. 
\end{prop}

\begin{proof}
If $r=p$ then every maximal parabolic subgroup has $r'$-index and there are at least two conjugacy classes of such subgroups in $G$. Now assume $r \ne p$ and set $i = d_q(r)$. 

\vs

\noindent \emph{Case 1. $i \in \{1,2\}$}

\vs

For $i=1$ we observe that $N_1$ and $P_{n/2-1}$ both have $r'$-index (note that $(q^{n/2}+1)_r = 1$). Now assume $i=2$. If $n \equiv 0 \imod{4}$ then $(q^{n/2}+1)_r = 1$ and $(q^{n-2}-1)_r = (q^{(n-2)/2}+1)_r$, which means that the $\C_1$-subgroups of type $N_1$ and 
${\rm O}_{2}^{+}(q) \times {\rm O}_{n-2}^{-}(q)$ have $r'$-index. And for $n \equiv 2 \imod{4}$ we can take a $\C_2$-subgroup of type ${\rm O}_{2}^{-}(q) \wr S_{n/2}$ and a $\C_3$-subgroup of type ${\rm GU}_{n/2}(q)$.

\vs

\noindent \emph{Case 2. $2 < i < n$}

\vs

If $i$ is odd, then  $(q^{n/2}+1)_r=1$ and we deduce that the $\C_1$-subgroups $N_1$ and $P_{n/2-1}$ have $r'$-index.

Now assume $i$ is even. If $i$ divides $n/2$ then $n \equiv 0 \imod{4}$ and we have $(q^{n/2}+1)_r=1$ and $(q^{n-2}-1)_r=(q^{(n-2)/2}+1)_r$. Therefore, the $\C_1$-subgroups $P_1$ and $N_1$ have $r'$-index. If $i$ does not divide $n$, then $N_1$ and a $\C_1$-subgroup of type ${\rm O}_{2}^{\e}(q) \times {\rm O}_{n-2}^{-\e}(q)$ have $r'$-index,
where $\e=-$ if $i$ divides $(n-2)/2$, otherwise $\e=+$. Finally, suppose $i$ divides $n$, but does not divide $n/2$, in which case $n/i$ is odd and $(q^{n/2}+1)_r = (q^{i/2}+1)_r(n/i)_r$.  Here one can check that a $\C_2$-subgroup of type ${\rm O}_{i}^{-}(q) \wr S_{n/i}$ has $r'$-index. We can also take a $\C_3$-subgroup of type ${\rm O}_{n/k}^{-}(q^k)$, where $k=2$ if $i \equiv 0 \imod{4}$, and $k$ is an odd prime divisor of $i$ if $i \equiv 2 \imod{4}$.

\vs

\noindent \emph{Case 3. $i = n$}

\vs

First observe that $|R_0| = (q^{n/2}+1)_r$ and $r \geqs n+1$. If $n \equiv 0 \imod{4}$ then a $\C_3$-subgroup of type ${\rm O}_{n/2}^{-}(q^2)$ has $r'$-index. Similarly, a $\C_3$-subgroup of type ${\rm GU}_{n/2}(q)$ has $r'$-index if $n \equiv 2 \imod{4}$. If $k \geqs 3$ is a prime divisor of $n$ with $n > 2k$, then a $\C_3$-subgroup of type ${\rm O}_{n/k}^{-}(q^k)$ has $r'$-index. Therefore, $|\mathcal{M}(R)| = 1$ only if one of the following holds:
\begin{itemize}\addtolength{\itemsep}{0.2\baselineskip}
\item[{\rm (a)}] $n=2a$, where $a \geqs 5$ is a prime; or
\item[{\rm (b)}] $n=2^a$, where $a \geqs 3$.
\end{itemize}

As noted above, $\mathcal{M}(R)$ contains a $\C_3$-subgroup of type ${\rm GU}_{n/2}(q)$ or ${\rm O}_{n/2}^{-}(q^2)$ in the two respective cases. In the usual manner, we deduce that there are no subfield subgroups with $r'$-index if and only if $\b(n/2,-1)$ holds. And then by inspecting \cite{BHR,KL} we see that any additional subgroup in $\mathcal{M}(R)$ must be an almost simple group in the collection $\mathcal{S}$. Since neither a $\C_3$-subgroup of type ${\rm GU}_{n/2}(q)$ or ${\rm O}_{n/2}^{-}(q^2)$, nor a subfield subgroup of $G$ contains a subgroup in $\mathcal{S}$, it follows that $|\mathcal{M}(R)| \geqs 2$ if there exists a subgroup in $\mathcal{S}$ with $r'$-index. We now divide our analysis of the subgroups in $\mathcal{S}$ into two cases.

\vs

\noindent \emph{Case 3(a). $i = n$ and $n=2a$, where $a \geqs 5$ is a prime}

\vs

First assume case (a) holds. If $n=10$ then by inspecting \cite[Table 8.69]{BHR} we deduce that there is a maximal subgroup in $\mathcal{S}$ with $r'$-index if and only if $f=1$, $|R|=r=11$ and $p = \boxtimes \imod{11}$. In fact, if $r=11$ then our hypotheses imply that $p = \boxtimes \imod{11}$ and thus $|\mathcal{M}(R)| \geqs 2$. Indeed, if $p \equiv x^2 \imod{11}$, then $11$ divides $p^5-1$ and thus $11$ divides $q^5-1$, which is incompatible with the assumption $i=10$.

For the remainder, we may assume $n \geqs 14$. By carefully applying the main theorem of \cite{GPPS}, we deduce that there are no possibilities with $r>2n+1$. If $r=2n+1$, then the only option is a subgroup with socle ${\rm L}_2(r)$. Here $r \equiv 1 \imod{4}$ and by inspecting \cite[Table 2(b)]{HM} we see that the Frobenius-Schur indicator of the corresponding representation is incompatible with an embedding in $G$ (indeed, we get an embedding in a symplectic group). So this case does not arise. 

To complete the analysis of case (a), we may assume $r=n+1$ and $n \geqs 14$. By inspecting \cite{GPPS}, we see that there are no $r'$-index subgroups in $\mathcal{S}$ if $|R_0|>r$, so we may assume $|R_0| = r$. Now $r$ divides $p^n-1$, so every maximal subfield subgroup of $G$ has $r'$-index and therefore we may assume $f$ is a $2$-power. In particular, since $G/T$ is an $r$-group, it follows that $G = T$. Let us assume $H \in \mathcal{S}$ has socle $S$ and $r'$-index. 

First assume $S = A_d$ is an alternating group, which is irreducibly embedded in $G$ via the fully deleted permutation module over $\mathbb{F}_p$. Here $f=1$ and we note that $p$ does not divide $n+1$ (since $r = n+1$). Therefore, for $p$ odd there exists a subgroup in $\mathcal{S}$ with socle $A_{n+1}$ and $r'$-index, so $|\mathcal{M}(R)| \geqs 2$. Now assume $q=2$. We claim that $a \equiv 1 \imod{4}$. To see this, suppose $a \equiv 3 \imod{4}$, which means that $r \equiv -1 \imod{8}$ and thus $q$ is a square modulo $r$, say $q \equiv x^2 \imod{r}$. But then $q^a \equiv x^{r-1} \equiv 1 \imod{r}$, which contradicts the fact that $d_q(r) = n = 2a$. This justifies the claim and we deduce that $A_{n+2} < G$ has $r'$-index. In particular, $|\mathcal{M}(R)| \geqs 2$ if $f=1$. 

In fact, we claim that $|\mathcal{M}(R)| \geqs 2$ for all $f$. To justify this, suppose $p$ is odd and consider the fully deleted permutation module for $A_{n+1}$ over $\mathbb{F}_q$. This is an orthogonal module, so the corresponding representation embeds $A_{n+1}$ in $\Omega_{n}^{\e}(q)$ and we see that $\e=-$ since $|\Omega_{n}^{+}(q)|$ is indivisible by $r$. Therefore, we have $R < A_{n+1} < G$. Now $A_{n+1}$ is not contained in a $\C_3$-subgroup of type ${\rm GU}_{n/2}(q)$ because $A_{n+1}$ does not have a nontrivial representation of dimension less than $n$ over any field. Therefore, $R$ is contained in some other maximal subgroup of $G$ (necessarily in $\mathcal{S}$) and this justifies the claim. An entirely similar argument applies when $p=2$, working with $R < A_{n+2} < G$. 

\vs

\noindent \emph{Case 3(b). $i = n$ and $n=2^a$, where $a \geqs 3$}

\vs

If $n=8$ then by inspecting \cite[Table 8.53]{BHR} we see that there are no subgroups in $\mathcal{S}$ with $r'$-index, so we may assume $a \geqs 4$ for the remainder. Suppose $H \in \mathcal{S}$ has $r'$-index and socle $S$. If $r \geqs 2n+1$, then the main theorem of \cite{GPPS} implies that $S = {\rm L}_2(r)$ with $r = 2n+1$ is the only possibility. But by inspecting \cite[Table 2(b)]{HM} we find that the Frobenius-Schur indicator is incompatible with an embedding in $T$, so this case does not arise.

Finally, let us assume $r = n+1$, so $r$ is a Fermat prime. First note that $q$ is odd. Indeed, $r = 2^a+1$ divides $2^{2a}-1$ and thus $r$ divides $q^{2a}-1$ if $q$ is even, which is incompatible with the fact that $d_q(r) = n = 2^a$. By inspecting \cite{GPPS} we see that there are no subgroups in $\mathcal{S}$ with $r'$-index if $|R_0|>r$, so we may assume $|R_0| = r$.

Suppose $S = A_d$ is irreducibly embedded in $G$ via the fully deleted permutation module over $\mathbb{F}_p$. Then for $q=p$ it follows that there is a subgroup in $\mathcal{S}$ with socle $A_{n+1}$ and $r'$-index. By repeating the argument in Case (a), this implies that $|\mathcal{M}(R)| \geqs 2$ for all $q$ and the proof of the proposition is complete.
\end{proof}

\section{Proof of Theorem \ref{t:main1}: Exceptional groups}\label{s:t2}

In this section, we complete the proof of Theorem \ref{t:main1} by handling the exceptional groups of Lie type. Here it will be convenient to adopt the standard Lie notation for groups of Lie type. So for example, we will often write $A_{n-1}(q)$ for ${\rm SL}_n(q)$ (and also ${\rm PGL}_n(q)$ and ${\rm L}_n(q)$) and $A_{n-1}^{-}(q)$ for ${\rm SU}_n(q)$ (and also ${\rm PGU}_n(q)$ and ${\rm U}_n(q)$). As before, we write $q = p^f$ with $p$ a prime. In addition, if $H$ is a subgroup of $G$, then we set $H_0 = H \cap T$.

\begin{prop}\label{p:2b2g2}
The conclusion to Theorem \ref{t:main1} holds when $T = {}^2B_2(q)$ or ${}^2G_2(q)'$.
\end{prop}

\begin{proof}
First assume $T = {}^2B_2(q)$, so $|T| = q^2(q-1)(q^2+1)$ and $q = 2^f$ with $f \geqs 3$ odd. The maximal subgroups of $G$ are listed in \cite[Table 8.16]{BHR}. As usual, let $i = d_q(r)$ and note that $i \in \{1,4\}$. If $i=1$ then $|R_0| = (q-1)_r$ and we see that $\mathcal{M}(R)$ contains subgroups $H$ with $H_0 = [q^{2}]{:}(q-1)$ and $D_{2(q-1)}$. Now assume $i=4$, so $|R_0| = (q^2+1)_r$ and $r$ divides $m=q+\e\sqrt{2q}+1$ for some $\e  = \pm$. In particular, $\mathcal{M}(R)$ contains the normaliser of a maximal cyclic torus of order $m$. By inspection, any additional $H \in \mathcal{M}(R)$ must be a subfield subgroup with $H_0 = {}^2B_2(q_0)$ and $q=q_0^k$, where $k$ is a proper prime divisor of $f$. Since $|H_0|_r = (q_0^2+1)_r$ we conclude that $|\mathcal{M}(R)| = 1$ if and only if $(q^{1/k})^2 \not\equiv -1 \imod{r}$ for every prime divisor $k \ne r,f$ of $f$. 

Now assume $T = {}^2G_2(q)'$, so $q = 3^f$ and $f$ is odd. If $f=1$ then $T \cong {\rm L}_2(8)$ and we find that $|\mathcal{M}(R)| = 1$ if and only if $r=3$. For the remainder, let us assume $f \geqs 3$ and note that $|T| = q^3(q-1)(q^3+1)$. We refer the reader to \cite[Table 8.43]{BHR} for the maximal subgroups of $G$. If $r=3$ then $\mathcal{M}(R) = \{H\}$, where $H_0 = [q^3]{:}(q-1)$ is a Borel subgroup of $T$. Now assume $r \geqs 5$ and set $i = d_q(r)$, so $i \in \{1,2,6\}$. If $i=1$ then $|R_0| = (q-1)_r$ and $\mathcal{M}(R)$ contains a Borel subgroup and the centraliser of an involution (so $H_0 = 2 \times {\rm L}_2(q)$ in the latter case).  Similarly, if $i=2$ then $\mathcal{M}(R)$ contains subgroups with $H_0 = 2 \times {\rm L}_2(q)$ and $(2^2 \times D_{(q+1)/2}){:}3$. Finally, suppose $i=6$. Here $|R_0| = (q^3+1)_r$ and $r$ divides $m=q+\e\sqrt{3q}+1$ for some $\e = 
\pm$. In this case, $\mathcal{M}(R)$ contains the normaliser of a maximal torus of order $m$ and any other subgroup in $\mathcal{M}(R)$ has to be a subfield subgroup of type ${}^2G_2(q_0)$, where $q=q_0^k$ and $k$ is a prime. We deduce that $|\mathcal{M}(R)| = 1$ if and only if $\a(3,-1)$ holds (see \eqref{e:ab}). 
\end{proof}

\begin{prop}\label{p:2f4}
The conclusion to Theorem \ref{t:main1} holds when $T = {}^2F_4(q)'$.
\end{prop}

\begin{proof}
Here $q = 2^f$ with $f$ odd. If $f=1$ then a {\sc Magma} calculation shows that $|\mathcal{M}(R)| \geqs 2$, so we may assume $f \geqs 3$. Note that
\[
|T| = q^{12}(q-1)(q^3+1)(q^4-1)(q^6+1)
\]
and the maximal subgroups of $G$ are determined in \cite{Malle} (noting the correction in \cite{Craven}, where Craven shows that ${}^2F_4(8)$ has three classes of maximal subgroups isomorphic to ${\rm PGL}_2(13)$ that were omitted in \cite{Malle}). Set $i = d_q(r)$, so $i \in \{1,2,4,6,12\}$.

First assume $i=1$, so $|R_0| = ((q-1)_r)^2$ and every maximal parabolic subgroup of $G$ has $r'$-index, whence $|\mathcal{M}(R)| \geqs 2$. Similarly, if $i=2$ then $|R_0| = ((q+1)_r)^2(3)_r$ and $\mathcal{M}(R)$ contains subgroups with $H_0 = {\rm SU}_3(q){:}2$ and ${\rm PGU}_3(q){:}2$. The case $i=6$ is entirely similar. 

Next assume $i=4$.  Here $r \geqs 5$, $|R_0| = ((q^2+1)_r)^2$ and $\mathcal{M}(R)$ contains a subgroup with $H_0 = {}^2B_2(q) \wr S_2$. Now
\[
q^2+1 = (q+\sqrt{2q}+1)(q-\sqrt{2q}+1)
\]
and thus $r$ divides $m = q+\e\sqrt{2q}+1$ for some $\e = \pm$. If $\e=+$ then $\mathcal{M}(R)$ contains the normaliser of a maximal torus of order $m^2$, and the same conclusion holds if $\e = -$ and $f \geqs 5$. So let us assume $\e = -$ and $f=3$, in which case $q=8$ and $r=5$. Here $|R_0| = 5^2$ and $\mathcal{M}(R)$ contains a subfield subgroup ${}^2F_4(2)$, so once again we conclude that $|\mathcal{M}(R)| \geqs 2$. 

Finally, let us assume $i=12$. Here $r \in \{13,37,61, \ldots\}$, $|R_0| = (q^6+1)_r = (q^4-q^2+1)_r$ and we note that 
\[
q^4-q^2+1 = (q^2+\sqrt{2q^3}+q+\sqrt{2q}+1)(q^2-\sqrt{2q^3}+q-\sqrt{2q}+1),
\]
so $r$ divides $m = q^2+\e\sqrt{2q^3}+q+\e\sqrt{2q}+1$ for some $\e = \pm$. Then $\mathcal{M}(R)$ contains the normaliser of a maximal torus of order $m$ and by inspecting \cite{Malle} we see that any additional subgroup in $\mathcal{M}(R)$ is a subfield subgroup of type ${}^2F_4(q_0)$, where $q=q_0^k$, $k$ is a prime and $(q_0^6+1)_r = (q^6+1)_r$. (Note that if $q=8$ then $r \in \{37,109\}$, so the maximal subgroups ${\rm PGL}_2(13) < {}^2F_4(8)$ do not have $r'$-index.) Therefore, $|\mathcal{M}(R)| = 1$ if and only if $\a(6,-1)$ holds.
\end{proof}

\begin{prop}\label{p:3d4}
The conclusion to Theorem \ref{t:main1} holds when $T = {}^3D_4(q)$.
\end{prop}

\begin{proof}
Let us first observe that
\begin{align*}
|T| & = q^{12}(q^2-1)(q^6-1)(q^8+q^4+1) \\
& = q^{12}(q^2-1)(q^6-1)(q^2-q+1)(q^2+q+1)(q^4-q^2+1)
\end{align*}
and the maximal subgroups of $G$ are listed in \cite[Table 8.51]{BHR}. If $r=p$ then every maximal parabolic subgroup has $r'$-index, so we may assume $r \ne p$. Let $i = d_q(r)$ be the order of $q$ modulo $r$ and note that $i \in \{1,2,3,6,12\}$. 

First assume $i=1$ and note that $|R_0| = ((q-1)_r(3)_r)^2$. If $r \geqs 5$ then $\mathcal{M}(R)$ contains subgroups of type $G_2(q)$ and $A_1(q^3)A_1(q)$, so we may assume $r=3$. Here we observe that maximal rank subgroups of the form $A_2(q) \times (q^2+q+1)$ have $3'$-index and by inspecting \cite[Table 8.51]{BHR} we deduce that any additional subgroups in $\mathcal{M}(R)$ must be of the form ${}^3D_4(q_0)$, where $q = q_0^k$ and $k \ne 3$ is a prime. If $k \geqs 5$ then  $q_0 \equiv 1 \imod{3}$ (since $q \equiv 1 \imod{3}$) and $(q_0-1)_3 = (q-1)_3$, so $H$ has $3'$-index. The same conclusion holds if $k=2$ and $q_0 \equiv 1 \imod{3}$. And if $k=2$ and $q_0 \equiv -1 \imod{3}$, then $(q_0+1)_3 = (q-1)_3$ and once again we deduce that $H$ has $3'$-index. Therefore, if $r=3$ then $|\mathcal{M}(R)| = 1$ if and only if $f=3^a$ for some $a \geqs 0$. 

Next assume $i=2$. As in the previous case, it is easy to see that $|\mathcal{M}(R)| \geqs 2$ if $r \geqs 5$, so we may assume $r=3$. Note that $f$ is odd since  $3$ divides $q-1$ if $f$ is even. Then $|R_0| = 9((q+1)_3)^2$ and maximal rank subgroups of type $A_2^{-}(q) \times (q^2-q+1)$ have $3'$-index. By considering subfield subgroups as above, we deduce that $|\mathcal{M}(R)| = 1$ if and only if $f = 3^a$ for some $a \geqs 0$. 

Suppose $i=3$, so $r \geqs 7$ and $|R_0| = ((q^2+q+1)_r)^2$. Then $\mathcal{M}(R)$ contains maximal rank subgroups of the form $A_2(q) \times (q^2+q+1)$ and $(q^2+q+1)^2$. Similarly, if $i=6$ then subgroups of type 
$A_2^{-}(q) \times (q^2-q+1)$ and $(q^2-q+1)^2$ have $r'$-index.

Finally, let us assume $i=12$, so $r \geqs 13$ and $|R_0| = (q^6+1)_r = (q^4-q^2+1)_r$. Here $\mathcal{M}(R)$ contains the normaliser of a maximal torus of order $q^4-q^2+1$. Any additional subgroup in $\mathcal{M}(R)$ has to be a subfield subgroup of type ${}^3D_4(q_0)$, where $q = q_0^k$ and $k \ne 3$ is a prime. As a consequence, we deduce that $|\mathcal{M}(R)| = 1$ if and only if $(q^{1/k})^6 \not\equiv -1 \imod{r}$ for every prime divisor $k \ne 3,r$ of $f$.
\end{proof}

\begin{prop}\label{p:g2}
The conclusion to Theorem \ref{t:main1} holds when $T = G_2(q)'$.
\end{prop}

\begin{proof}
If $q=2$ then $T \cong {\rm U}_3(3)$ and $|\mathcal{M}(R)| = 1$ if and only if $r \in \{3,7\}$. For the remainder we will assume $q \geqs 3$, so $|T| = q^6(q^2-1)(q^6-1)$ and the maximal subgroups of $G$ are recorded in \cite{BHR} (Table 8.30 for $p=2$, Table 8.41 for $p \geqs 5$ and Table 8.42 for $p=3$).

\vs

\noindent \emph{Case 1. $p=2$}

\vs

First assume $p=2$ and $q \geqs 4$. Let $i = d_q(r)$, so we have $i \in \{1,2,3,6\}$. Suppose $i=1$ and note that $|R_0| = ((q-1)_r)^2(3)_r$. Here a maximal rank subgroup of type $A_2(q)$ has $r'$-index, and so does $A_1(q)^2$  if $r \geqs 5$. For $r=3$ we note that a subfield subgroup of type $G_2(q_0)$ has $3'$-index if and only if $q=q_0^k$ and $k \ne 3$ is a prime. Therefore, if $i=1$ then $|\mathcal{M}(R)| = 1$ if and only if $r=3$ and $f = 3^a$ for some $a \geqs 0$. The same conclusion holds for  $i=2$.

Next assume $i=3$, so $r \geqs 7$. Here $\mathcal{M}(R)$ contains a maximal rank subgroup of type $A_2(q)$. It also contains a subfield subgroup of type $G_2(q_0)$ if $q=q_0^k$, $k \ne r$ is a prime and $q_0^3 \equiv 1 \imod{r}$. By inspecting \cite[Table 8.30]{BHR}, we conclude that $|\mathcal{M}(R)| = 1$ if and only if $\a(3,1)$ holds. Similarly, if $i=6$ then $q \geqs 4$ (recall that $2^6-1$ does not have a primitive prime divisor) and $|\mathcal{M}(R)| = 1$ if and only if $q \geqs 8$ and $\a(3,-1)$ holds. (Note that if $i=6$ and $q=4$, then $r=13$ and $\mathcal{M}(R)$ contains a subgroup with socle ${\rm L}_2(13)$.)

\vs

\noindent \emph{Case 2. $p=3$}

\vs

It is straightforward to show that $|\mathcal{M}(R)| \geqs 2$ when $p=3$. Indeed, if $r=p$ then every maximal parabolic subgroup has $r'$-index, so we can assume $r \ne p$ and we define $i = d_q(r)$ as above. If $i \in \{1,3\}$ then $\mathcal{M}(R)$ contains at least two subgroups of type $A_2(q)$ (note that $G$ has two conjugacy classes of such subgroups). Similarly, if $i \in \{2,6\}$ then there are at least two subgroups of type $A_2^{-}(q)$ in $\mathcal{M}(R)$.

\vs

\noindent \emph{Case 3. $p \geqs 5$}

\vs

Finally, let us assume $p \geqs 5$. For $r=p$ we note that every maximal parabolic subgroup has $r'$-index, so we may assume $r \ne p$. Define $i = d_q(r) \in \{1,2,3,6\}$ as before. The conditions here are more complicated than in Cases 1 and 2.

First assume $i=1$, so $|R_0| = ((q-1)_r)^2(3)_r$. Here $\mathcal{M}(R)$ contains a maximal rank subgroup of type $A_2(q)$, and it also contains one of type $A_1(q)^2$ if $r \geqs 5$. Now assume $r=3$.  If $q=q_0^k$ with $k$ a prime, then a subfield subgroup of type $G_2(q_0)$ has $3'$-index if and only if $(q_0-1)_3 = (q-1)_3$, which is equivalent to the condition $k \ne 3$. It follows that there are no subfield subgroups in $\mathcal{M}(R)$ if and only if $f = 3^a$ for some $a \geqs 0$. By inspecting \cite[Table 8.41]{BHR}, we note that if $f=1$ and $|R| = 3^3$, then $\mathcal{M}(R)$ contains a subgroup ${\rm U}_3(3){:}2$. Therefore, we conclude that if $i=1$ then $|\mathcal{M}(R)| = 1$ if and only if $r=3$ and either $f = 3^a$ with $a \geqs 1$, or $f=1$ and $q \equiv 1 \imod{9}$. The same conclusion holds for $i=2$, replacing the condition $q \equiv 1 \imod{9}$ by $q \equiv -1 \imod{9}$.

Next assume $i=3$, so $r \in \{7,13,19, 31, \ldots\}$ and $|R_0| = (q^3-1)_r$. A maximal rank subgroup of type $A_2(q)$ has $r'$-index and we note that there are no subfield subgroups in $\mathcal{M}(R)$ if and only if $\a(3,1)$ holds. So if we assume that the latter condition holds, then by inspecting \cite[Table 8.41]{BHR} we deduce that $|\mathcal{M}(R)| = 1$ if $r>13$, or $|R_0| >r$, or $f>3$, or if one of the following holds:
\begin{itemize}\addtolength{\itemsep}{0.2\baselineskip}
\item[{\rm (a)}] $f=3$ and either $r=13$, or $r=7$ and $p \equiv \pm 1, \pm 3 \imod{9}$.
\item[{\rm (b)}] $f=2$, $r \in \{7,13\}$ and $p = \square \imod{13}$.
\item[{\rm (c)}] $f=1$, $r=13$ and $p = \boxtimes \imod{13}$.
\end{itemize}
(Note that if $f=1$ and $|R_0| = r = 7$, then ${\rm U}_3(3){:}2 \in \mathcal{M}(R)$.) The same conclusion holds for $i=6$, replacing the condition $\a(3,1)$ by $\a(3,-1)$.
\end{proof}

\begin{prop}\label{p:f4}
The conclusion to Theorem \ref{t:main1} holds when $T = F_4(q)$.
\end{prop}

\begin{proof}
Here $|T| = q^{24}(q^2-1)(q^6-1)(q^8-1)(q^{12}-1)$ and we may assume $r \ne p$. Let $i = d_q(r)$ and note that $i \in \{1,2,3,4,6,8,12\}$. The maximal subgroups of $G$ have been determined (up to conjugacy) by Craven (see Tables 1.1, 7.1 and 7.2 in \cite{Craven}). By inspection, it is straightforward to identify the maximal subgroups with $r'$-index recorded in Table \ref{tab:ex1}, so to complete the proof we may assume $i \in \{8,12\}$ and $q$ is odd.

Suppose $i=8$ and $q$ is odd, in which case $|R_0| = (q^4+1)_r$ and $\mathcal{M}(R)$ contains a subgroup of type $B_4(q)$. Note that $r \in \{17, 41, 73, \ldots\}$. There are no subfield subgroups in $\mathcal{M}(R)$ if and only if $\a(4,-1)$ holds. Given this condition, by inspecting \cite[Tables 1.1, 7.1]{Craven} we deduce that $|\mathcal{M}(R)| = 1$ if and only if $r>17$, or $|R| >r$, or $f>2$, or one of the following holds:
\begin{itemize}\addtolength{\itemsep}{0.2\baselineskip}
\item[{\rm (a)}] $f = 2$ and either $p=3$ or $p = \square \imod{17}$.
\item[{\rm (b)}] $f=1$ and $p = \boxtimes \imod{17}$.
\end{itemize}

Finally suppose $i=12$ and $q$ is odd, so $r \in \{13,37,61, \ldots\}$, $|R_0| = (q^4-q^2+1)_r$ and $\mathcal{M}(R)$ contains a subgroup of type ${}^3D_4(q)$. There are no subfield subgroups in $\mathcal{M}(R)$ if and only if $\a(6,-1)$ holds. If this condition holds, then by inspecting \cite[Tables 1.1, 7.1]{Craven} we see that $|\mathcal{M}(R)| = 1$ if and only if $r>13$, or $|R| > r$, or $f \not\in \{1,3\}$, or $f=3$ and $p \equiv \pm 1 \imod{7}$. (Note that if $|R_0| = r = 13$ then ${}^3D_4(2).3 \in \mathcal{M}(R)$ if 
$f=1$, and similarly ${\rm PGL}_2(13) \in \mathcal{M}(R)$ if $f=3$ and $p \equiv \pm 2, \pm 3 \imod{7}$.)
\end{proof}

\begin{prop}\label{p:e6}
The conclusion to Theorem \ref{t:main1} holds when $T = E_6(q)$.
\end{prop}

\begin{proof}
Set $e = (3,q-1)$ and note that 
\[
|T| = \frac{1}{e}q^{36}(q^2-1)(q^5-1)(q^6-1)(q^8-1)(q^9-1)(q^{12}-1).
\]
We may assume $r \ne p$. Let $i = d_q(r)$, so $i \in \{1,2,3,4,5,6,8,9,12\}$. If $i \ne 9$ then by carefully inspecting \cite{Craven} we can identify at least two maximal subgroups with $r'$-index (see Table \ref{tab:ex1}). 

For example, suppose $i=1$. If $r \geqs 5$, then $|R_0| = ((q-1)_r)^6(5)_r$ and it is easy to see that the maximal parabolic subgroups $P_1$ and $P_6$ have $r'$-index. Now assume $r=3$, so $|R_0| = ((q-1)_3)^63^3$ and a maximal rank subgroup of type $A_2(q)^3$ has $r'$-index. If  $q \geqs 7$ then we can take the normaliser of a split maximal torus (a subgroup of type $(q-1)^6.W$, noting that  $W = {\rm O}_6^{-}(2)$ is the Weyl group of $T$ and we have $|W|_3 = 3^4$). And for $q=4$ we can take a subfield subgroup of type ${}^2E_6(2)$. 

To complete the proof, let us assume $i=9$. Here $r \geqs 19$, $|R_0| = (q^6+q^3+1)_r$ and a  maximal rank subgroup of type $A_2(q^3)$ has $r'$-index. By inspecting \cite[Table 7.3]{Craven}, it remains to determine whether or not there are any subfield subgroups or $\mathcal{S}$ collection subgroups in $\mathcal{M}(R)$, where the latter comprises the almost simple subgroups recorded in \cite[Table 1.2]{Craven}. As usual, there are no subfield subgroups with $r'$-index if and only if $\a(9,1)$ holds. Finally, suppose $H \in \mathcal{S}$ has socle $S$ and observe that $H$ has $r'$-index if and only if $|R|=r=19$ and one of the following holds:
\begin{itemize}\addtolength{\itemsep}{0.2\baselineskip}
\item[{\rm (a)}] $S = {\rm J}_3$, $q=4$ and $G = T$; or
\item[{\rm (b)}] $S = {\rm L}_2(19)$, $G=T$ and either $f=2$ and $p = \boxtimes \imod{5}$, or $f=1$, $p = \square \imod{5}$ and $p = \square \imod{19}$.
\end{itemize}
Therefore, for $i=9$ we conclude that there are no $\mathcal{S}$ collection subgroups in $\mathcal{M}(R)$ if and only if $r>19$, or $|R|>r$, or $f>2$, or $f=2$ and $p = \square \imod{5}$, or $f=1$ and either $p = \boxtimes \imod{5}$ or $p = \boxtimes \imod{19}$.
\end{proof}

\begin{prop}\label{p:2e6}
The conclusion to Theorem \ref{t:main1} holds when $T = {}^2E_6(q)$.
\end{prop}

\begin{proof}
Set $e' = (3,q+1)$ and note that 
\[
|T| = \frac{1}{e'}q^{36}(q^2-1)(q^5+1)(q^6-1)(q^8-1)(q^9+1)(q^{12}-1).
\]
We may assume $r \ne p$. Let $i = d_q(r)$ and note that $i \in \{1,2,3,4,6,8,10,12,18\}$. By inspecting \cite{Craven}, we deduce that $|\mathcal{M}(R)| = 1$ only if $i=18$ (see Table \ref{tab:ex1}).

Suppose $i=18$, so $r \geqs 19$, $|R_0| = (q^6-q^3+1)_r$ and a maximal rank subgroup of type $A_2^{-}(q^3)$ has $r'$-index. By inspecting \cite[Table 7.4]{Craven}, we see that any additional subgroup in $\mathcal{M}(R)$ is either a subfield subgroup or is contained in the collection $\mathcal{S}$. As usual, there are no subfield subgroups in $\mathcal{M}(R)$ if and only if $\b(9,-1)$ holds. Finally, let us assume $H \in \mathcal{S}$ has socle $S$. By considering \cite[Table 1.3]{Craven} we deduce that such a subgroup has $r'$-index if and only if $S = {\rm L}_2(19)$, $|R| = r = 19$, $f=1$, $p = \square \imod{5}$ and $p = \boxtimes  \imod{19}$.
\end{proof}

\begin{prop}\label{p:e7}
The conclusion to Theorem \ref{t:main1} holds when $T = E_7(q)$.
\end{prop}

\begin{proof}
Set $d = (2,q-1)$ and note that 
\[
|T| = \frac{1}{d}q^{63}(q^2-1)(q^6-1)(q^8-1)(q^{10}-1)(q^{12}-1)(q^{14}-1)(q^{18}-1).
\]
We may assume $r \ne p$. Let $i = d_q(r)$, so $i \in \{1,2, \ldots, 10,12,14,18\}$. By considering maximal rank subgroups of $G$ (see  \cite{LSS}), it is straightforward to show that $|\mathcal{M}(R)| \geqs 2$ if $i \ne 18$ (see Table \ref{tab:ex2}).

To complete the proof, we may assume $i = 18$. First observe that $r \in \{19,37,73, \ldots\}$ and a maximal rank subgroup of type ${}^2E_6(q) \times (q+1)$ has $r'$-index. In the usual manner, we see that there are no subfield subgroups in $\mathcal{M}(R)$ if and only if $\a(18,1)$ holds. We now need to determine whether or not there are any additional subgroups in $\mathcal{M}(R)$ and we can use 
\cite[Theorem 1.1]{Craven2} to do this, which gives an almost complete classification of the maximal subgroups of $G$ up to conjugacy. In this way, we deduce that any other subgroup $H$ with $r'$-index must be almost simple with socle $S$ and one of the following holds:
\begin{itemize}\addtolength{\itemsep}{0.2\baselineskip}
\item[{\rm (a)}] $|R| = r = 19$ and either $S = {\rm U}_3(8)$ with $q=p \geqs 3$, or $S = {\rm L}_2(37)$ and either $q=p \ne 37$ is a square modulo $37$, or $q=p^2$ and $p$ is a nonsquare modulo $37$.
\item[{\rm (b)}] $|R| = r = 37$ and $S = {\rm L}_2(37)$, with the same conditions on $q$ as in (a).
\end{itemize}
Therefore, we conclude that there are no $r'$-index subgroups in the collection $\mathcal{S}$ if and only if $r>37$, or $|R|>r$, or $f>2$, or one of the following holds:
\begin{itemize}\addtolength{\itemsep}{0.2\baselineskip}
\item[{\rm (i)}] $f=2$ and $p = \square \imod{37}$; or
\item[{\rm (ii)}] $f=1$ and either $r=19$ and $q=2$ (note that $2$ is not a square modulo $37$), or $r=37$ and $p = \boxtimes \imod{37}$.
\end{itemize}
These conditions are recorded in Table \ref{tab:main2}.
\end{proof}

{\scriptsize
\begin{table}
\[
\begin{array}{clll} \hline
i & F_4(q) & E_6(q) & {}^2\!E_6(q) \\ \hline

1 & D_4(q), A_2(q)^2 & \mbox{$P_1,P_6$ ($r \geqs 5$)} &F_4(q), A_2(q^2)A_2(q) \\
& & \mbox{$A_2(q)^3, (q-1)^6.W$ ($r=3$, $q \geqs 7$)} &   \\ 
& & \mbox{$A_2(q)^3, E_6(2)$ ($r=3$, $q = 4$)} &   \\

2 & D_4(q), A_2^{-}(q)^2 & F_4(q), A_2(q^2)A_2^{-}(q) & \mbox{$A_5^{-}(q)A_1(q), (q+1)^6.W$ ($r \geqs 5$)} \\
& & & \mbox{$A_2^{-}(q)^3, (q+1)^6.W$ ($r=3$, $q \geqs 4$)}\\
& & & \mbox{$A_2^{-}(q)^3, (q+1)^6.W$ ($r=3$, $q=2$, $G = T.3$)} \\
& & & \mbox{$A_2^{-}(q)^3, {\rm Fi}_{22}$ ($r=3$, $q=2$, $G = T$)} \\
3 & {}^3\!D_4(q), A_2(q)^2 & A_2(q)^3, (q^2+q+1)^3 & A_2(q^2)A_2(q), {}^3\!D_4(q) \times (q^2-q+1) \\

4 & \mbox{$D_4(q)$, $B_4(q)$ $(p \ne 2)$} & P_1, P_6 & D_4(q) \times (q+1)^2, D_5^{-}(q) \times (q+1)  \\
& \mbox{$C_4(q)$, 2 classes $(p=2)$} & &  \\  

5 & & P_1, P_6 & \\

6 & {}^3\!D_4(q), A_2^{-}(q)^2 & F_4(q), {}^3\!D_4(q) \times (q^2+q+1) & A_2^{-}(q)^3, {}^3\!D_4(q) \times (q^2-q+1) \\

8 & \mbox{$B_4(q)$ $(p \ne 2)$} & P_1, P_6 & P_2, D_5^{-}(q) \times (q+1) \\
& \mbox{$C_4(q)$, 2 classes $(p=2)$} & & \\ 

9 & & A_2(q^3) &  \\ 

10 & & & P_1, D_5^{-}(q) \times (q+1) \\

12 & \mbox{${}^3\!D_4(q)$ $(p \ne 2)$} & F_4(q), {}^3\!D_4(q) \times (q^2+q+1) & F_4(q), {}^3\!D_4(q) \times (q^2-q+1) \\
& \mbox{${}^3\!D_4(q)$, 2 classes $(p=2)$} & & \\ 

18 & & & A_2^{-}(q^3)  \\ \hline
\end{array}
\]
\caption{The cases $r \ne p$ for $T = F_4(q), E_6(q)$ or ${}^2\!E_6(q)$}
\label{tab:ex1}
\end{table}}

{\scriptsize
\begin{table}
\[
\begin{array}{cll} \hline
i & E_7(q) & E_8(q) \\ \hline
1 & \mbox{$(q-1)^7.W$, $E_6(q) \times (q-1)$ ($r \ne 7$, $q \geqs 7$)} & \mbox{$(q-1)^8.W$, $D_8(q)$ ($r \geqs 7$)} \\
& \mbox{$(q-1)^7.W$, $A_1(q)^7$ ($r = 7$, $q \geqs 11$)} & \mbox{$(q-1)^8.W$, $A_4(q)^2$ ($r=5$)} \\
& \mbox{$E_6(q) \times (q-1)$, $E_7(2)$ ($r = 3$, $q = 4$)} & \mbox{$(q-1)^8.W$, $A_2(q)^4$ ($r=3$, $q \geqs 7$)} \\
& & \mbox{$E_8(2)$, $A_2(q)^4$ ($r=3$, $q=4$)} \\
2 & \mbox{$(q-1)^8.W$, $D_6(q)A_1(q)$ ($r \ne 3,7$)} & \mbox{$(q+1)^8.W$, $D_8(q)$ ($r \geqs 7$)} \\
& \mbox{$(q+1)^7.W$, ${}^2\!E_6(q) \times (q+1)$ ($r = 3$)} & \mbox{$(q+1)^8.W$, $A_4^{-}(q)^2$ ($r = 5$)}\\
& \mbox{$(q+1)^7.W$, $A_1(q)^7$ ($r = 7$)} & \mbox{$(q+1)^8.W$, $A_2^{-}(q)^4$ ($r = 3$)} \\
3 & A_1(q^3)\,{}^3\!D_4(q), A_5(q)A_2(q) & (q^2+q+1)^4, A_2(q)E_6(q) \\
4 & A_7(q), D_6(q)A_1(q) & \mbox{$(q^2+1)^4$, $D_8(q)$ ($r \geqs 7$)} \\ 
& & \mbox{$(q^2+1)^4$, $A_4^{-}(q^2)$ ($r = 5$)} \\
5 & A_7(q), D_6(q)A_1(q) & (q^4+q^3+q^2+q+1)^2, A_4(q)^2  \\
6 & A_1(q^3)\, {}^3\!D_4(q), A_5^{-}(q)A_2^{-}(q) & (q^2-q+1)^4, A_2^{-}(q)^4 \\
7 & A_7(q), A_1(q^7) & A_8(q), A_1(q)E_7(q) \\
8 & A_7(q), D_6(q)A_1(q) & D_8(q), D_4(q^2) \\
9 & P_7, E_6(q) \times (q-1) & A_8(q), A_1(q)E_7(q) \\ 
10 & A_7^{-}(q), D_6(q)A_1(q) & (q^4-q^3+q^2-q+1)^2, A_4^{-}(q)^2 \\
12 & \mbox{$E_6^{\e}(q) \times (q-\e)$, $\e = \pm$} & (q^4-q^2+1)^2, {}^3\!D_4(q)^2 \\ 
14 & A_7^{-}(q), A_1(q^7) & A_8^{-}(q), A_1(q)E_7(q) \\
15 & & q^8-q^7+q^5-q^4+q^3-q+1 \\
18 & {}^2\!E_6(q) \times (q+1) & A_8^{-}(q), {}^2\!E_6(q)A_2^{-}(q) \\ 
20 & & \mbox{$A_{4}^{-}(q^2)$ ($2$ classes)} \\
24 & & A_2^{-}(q^4), {}^3\!D_4(q^2) \\ 
30 & & q^8+q^7-q^5-q^4-q^3+q+1 \\ \hline
\end{array}
\]
\caption{The cases $r \ne p$ for $T = E_7(q)$ or $E_8(q)$}
\label{tab:ex2}
\end{table}
}

Finally, we turn to the groups with socle $T = E_8(q)$. Here we will need the following result of Craven \cite{Craven4}.

\begin{prop}\label{p:e8psl}
Let $G$ be an almost simple group with socle $T = E_8(q)$, where $q=p^f$ and $p \ne 61$ is a prime. Then $G$ has a maximal subgroup with socle ${\rm L}_2(61)$ if and only if $f$ is minimal such that $d_q(61)$ is indivisible by $4$.
\end{prop}

\begin{prop}\label{p:e8}
The conclusion to Theorem \ref{t:main1} holds when $T = E_8(q)$.
\end{prop}

\begin{proof}
First observe that
\[
|T| = q^{120}(q^2-1)(q^8-1)(q^{12}-1)(q^{14}-1)(q^{18}-1)(q^{20}-1)(q^{24}-1)(q^{30}-1).
\]
If $r=p$ then every maximal parabolic subgroup has $r'$-index, so assume $r \ne p$. As usual, set $i = d_q(r)$ and note that
\[
i \in \{1, 2, \ldots, 10, 12, 14, 15, 18, 20, 24, 30\}.
\]
For $i \ne 15,30$, we can use the information in 
\cite[Tables 5.1, 5.2]{LSS} on maximal rank subgroups to identify two maximal subgroups with $r'$-index (see Table \ref{tab:ex2}). 

To complete the proof, we may assume $i \in \{15,30\}$, so $r \in \{31, 61, 151, \ldots\}$ and we have $|R_0|=(q^{30}-1)_r$. In addition, we note that the normaliser in $G$ of a maximal torus of order $\Phi_{i}(q) = q^8-\e q^7+\e q^5-q^4+\e q^3-\e q+1$ has $r'$-index, where $\e=+$ if $i = 15$ and $\e = -$ if $i=30$. In the usual manner, we see that there are no subfield subgroups in $\mathcal{M}(R)$ if and only if $\a(30,1)$ holds. By considering \cite[Theorem 8]{LS03}, we deduce that any additional maximal subgroup $H$ with $r'$-index is either an exotic local subgroup, or an almost simple subgroup in the collection $\mathcal{S}$.  

First we claim that $|\mathcal{M}(R)| \geqs 2$ if $|R| = r = 31$. To see this, first observe that $r$ divides $p^{30}-1$ and thus every subfield subgroup of $G$ has $r'$-index. Therefore, we may assume $f=1$ and we note that $q \geqs 3$ since  $31=2^5-1$. By inspecting \cite[Table 1]{CLSS} we see that $G$ has a maximal exotic local subgroup $H = 2^{5+10}.{\rm L}_5(2)$ with $r'$-index and this justifies the claim. For the remainder, we may assume $r>31$ or $|R|>r$, which means that $G$ has no maximal exotic local subgroups with $r'$-index (see \cite{CLSS}).

Finally, let us assume $H$ is almost simple with socle $S$, and let ${\rm Lie}(p)$ be the set of simple groups of Lie type in characteristic $p$. 

First assume $S \in {\rm Lie}(p)$. Here we apply \cite[Theorem 1.1]{Craven3}, which describes the possibilities for $S$ under the assumption that $S$ is not isomorphic to a group of the form ${\rm L}_{2}(q')$. It is easy to see that none of the possibilities have order divisible by $r$, so we may assume $S = {\rm L}_{2}(q')$ with $q' = p^e$ for some $e \geqs 1$.  By applying the main theorem of \cite{LS98}, we may assume $q' \leqs 1312d$, where $d = (2,p-1)$. Now such a subgroup has $r'$-index only if $r$ divides $p^{2e}-1$, which implies that $r$ divides $q^{2e}-1$. Therefore, $e$ has to be divisible by $15$ and thus $q' \geqs 2^{15} > 1312d$. We conclude that there are no such subgroups in $\mathcal{M}(R)$.

Finally, let us assume $S \not\in {\rm Lie}(p)$. Here the possibilities for $S$ are described in \cite{LS99} and the fact that $|H|$ is divisible by $r$ is very restrictive. Indeed, we see that $H$ has $r'$-index only if $r \in \{31,61\}$ and $|R|=r$, so in view of our earlier observations, we may assume $|R| = r = 61$ and it follows that $S = {\rm L}_2(61)$ is the only possibility. By Proposition \ref{p:e8psl}, $G$ has a maximal subgroup of this form if and only if $q=p^f$ and $f$ is minimal such that $d_{p^f}(r)$ is indivisible by $4$. Since $d_p(r)$ divides $60$, it follows that either $d_p(r)$ or $d_{p^2}(r)$ is indivisible by $4$. Therefore $f \leqs 2$, with equality if and only if $d_p(r) = 60$ and $i=30$. 

In conclusion, if $|R| = r = 61$ then $\mathcal{M}(R)$ contains an almost simple subgroup with socle ${\rm L}_2(61)$ if $f=1$ (since $d_q(r) = d_p(r) \in \{15,30\}$ is indivisible by $4$) and it does not contain such a subgroup if $f>2$. And for $f=2$, there is no such subgroup in $\mathcal{M}(R)$ if and only if $d_p(r) \in \{15,30\}$ (in which case $i=15$).
\end{proof}

This completes the proof of Theorem \ref{t:main1} for exceptional groups of Lie type. By combining this with Propositions \ref{p:spor} and \ref{p:alt}, together with the results in Sections \ref{s:t1}, we conclude that the proof of Theorem \ref{t:main1} is complete. 

\section{Proofs of the main corollaries}\label{s:corols}

In this section we prove Corollaries \ref{c:NGR}, \ref{c:NGR2}, \ref{c:Or}, \ref{c:OrH}  and \ref{c:max}. 

\begin{proof}[Proof of Corollary \ref{c:NGR}]
As in the statement of the corollary, let $G$ be an almost simple group with socle $T$, let $R$ be a Sylow $r$-subgroup of $G$ and assume $G/T$ is an $r$-group. Set $R_0 = R \cap T$, which is a Sylow $r$-subgroup of $T$. Our goal is to show that $\mathcal{M}(R) = \{N_G(R_0)\}$ if and only if $(G,r)$ is one of the cases recorded in Table \ref{tab:NGR}, which is presented in Section \ref{s:tables}. 

We proceed by inspecting the triples $(G,r,H)$ with $\mathcal{M}(R) = \{H\}$ arising in Theorem \ref{t:prime2} (for $r=2$) and Theorem \ref{t:main1} (for $r$ odd). Here $R_0 \ne 1$ by Lemma \ref{l:trivial} and thus $H = N_G(R_0)$ if and only if $R_0 = O_r(H_0)$. We consider each possibility for $T$ in turn.

First assume $T = A_n$ is an alternating group. If $r = 2$ then Theorem \ref{t:prime2} implies that $H_0 = A_{n-1}$ and we deduce that $n = 5$ is the only possibility. Here $T \cong {\rm L}_2(4)$ and this case is recorded as $({\rm L}_2(4), 2, P_1)$ in Table \ref{tab:NGR}. Now assume $r$ is odd, so $G = T$ and the possibilities are presented in part (i) of Theorem \ref{t:main1}. By inspection, it is easy to check that $H = N_G(R_0)$ if and only if $n=6$, $r=3$ and $H = (S_3 \wr S_2) \cap G$, which is listed as $({\rm L}_2(9),3,P_1)$ in Table \ref{tab:NGR}, or if the conditions in case (i)(d) of Theorem \ref{t:main1} are satisfied. In the latter case, we have $n = r$ and $H = {\rm AGL}_1(r) \cap G = N_G(R)$, as recorded in the first row of Table \ref{tab:NGR} (note that the special case $r=5$ appears as $({\rm L}_2(5),5,P_1)$). 

The sporadic groups are entirely straightforward and the relevant cases can be read off from Table \ref{tab:spor}. So for the remainder, we may assume $T$ is a group of Lie type over $\mathbb{F}_q$, where $q = p^f$ and $p$ is a prime. As before, let $r_i$ be a primitive prime divisor of $q^i-1$. 

First suppose $T$ is an exceptional group. If $r=2$, then $G=T = {}^2B_2(q)$ and  $H = N_G(R)$ is the only possibility, and this case is recorded in Table \ref{tab:NGR}. Now assume $r$ is odd. Here we inspect the cases arising in Table \ref{tab:main2}, considering each possibility for $T$ in turn. 

Suppose $T = {}^2B_2(q)$. Here $r = r_4$ is a primitive prime divisor of $q^4-1$ and $H = N_G(S)$, where $S$ is a cyclic maximal torus of $T$ of order $q\pm \sqrt{2q}+1$ (where the sign is chosen so that $r$ divides $|S|$). Then $R_0 \normeq H_0$, so $H = N_G(R_0)$ and this case is recorded in Table \ref{tab:NGR}. The same conclusion holds for both cases with $T = {}^2G_2(q)$ in Table \ref{tab:main2}, and so they also appear in Table \ref{tab:NGR}.

Next assume $T = G_2(q)$ with $q \geqs 3$. Here $H_0 = {\rm SL}_3^{\e}(q).2$ and $R_0$ is not normal in $H_0$, so $H \ne N_G(R_0)$ in these cases. Similarly, by inspecting the remaining cases in Table \ref{tab:main2}, we deduce that $H = N_G(R_0)$ if and only if $H = N_G(S)$, where $S$ is a cyclic maximal torus of $T$ whose order is divisible by $r$. All of these cases are listed in Table \ref{tab:NGR}.

Finally, let us assume $T$ is a classical group. First suppose $r=2$, so the possibilities for $(G,H)$ are given in Table \ref{tab:prime2}. 

First assume $T = {\rm L}_2(q)$ with $q \geqs 4$. Clearly, if $q$ is even and $H$ is a Borel subgroup, then $H = N_G(R_0)$. Similarly, if $q=5$ and $H$ is of type $2^{1+2}.{\rm O}_2^{-}(2)$, then $H_0 = A_4$ and the same conclusion holds (this case is recorded as $({\rm L}_2(4),2,P_1)$ in Table \ref{tab:NGR}). Next assume $H$ is of type ${\rm GL}_1(q) \wr S_2$, so $H_0 = D_{q-1}$, $q \geqs 9$ and $q \equiv 1 \imod{4}$. Here $R_0$ is normal in $H_0$ if and only if $R_0 = H_0$, so $q - 1 = 2^k$ for some $k$ and \cite[Lemma 2.6]{BTV} implies that either $q = 9$ or $q$ is a Fermat prime. We conclude that $\mathcal{M}(R) = \{ N_G(R_0) \}$ if and only if $q=9$ and $G \not\leqs {\rm P\Sigma L}_2(q)$, or if $q=p = 2^k+1 \geqs 17$ (note that in the latter case we have $|R_0| \geqs 2^4$). Similarly, if $H$ is of type ${\rm GL}_1(q^2)$, then $H_0 = D_{q+1}$ with $q = p \equiv 3 \imod{4}$ and we need $H_0 = R_0$. Therefore  $q+1 = 2^k$ for some $k$ and by appealing to \cite[Lemma 2.6]{BTV} we deduce that $q \geqs 7$ is a Mersenne prime. If $q=7$ then $|R_0| = 2^3$ and so we need $G = {\rm PGL}_2(7)$ (see Table \ref{tab:prime2}). On the other hand, if $q \geqs 31$ then $|R_0| \geqs 2^5$ and there are no additional conditions on $G$. These cases are recorded in Table \ref{tab:NGR}.

Continuing to assume $r=2$, it is easy to see that $R_0 \normeq H_0$ when $p=2$ and either $T = {\rm L}_3(q)$ or ${\rm PSp}_4(q)$ with $H$ of type $P_{1,2}$, or if $T = {\rm U}_3(q)$ and $H$ is of type $P_1$. All of these cases appear in Table \ref{tab:NGR}. And for the remaining possibilities in Table \ref{tab:prime2}, one checks that $R_0$ is not normal in $H_0$ and so none of these cases arise. This completes the proof of Corollary \ref{c:NGR} when $T$ is a classical group and $r=2$.

Finally, let us assume $T$ is classical and $r$ is odd. Here we inspect the triples 
$(G,r,H)$ in Table \ref{tab:main1}.

First assume $T = {\rm L}_2(q)$. Clearly, if $r = p$ and $H$ is of type $P_1$, then $H = N_G(R_0)$ and this case is recorded in Table \ref{tab:NGR}. Similarly, if $r = r_2$ and $H$ is of type ${\rm GL}_1(q^2)$, then $H = N_G(S)$ with $S$ a cyclic maximal torus of $T$ of order $(q+1)/2$, hence $R_0 \normeq H_0$ and so these cases also appear in Table \ref{tab:NGR}. 

Next suppose $T = {\rm L}_n(q)$ with $n \geqs 3$. If $n = r = r_1$ and $H$ is of type ${\rm GL}_1(q) \wr S_n$, then it is straightforward to check that $R_0$ is not normal in $H_0$. For example, if $n=3$ then $R_0$ is normal in $H_0$ if and only if $q-1$ is a $3$-power, which is not possible by \cite[Lemma 2.6]{BTV} (note that $q \ne 4$ since $f=3^a \geqs 1$, as indicated in Table \ref{tab:main1}). On the other hand, if $G = {\rm PGL}_3(4)$, $r=3$ and $H$ is of type ${\rm GU}_3(2)$, then $H_0 = {\rm U}_3(2) = 3^2{:}Q_8$ and $R_0 \normeq H_0$, so this case is recorded in Table \ref{tab:NGR}. To complete the proof for linear groups, we may assume $r = r_n$ and $H$ is a $\C_3$-subgroup of type ${\rm GL}_{n/t}(q^t)$, where $n = t^a$ and $t$ is an odd prime. Here $R_0$ is normal in $H_0$ if and only if $n = t$, in which case $H = N_G(S)$ is the normaliser of a Singer cycle $S$ in $T$. So the special case with $n=t$ is recorded in Table \ref{tab:NGR}.

Now assume $T = {\rm U}_n(q)$ with $n \geqs 3$. To begin with, suppose $n=3$, so $q \geqs 3$. Clearly, we have $H = N_G(R_0)$ when $r=p$ and $H = P_1$ is a Borel subgroup, so this case appears in Table \ref{tab:NGR}. If $r = r_2 = 3$ and $H$ is of type ${\rm GU}_1(q) \wr S_3$ then $R_0$ is normal in $H_0$ if and only if $q+1$ is a $3$-power, which is equivalent to the condition $q=8$ by \cite[Lemma 2.6]{BTV} (recall that $q \ne 2$). So the case $q=8$ is recorded in Table \ref{tab:NGR}. Similarly, if $r=r_6$ and $H$ is of type ${\rm GU}_1(q^3)$, then $R_0 \normeq H_0$ and $H = N_G(R_0)$. To complete the proof for unitary groups, we may assume $n \geqs 4$. By inspecting the cases in Table \ref{tab:main1}, we quickly deduce that $R_0$ is normal in $H_0$ if and only if $r = r_{2n}$ and $H$ is of type ${\rm GU}_1(q^n)$, where $n \geqs 5$ is a prime, and this latter case is included in Table \ref{tab:NGR}.

To complete the proof for classical groups, we may assume $T$ is a symplectic or orthogonal group. By inspecting the cases in Table \ref{tab:main1}, it is easy to see that $R_0$ is not normal in $H_0$ and so none of these cases appear in Table \ref{tab:NGR}. 
\end{proof}

Next we turn to Corollary \ref{c:NGR2}. As before, $G$ is an almost simple group with socle $T$ and $R$ is a Sylow $r$-subgroup of $G$. Recall that $R$ is weakly subnormal in $G$ if and only if $\mathcal{M}(R) = \{N_G(R)\}$, so we need to determine the cases in Table \ref{tab:NGR} with $N_G(R_0) = N_G(R)$, where $R_0 = R \cap T$. With this aim in mind, the following elementary lemma will be useful.

\begin{lem}\label{l:equiv}
Let $G$, $T$ and $R$ be defined as above and assume $\mathcal{M}(R) = \{H\}$, where $H=N_G(R_0)$ and $R_0 = R \cap T$. Then the following properties are equivalent:
\begin{itemize}\addtolength{\itemsep}{0.2\baselineskip}
\item[{\rm (i)}] $N_G(R_0)=N_G(R)$;
\item[{\rm (ii)}] $\mathcal{M}(O_r(H))=\{H\}$;
\item[{\rm (iii)}]  $G=O_r(H)T$;
\item[{\rm (iv)}] $[R,H_0]\leqs R_0$, where $H_0 = H \cap T$.
\end{itemize}
\end{lem}

\begin{proof}
Since $H=N_G(R_0)$, we have $R_0 \normeq H$ and hence $R_0\leqs O_r(H)$. In addition, $R_0 \leqs H_0$ and thus $R_0\leqs O_r(H)\cap H_0$. Since $T\normeq G=RT$, it follows that $H_0 \normeq H=RH_0$ and $R\cap H_0=R_0$.

If (i) holds, then $R\normeq H$ and thus $R=O_r(H)$ since $R$ is a Sylow $r$-subgroup of $H$, whence (ii) holds. 

Next assume (ii) holds and note that $O_r(H)\leqs O_r(H)T\leqs G$. Since  $T\not\leqs H$ and $H$ is the unique maximal subgroup of $G$ containing $O_r(H)$, we must have $G=O_r(H)T$ and thus (iii) is satisfied.

Now suppose (iii) holds. By Dedekind's modular law, we have $R=O_r(H)(R\cap T)=O_r(H)R_0$. As both $O_r(H)$ and $R_0$ are normal subgroups of $H$, we have $R=O_r(H)R_0\normeq H$ which implies that $R=O_r(H)$. And since $H_0\normeq  H=RH_0$, it follows that $[R,H_0]\leqs R\cap H_0=R_0$ as in (iv).

Finally, suppose (iv) holds. Since $R_0\normeq H=RH_0$, the hypothesis implies that $H_0/R_0$ centralises $R/R_0$ in $H/R_0$, whence $R/R_0\normeq H/R_0$ and thus $R\normeq H$. Hence (i) holds and the proof of the lemma is complete.
\end{proof}

\begin{proof}[Proof of Corollary \ref{c:NGR2}]
Let $G$ be an almost simple group with socle $T$ and let $R$ be a Sylow $r$-subgroup of $G$. Recall that it suffices to determine the cases in Table \ref{tab:NGR} with $H = N_G(R_0) = N_G(R)$, noting that the latter property clearly holds when $G = T$. So we may assume $G \ne T$ and we will work repeatedly with criterion (iv) in Lemma \ref{l:equiv}. In particular, the lemma implies that $N_G(R_0) \ne N_G(R)$ if $R$ acts nontrivially on $H_0/R_0$. In each case, $T$ is a simple group of Lie type over $\mathbb{F}_q$, where $q= p^f$ and $p$ is a prime.

We begin by considering the groups with socle $T = {\rm L}_2(q)$. Set $d = (2,q-1)$. First assume $r = p$ and $H$ is a Borel subgroup, so $H_0/R_0$ is a cyclic maximal torus of $T$ of order $(q-1)/d$. Now $R$ contains a field automorphism $x$ of order $r$ and the centraliser of $x$ on $H_0/R_0$ has order $(q_0-1)/d < (q-1)/d$, where $q=q_0^r$. So by appealing to Lemma \ref{l:equiv} we conclude that $N_G(R_0) \ne N_G(R)$. Clearly, if $r=2$ and $q \geqs 7$ is a Mersenne or Fermat prime, then $G = {\rm PGL}_2(q)$ and $N_G(R_0) = R$, so $N_G(R_0) = N_G(R)$ and this case is recorded in part (ii) of the corollary. The same conclusion holds if $q=9$ and $G \ne {\rm P\Sigma L}_2(9)$ as in part (iii) of Corollary \ref{c:NGR2}, noting that $|\mathcal{M}(R)| > 1$ when $G = {\rm P\Sigma L}_2(9)$.

Finally, suppose $T = {\rm L}_2(q)$ and $r = r_2$ is a primitive prime divisor of $q^2-1$. Note that $r \geqs 3$ and thus $G = T.r^a$ for some $a \geqs 1$, where $r^a$ divides $f$. It is straightforward to check that $N_G(R_0) = N_G(R)$ when $G = {\rm L}_2(8).3$, so we may assume $q>8$. In particular, this means that we can find a primitive prime divisor $s$ of $p^{2f}-1$ and we note that $s \geqs 2f+1 > r$. Now $H_0 = D_{2(q+1)/d}$ and $R$ contains a field automorphism $x$ of order $r$. Since $s$ does not divide $2(q_0+1)/d$, where $q=q_0^r$, it follows that $x$ does not centralise the subgroup $D_{2s}$ of $H_0/R_0$ and we deduce that $N_G(R_0) \ne N_G(R)$.

Next assume $T = {\rm L}_3(q)$. The special case $G = {\rm PGL}_3(4)$ with $r=3$ can be checked directly. Next assume $r=p=2$ and $G \not\leqs {\rm P\Gamma L}_3(q)$, so $R$ contains an involutory graph or graph-field automorphism. For $q=2$, it is straightforward to check that $N_G(R_0) = N_G(R)$, so $G = {\rm L}_3(2).2$ is recorded in part (iv) of the corollary. Similarly, if $q=4$ then one checks that $N_G(R_0) = N_G(R)$ if and only if $G = {\rm L}_3(4).2_3$ and so this example also appears in part (iv) (note that $G$ in this case contains involutory graph automorphisms). Now assume $q \geqs 8$ and note that $H_0/R_0$ is the image modulo scalars of a split maximal torus $C_{q-1} \times C_{q-1}$ of ${\rm SL}_3(q)$. Suppose $R$ contains an involutory graph automorphism $x$. Then $x$ inverts the torus $C_{q-1} \times C_{q-1}$ and so it does not centralise $H_0/R_0$. Finally, suppose $x \in R$ is an involutory graph-field automorphism, so $q=q_0^2$ and we may assume $q_0 \geqs 4$. Let $s$ 
be a prime divisor of $q_0-1$. The centraliser of $x$ in the torus $C_{q-1} \times C_{q-1}$ has order $(q_0+1)^2$, which is indivisible by $s$, so $x$ acts nontrivially on $H_0/R_0$ and thus $N_G(R_0) \ne N_G(R)$. 

Finally, suppose $T = {\rm L}_3(q)$ and $r = r_3$, so $H$ is of type ${\rm GL}_1(q^3)$. Here $G = T.r^a$ and $H_0 = C_{(q^2+q+1)/d}.3$, where $r^a$ divides $f$ and we set $d = (3,q-1)$. By working with a primitive prime divisor of $p^{3f}-1$ we can argue as above to show that $N_G(R_0) \ne N_G(R)$ (see the case $T = {\rm L}_2(q)$ with $r=r_2$). 

Now suppose $T = {\rm U}_3(q)$ and set $d = (3,q+1)$. First assume $r=p$ and $H$ is of type $P_1$, so $H_0/R_0 = C_{(q^2-1)/d}$. If $p=2$ then $R$ must contain an involutory graph automorphism $x$, which is induced by the order two field automorphism of $\mathbb{F}_{q^2}$. Therefore, the centraliser of $x$ on $H_0/R_0$ has order $(q-1)/d$ and we conclude that $N_G(R_0) \ne N_G(R)$. An entirely similar argument applies when $p$ is odd, in which case $R$ contains a field automorphism of order $r$. The special cases $G = {\rm U}_3(8).3$ and ${\rm U}_3(8).3^2$ with $r=3$ can be checked directly (there are three groups of the form ${\rm U}_3(8).3$, up to isomorphism, and one checks that $N_G(R_0) = N_G(R)$ in every case). Finally, the case where  $r=r_6$ is entirely similar to the analogous case with $T = {\rm L}_3(q)$ and $r=r_3$ handled above, this time working with a primitive prime divisor of $p^{6f}-1$. We omit the details.

Next consider the case where $T = {\rm PSp}_4(q)$ and $r=p=2$, with $q \geqs 4$ and $H_0/R_0 = C_{q-1} \times C_{q-1}$. Write $q=2^f$ and note that ${\rm Out}(T) = C_{2f}$ is cyclic. Suppose $f$ is odd, so $G = T.2$ and $R$ contains an involutory graph automorphism $x$. Since $x$ interchanges long and short root subgroups, we deduce that $x$ does not centralise the torus $H_0/R_0$. Similarly, if $f$ is even and $q=q_0^2$ then $R$ contains an involutory field automorphism with centraliser $C_{q_0-1} \times C_{q_0-1}$ on $H_0/R_0$. So in both cases, by appealing to Lemma \ref{l:equiv}, we conclude that $N_G(R_0) \ne N_G(R)$.

To complete the proof for classical groups, we may assume $n \geqs 5$ is a prime and either $T = {\rm L}_n(q)$ and $r = r_n$, or $T = {\rm U}_n(q)$ with $r=r_{2n}$. In both cases, we have $G = T.r^a$, where $r^a$ divides $f$, and we can argue as above, working with primitive prime divisors of $p^{nf}-1$ and $p^{2nf}-1$, respectively.

Finally, let us turn to the cases where $T$ is an exceptional group of Lie type. The groups with $T = {}^3D_4(q)$ or $E_8(q)$ are straightforward; in both cases, $R$ contains a field automorphism of order $r$, which acts nontrivially on $H_0/R_0$ due to the existence of a suitable primitive prime divisor (e.g. a primitive prime divisor of $p^{12f}-1$ for $T = {}^3D_4(q)$). The case $T = {}^2G_2(q)$ with $H_0/R_0 = C_{q-1}$ is also entirely straightforward. So it remains to deal with the following cases:
\begin{itemize}\addtolength{\itemsep}{0.2\baselineskip}
\item[{\rm (a)}] $T = {}^2B_2(q)$, $r=r_4$, $H_0 = (q \pm \sqrt{2q}+1){:}4$
\item[{\rm (b)}] $T = {}^2G_2(q)$, $r=r_6$, $H_0 = (q \pm \sqrt{3q}+1){:}6$
\item[{\rm (c)}] $T = {}^2F_4(q)$, $r=r_{12}$, $H_0 = (q^2 \pm \sqrt{2q^3}+q\pm \sqrt{2q}+1){:}12$
\end{itemize}
In each case, $G = T.r^a$ and $r^a$ divides $f$, so $R$ contains a field automorphism $x$ of order $r$. In particular, $q = q_0^r$. Given an integer $n$, we will write $n_r$ to denote the $r$-part of $n$.

Consider case (a). Fix $\e = \pm$ so that $r$ divides $q + \e\sqrt{2q} + 1$ and note that $(q^2+1)_r = (q+\e \sqrt{2q}+1)_r$ since $(q+\sqrt{2q}+1,q-\sqrt{2q}+1) = 1$. Write $(q_0^2+1)_r = r^b \geqs 1$ and observe that $(q^2+1)_r = r^{b+1}$. If we write $H_0 = S{:}4$, then $C_S(x)$ is a cyclic group of order $q_0 + \e'\sqrt{2q_0}+1$ for some $\e' = \pm$ and it suffices to show that 
\[
A = \frac{q+\e\sqrt{2q}+1}{q_0+\e'\sqrt{2q_0}+1}
\]
is an $r$-group if and only if $(q_0,r) = (2,5)$ and $\e = \e' = -$. Indeed, if $A$ is divisible by a prime $s \ne r$, then $x$ acts nontrivially on $H_0/R_0$ and thus $N_G(R_0) \ne N_G(R)$ by Lemma \ref{l:equiv}. There are two cases to consider:
\begin{itemize}\addtolength{\itemsep}{0.2\baselineskip}
\item[{\rm (I)}] $(q_0^2+1)_r = (q_0 + \e'\sqrt{2q_0} + 1)_r$
\item[{\rm (II)}] $(q_0^2+1)_r = (q_0 - \e'\sqrt{2q_0} + 1)_r$
\end{itemize}

First consider (I). Here $A_r = r$ and so it suffices to show that $A>r$. Write $q_0 = 2^k$ (with $k \geqs 1$ odd) and set
\[
g(k) = \frac{2^{kr} + \e 2^{(kr+1)/2} + 1}{2^k + \e' 2^{(k+1)/2} + 1}.
\]
It is easy to check that $g$ is increasing as a function of $k$, and it is also easy to see that $g(1)>r$. Hence $A>r$ and the result follows.

Now let us turn to (II), so $(q_0 + \e'\sqrt{2q_0} + 1)_r=1$ and $A_r = r^{b+1}$. If $q_0 - \e'\sqrt{2q_0} + 1$ is divisible by a prime $s \ne r$ then $s$ divides $A$ and the result follows. Now assume $q_0 - \e'\sqrt{2q_0} + 1 = r^b = (q_0^2+1)_r$. Then $A$ is an $r$-power if and only if 
\[
\frac{q+\e \sqrt{2q}+1}{q_0^2+1} = r.
\]
It is a straightforward exercise to show that this holds if and only if $q_0 = 2$, $r=5$ and $\e = -$, in which case $G = {}^2B_2(32).5$ and one can check that $N_G(R_0) = N_G(R)$. This is the special case recorded in part (vi) of Corollary \ref{c:NGR2}.

An entirely similar argument applies in case (b), noting that $r = r_6 \geqs 7$, 
\[
(q_0+\sqrt{3q_0}+1)(q_0-\sqrt{3q_0}+1)= q_0^2-q_0+1
\]
and the two factors on the left are coprime. And similarly in (c), where $r = r_{12} \geqs 13$ and 
\[
(q_0^2+\sqrt{2q_0^3}+q_0+\sqrt{2q_0}+1)(q_0^2-\sqrt{2q_0^3}+q_0-\sqrt{2q_0}+1) = q_0^4-q_0^2+1.
\]
We leave the reader to check that $N_G(R_0) \ne N_G(R)$ in these cases.
\end{proof}

Next we turn to the proof of Corollary \ref{c:Or}. 

\begin{proof}[Proof of Corollary \ref{c:Or}]
Let $G$ be an almost simple group with socle $T$ and $\mathcal{M}(R) = \{H\}$, where $R$ is a Sylow $r$-subgroup of $G$. Set $H_0=H\cap T$ and $R_0 = R \cap T$, noting that $R_0$ is a Sylow $r$-subgroup of $H_0$ and we have $R_0 \normeq  R$ and $H_0 \normeq H$. By Lemma \ref{l:trivial} we have $R_0 \ne 1$.

Our goal is to show that $O_r(H) \ne 1$ if and only if  $(G,r,H)$ is one of the cases in Table \ref{tab:NGR} or \ref{tab:Or}. Notice that if $O_r(H_0)\ne 1$, then since $H_0\normeq H$ and $O_r(H_0)$ is characteristic in $H_0$, we have $O_r(H_0)\normeq H$ and thus $O_r(H) \ne 1$. In particular, $O_r(H) \ne 1$ for every case in Table \ref{tab:NGR}. 

We now divide the proof into two cases, according to the parity of $r$.  

\vs

\noindent \emph{Case 1. $r=2$.}

\vs

First assume $r=2$, noting that the possibilities for $(G,H)$ are determined in Theorem \ref{t:prime2}. If $T = A_n$ then $n=2^k+1$ with $k \geqs 2$ and $H=A_{n-1}$ or $S_{n-1}$. Clearly, $O_2(H) \ne 1$ if and only if $n = 5$, noting that this case is  recorded in Table \ref{tab:NGR} as $({\rm L}_2(4), 2, P_1)$. Similarly, if $G = {}^2B_2(q)$ and $H = N_G(R)$ then $O_2(H) \neq 1$ and so this case also appears in Table \ref{tab:NGR}.

To complete the proof for $r=2$, let us assume $T$ is a classical group, so the possibilities for $(G,H)$ are listed in Table \ref{tab:prime2}. We claim that $O_2(H_0) \ne 1$ in every case, noting that this property has already been established for the cases appearing in Table \ref{tab:NGR}. If $T = {\rm L}_2(q)$ and $H$ is of type ${\rm GL}_1(q) \wr S_2$ or ${\rm GL}_1(q^2)$ then $H_0 = D_{q \pm 1}$ and the additional congruence condition on $q$ implies that $Z(H_0)$ has order $2$, whence $O_2(H_0) \ne 1$. Next suppose $T = {\rm L}_n^{\e}(q)$, $n \geqs 3$ and $H$ is of type ${\rm GL}_{n-1}^{\e}(q) \times {\rm GL}_1^{\e}(q)$. Here $n$ is odd and $H_0 = C_T(z)$ for some involution $z \in T$, so $O_2(H_0) \ne 1$ as required. Similarly, if $T$ is an orthogonal group, then once again $H_0$ contains a central involution (see \cite[Tables B.8. B.10 and B.11]{BG}, for example) and thus $O_2(H_0) \ne 1$. The reader can check that all of the cases in Table \ref{tab:prime2} that do not appear in Table \ref{tab:NGR} have been included in Table \ref{tab:Or}.
 
\vs

\noindent \emph{Case 2. $r>2$.}

\vs

Let $r$ be an odd prime and recall that $O_r(H) \ne 1$ if $O_r(H_0) \ne 1$. We claim that the converse holds. To see this, suppose that $O_r(H_0)=1$ and $O_r(H) \ne 1$. Then $O_r(H)\cap H_0=1$ and $[O_r(H),H_0] \leqs O_r(H)\cap H_0=1$, so $O_r(H)$ centralises $R_0 \neq 1$. In other words, any nontrivial element of $O_r(H)$ induces an automorphism of $T$ 
which centralises a Sylow $r$-subgroup of $T$ and has order a power of $r$.   However, this contradicts \cite[Theorem A]{Gross}. This justifies the claim and so it suffices to determine all the triples $(G,r,H)$ arising in Theorem \ref{t:main1} with $O_r(H_0)\neq 1$.
 
First assume $G = T = A_n$ is an alternating group, so the relevant cases are listed in parts (a)-(d) of Theorem \ref{t:main1}(i). Here it is easy to check that $O_r(H_0)\neq 1$ in case (b) with $r=3$, in case (d), and also in case (c) when $r=3$. The first two possibilities appear in Table \ref{tab:NGR} (note that case (b) is listed as $({\rm L}_2(9),3,P_1)$), while the latter triple $(A_9, 3, (S_3\wr S_3)\cap A_9)$ has been included in Table \ref{tab:Or}. The case where $T$ is a sporadic group is easy and each possibility for $(G,r,H)$ is recorded in Table \ref{tab:NGR}.

For the remainder, we may assume $T$ is a simple group of Lie type over $\mathbb{F}_q$, where $q = p^f$ with $p$ a prime. As usual, $r_i$ denotes a primitive prime divisor of $q^i-1$.

Suppose $T$ is a classical group, so the possibilities for $(G,r,H)$ are given in  Table \ref{tab:main1}. First assume $T = {\rm L}_n(q)$. If $n=2$ then both cases in Table \ref{tab:main1} are included in Table \ref{tab:NGR}, so we may assume $n \geqs 3$.
The special case $G={\rm PGL}_3(4)$ with $r=3$ is in Table \ref{tab:NGR}. If $r = r_1$ and $H$ is of type ${\rm GL}_1(q) \wr S_n$, then it is clear that $O_r(H_0)\neq 1$ and so this case appears in Table \ref{tab:Or}. Finally, if $r=r_n$ and $H$ is of type ${\rm GL}_{n/t}(q^t)$, then one checks that $O_r(H) \neq 1$ if and only if $n=t$, in which case the triple $(G,r,H)$ is included in Table \ref{tab:NGR}.

Now assume $T={\rm U}_n(q)$ with $n \geqs 3$. If $n=3$, then by inspecting Table \ref{tab:main1} we see that $O_r(H_0)\neq 1$ unless $G={\rm U}_3(3)$, $r=7$ and $H = {\rm L}_2(7)$. In particular, all of the cases with $O_r(H) \ne 1$ are in Tables \ref{tab:NGR} or \ref{tab:Or}. Now assume $n\geqs 4$. Here $O_r(H_0)\neq 1$ if and only if $r=r_2$ and $H$ is of type ${\rm GU}_1(q)\wr S_n$, or if $r=r_{2n}$ and $H$ is of type ${\rm GU}_1(q^n)$. The latter case appears in Table \ref{tab:NGR}, while the former has been included in Table \ref{tab:Or}. Finally, it is easy to see that $O_r(H_0)=1$ for the cases in Table \ref{tab:main1} with $T$ a symplectic or orthogonal group. 

Finally, let us assume $T$ is an exceptional group of Lie type, so the possibilities for $(G,r,H)$ are recorded in Table \ref{tab:main2}. First note that all the cases with $T = {}^2B_2(q)$, ${}^2G_2(q)$, ${}^2F_4(q)'$ or $E_8(q)$ appear in Table \ref{tab:NGR} and it is easy to see that $O_r(H_0)=1$ when $T=F_4(q)$, $E_6^{\e}(q)$ or $E_7(q)$. If $T = G_2(q)$ and $H_0 = {\rm SL}_3^{\e}(q).2$, then $O_r(H_0) \ne 1$ if and only if $r = r_{(3-\e)/2} = 3$, so these cases are recorded in Table \ref{tab:Or}.
Finally, suppose $T = {}^3D_4(q)$. The case where $H$ is of type $q^4-q^2+1$ is recorded in Table \ref{tab:NGR}. Now, if $r = r_{(3-\e)/2} = 3$ with $\e = \pm$ then $r$ divides $q^2+\e q +1$ and thus $O_r(H_0) \ne 1$ when $H$ is of type $A_2^{\e}(q) \times (q^2+\e q +1)$. Therefore, the latter case is included in Table \ref{tab:Or}.
\end{proof}

Next let us consider Corollary \ref{c:OrH}. As before, $G$ is an almost simple group with socle $T$ and $\mathcal{M}(R)=\{H\}$, where $R$ is a Sylow $r$-subgroup of $G$, and we need to determine the cases with $\mathcal{M}(O_r(H))=\{H\}$. Clearly, this holds only if $O_r(H) \ne 1$, so the possibilities are listed in Tables \ref{tab:NGR} and \ref{tab:Or}. If $(G,r,H)$ is one of the cases in Table \ref{tab:NGR}, then Lemma \ref{l:equiv} implies that $\mathcal{M}(R)=\{H\}$ if and only if $R$ is weakly subnormal in $G$, as in part (i) of Corollary \ref{c:OrH}, so we only need to consider the cases in Table \ref{tab:Or}. 

\begin{proof}[Proof of Corollary \ref{c:OrH}]
As explained above, we need to inspect the triples $(G,r,H)$ in Table \ref{tab:Or}. In each case we have $O_r(H) \ne 1$ and $H \ne N_G(R_0)$, so $R$ is not normal in $H$. Let us also note that $O_r(H)\leqs O_r(H)T \leqs G$ and $T\not\leqs H$, so $\mathcal{M}(O_r(H)) = \{H\}$ only if  $G=O_r(H)T$. We consider each case separately.

First assume $T=A_9$, $r=3$ and $H$ is of type $S_3\wr S_3$, so $G=T$ and $O_r(H)$ is elementary abelian of order $27$. Here it is easy to see that $\mathcal{M}(O_r(H)) \ne \{H\}$. For example, $O_r(H)$ is contained in an intransitive maximal subgroup $(S_6 \times S_3) \cap G$.

In each of the remaining cases, $T$ is a simple group of Lie type over $\mathbb{F}_q$, where $q=p^f$ and $p$ is a prime.   

Suppose $T={\rm L}_2(q)$, $r =2$, $q$ is odd and $H$ is of type ${\rm GL}_1(q)\wr S_2$, in which case $H \leqs N_G(S)$ where $S = C_{q-1}$ is a split maximal torus of ${\rm PGL}_2(q)$. Write $S=O(S) \times O_2(S)$ where $O(S)=O_{2'}(S)$.  Note that $O_2(S) \leqs O_2(H)$, $[O(S),O_2(H)] = 1$ and $O(S) \ne 1$.   Let $S_0 = S \cap T$. 

If $ O_2(H) \leqs S$, then $|\mathcal{M}(O_2(H))| \geqs 2$ since $O_2(H)$ is contained in a Borel subgroup of $G$. So let us assume otherwise. Then $q = q_0^2$, $G \not\leqs {\rm PGL}_2(q)$ and we may assume $G = O_2(H)T$ and $H$ contains a nontrivial field automorphism $x$ normalising $S$. Then $x$ either centralises or inverts $O(S)$, so $x$ has order $2$ and there are two cases to consider.

If $x$ centralises $O(S)$, then $|O(S)|$ divides $q_0 - 1$, whence $O_2(H) \leqs \langle S, x \rangle$ is contained in a Borel subgroup and $|\mathcal{M}(O_2(H))| \geqs 2$. 

Now assume $x$ inverts $O(S)$. Then $xy$ centralises $O(S)$, where $y \in T$ is an element inverting $S$, and we may choose $y \in C_T(x)$. It follows that $|O(S)|$ divides $q_0 + 1$, which in turn implies that $q_0-1$ is a power of $2$ and thus $q_0 = 9$ or $q_0=p \geqs 5$ is a Fermat prime (as stated in Table \ref{tab:Or}, we have $q \geqs 13$ and thus $q_0 \ne 3$). Since $G=O_2(H)T$,  we see that $G \leqs  {\rm PGL}_2(q).\langle x \rangle$ and $O_2(H) \leqs \langle O_2(S), x \rangle$. Also note that the condition $G \not\leqs {\rm P\Sigma L}_2(q)$ in Table \ref{tab:Or} rules out $G = \langle T, x \rangle$. The case $q_0 = 9$ can be handled using {\sc Magma} and we deduce that $\mathcal{M}(O_2(H)) = \{H\}$ if and only if $G = T.2_3$ (nonsplit extension) or $T.2^2 = {\rm PGL}_2(81).\la x \ra$. Finally, if $q_0 =p \geqs 5$ is a Fermat prime, then there are only two possibilities for $G$, namely the nonsplit extension $T.2_3$ and ${\rm Aut}(T) = T.2^2$. In both cases, it is straightforward to check that $H$ is the unique maximal subgroup containing  $O_2(H)$. Indeed, the element $xy \in G$ interchanges the two Borel subgroups containing $S$, while $xy$ does not normalise any subfield subgroup containing $O_2(S_0)$. These two groups are recorded in part (ii) of the corollary.   

Next consider the case $T={\rm L}_2(q)$ with $r =2$ and $H$ of type ${\rm GL}_1(q^2)$. Here $q = p \equiv 3 \imod{4}$ and $H = N_G(S)$, where $S = O(S) \times O_2(S)$ is a nonsplit maximal torus of $G \cap {\rm PGL}_2(q)$. Since $G = T$ or ${\rm PGL}_2(q)$, we see that $O(S) \ne 1$ and $O_2(H) = O_2(S)$, so by inspection of the maximal subgroups of $G$ (see \cite[Table 8.1]{BHR}, for example) we deduce that $\mathcal{M}(O_2(H)) = \{H\}$. These cases as recorded in part (iii) of the corollary. 

We now turn to the remaining cases in Table \ref{tab:Or}. Suppose $T = {\rm L}_n(q)$ with $n \geqs 3$. First assume $r=2$ and $H$ is of type ${\rm GL}_{n-1}(q) \times {\rm GL}_1(q)$, so $G = T.2$ contains a graph automorphism. If $(n,q) \ne (3,3)$ then $O_2(H)$ has order $2$ and is contained in $T$, whence $|\mathcal{M}(O_2(H))| \geqs 2$. On the other hand, if $n = q =3$ then it is straightforward to verify that $\mathcal{M}(O_2(H)) = \{H\}$ as in part (iv). Next suppose $n = r = r_1$ and $H$ is of type ${\rm GL}_1(q) \wr S_n$. If $n \geqs 5$ then it is clear that $O_r(H)$ is contained in a maximal parabolic subgroup of $G$ and thus $|\mathcal{M}(O_r(H))| \geqs 2$. And the same conclusion holds when $n=3$ since the conditions in Table \ref{tab:Or} imply that $q \ne 4$ and thus $q-1$ is a not power of $3$.
 
 A very similar argument shows that $|\mathcal{M}(O_2(H))| \geqs 2$ for each case in Table \ref{tab:Or} with $T = {\rm U}_n(q)$.  

Next suppose $T = \Omega_n(q)$ with $n$ odd and let $V$ be the natural module. Here $r=2$ and $H$ is of type ${\rm O}_{n-1}^{+}(q) \times {\rm O}_1(q)$, which is the stabiliser in $G$ of a nondegenerate hyperplane of plus-type. For all $1 \leqs k \leqs n$, we note that $O_2(H)$ stabilises a nondegenerate $k$-dimensional subspace of $V$ and thus $O_2(H)$ is contained in several maximal subgroups.
Similarly if $T = {\rm P\Omega}^{\e}_n(q)$ with $n$ even, then $r=2$ and $O_2(H)$ stabilises a nondegenerate $k$-space for all even $2 \leqs k \leqs n$, hence the same conclusion holds. 

To complete the proof, we may assume $r=3$ and $T=G_2(q)$ or ${}^3D_4(q)$. First assume $T = G_2(q)$. Here $|O_3(H)|=3$ (since the centraliser of the derived subgroup of $H$ has order $3$), so $O_3(H) \leqs T$ and we may assume $G=T$.    It also follows that  $O_3(H) \leqs C_G(x)$, where $x \in H$ is an involution.
  For $p=2$ this implies that $O_3(H)$ is contained in a parabolic subgroup, whereas for $p$ odd we deduce that $O_3(H)$ is contained in a maximal rank subgroup of type $A_1(q)^2$. In both cases, the given overgroup of $O_3(H)$ is not contained in $H$ and thus $|\mathcal{M}(O_3(H))| \geqs 2$. 

Finally, let us assume $T ={}^3D_4(q)$ and $r=3$, in which case $H$ is the normaliser in $G$ of a cyclic subgroup of order $3$.  As in the previous case, we deduce that $G=T$ and $O_3(H)$ centralises an involution. For $p=2$, we conclude that $O_3(H)$ is contained in a parabolic subgroup, while it is contained in a subgroup of type $A_1(q^3)A_1(q)$ when $p$ is odd. Neither overgroup is contained in $H$ and thus $O_3(H)$ is contained in at least two maximal subgroups of $G$.

This exhausts all the possibilities in Table \ref{tab:Or} and completes the proof. 
\end{proof}

Finally, we provide a proof of Corollary \ref{c:max}. 

\begin{proof}[Proof of Corollary \ref{c:max}]
As in the corollary, let $G$ be a finite group, let $r$ be a prime divisor of $|G|$ with $O_r(G)=1$ and let $R$ be a Sylow $r$-subgroup of $G$.

First we need to show that $R$ is a maximal subgroup of $G$ if (i) or (ii) holds in the statement of the corollary. This is clear for (i), so let us consider (ii). 
Here $r=2$ and $G=ER$, where $E=E(G)=T_1\times T_2\times\cdots \times T_k$ is a minimal normal subgroup of $G$ and each $T_i \cong T$ is a nonabelian simple group. Moreover, $J=N_G(T_1)/C_G(T_1)$ is an almost simple group with a maximal Sylow $2$-subgroup and the possibilities for $J$ are as follows:
\begin{equation}\label{e:J}
{\rm PGL}_2(7), \; {\rm PGL}_2(9), \; {\rm M}_{10}, \; {\rm L}_{2}(9).2^2, \; {\rm L}_2(q),\; {\rm PGL}_2(q),
\end{equation}
where $q>7$ is a Mersenne or Fermat prime.

Clearly, if $k=1$ then $G=J$ is almost simple and $R$ is a maximal subgroup. Now assume $k \geqs 2$, noting that $R$ acts transitively on the set $\{T_1,T_2,\ldots, T_k\}$.

First we claim that $N_R(T_1)$ is a Sylow $2$-subgroup of $N_G(T_1)$, which then implies that 
\[
N_R(T_1)C_G(T_1)/C_G(T_1) \cong N_R(T_1)/C_R(T_1)
\]
is a Sylow $2$-subgroup of $J$. Since $E \normeq G=ER$ and $T_1\normeq E$,  it follows that $E \cap R$ is a Sylow $2$-subgroup of $E$ and we have $E \leqs N_G(T_1)$. By Dedekind's modular law, $N_G(T_1)=E(N_G(T_1)\cap R)=EN_R(T_1)$. Moreover, $E \cap N_R(T_1)=E \cap R\cap N_G(T_1)=E \cap R$. Hence
\[
|N_G(T_1)|=|EN_R(T_1)|=|E||N_R(T_1)|/|E \cap R|=|N_R(T_1)||E|_{2'}
\]
and we deduce that $N_R(T_1)$ is a Sylow $2$-subgroup of $N_G(T_1)$.

In order to conclude that $R$ is a maximal subgroup of $G$, it suffices to show that $T_1\cap R$ is a maximal proper $N_R(T_1)$-invariant subgroup of $T_1$ containing $T_1\cap R$. Indeed, if this holds then Lemma \ref{l:simpleprod} implies that $N\cap R$ is a maximal $R$-invariant subgroup of $N$ containing $N\cap R$ and we deduce that $R$ is maximal via Lemma \ref{l:normal}.  
 
To do this, first observe that $N_G(T_1)=EN_R(T_1)=T_1C_G(T_1)N_R(T_1)$ since we have $E \leqs T_1C_G(T_1)\leqs N_G(T_1)$. Also note that $T_1\cap C_G(T_1)=1$, where both $T_1$ and $C_G(T_1)$ are normal subgroups of $N_G(T_1)$. Hence $T_1C_G(T_1)\cong T_1\times C_G(T_1)$ and 
\[
N_R(T_1)\cap (T_1C_G(T_1))=(N_R(T_1)\cap T_1)(N_R(T_1)\cap C_G(T_1))=(R\cap T_1)C_R(T_1).
\]
We claim that $T_1\cap N_R(T_1)C_G(T_1)=T_1\cap R$. Clearly 
\[
T_1\cap R\leq T_1\cap N_G(T_1)\cap R\leq T_1\cap N_R(T_1)\leq T_1\cap N_R(T_1)C_G(T_1).
\]
For the reverse inclusion, let $x\in T_1\cap N_R(T_1)C_G(T_1)$. Then $x=ab\in T_1$ with $a\in N_R(T_1)$ and $b\in C_G(T_1)$.
Hence $a=xb^{-1}\in T_1C_G(T_1)\cap N_R(T_1)=(R\cap T_1)C_R(T_1)$ which implies that $xb^{-1}=uv$ for some $u \in R\cap T_1$ and $v \in C_R(T_1)$. Thus $u^{-1}x=vb\in T_1\cap C_G(T_1)=1$ and so $x=u\in R\cap T_1$. Therefore $T_1\cap N_R(T_1)C_G(T_1)=T_1\cap R$.

Let $T_1\cap R\leqs U\leqs T_1$ be an $N_R(T_1)$-invariant subgroup of $T_1$. 
Since $N_R(T_1)C_G(T_1)/C_G(T_1)$ is a maximal subgroup of $N_G(T_1)/C_G(T_1)$, it follows that $N_R(T_1)C_G(T_1)$ is a maximal subgroup of $N_G(T_1)$. As $U$ is $N_R(T_1)$-invariant, $UN_R(T_1)$ is a subgroup of $N_G(T_1)$ and hence 
\[
C_G(T_1)N_R(T_1)\leqs C_G(T_1)N_R(T_1)U\leqs N_G(T_1).
\]
It follows that $C_G(T_1)N_R(T_1)U$ is either $C_G(T_1)N_R(T_1)$ or $N_G(T_1)$, so there are two cases to consider.

First assume $C_G(T_1)N_R(T_1)U = C_G(T_1)N_R(T_1)$. Since $U\leqs T_1$, we have 
\[
T_1\cap (C_G(T_1)N_R(T_1))=T_1\cap (C_G(T_1)N_R(T_1)U)=U(T_1\cap C_G(T_1)N_R(T_1))
\]
and this implies that $U\leqs T_1\cap C_G(T_1)N_R(T_1)=T_1\cap R$. Hence $U=T_1\cap R$. Now assume $C_G(T_1)N_R(T_1)U=N_G(T_1)$. 
Since $T_1\cap R\leqs U$ and $T_1\cap R\leqs N_R(T_1)$, we have 
\[
T_1\cap R\leqs U\cap C_G(T_1)N_R(T_1)\leqs T_1\cap C_G(T_1)N_R(T_1)=T_1\cap R,
\]
which implies that 
\[
U\cap C_G(T_1)N_R(T_1)=T_1\cap C_G(T_1)N_R(T_1)=T_1\cap R.
\]
Since $C_G(T_1)N_R(T_1)U=C_G(T_1)N_R(T_1)T_1$, we deduce that $|U|=|T_1|$ and thus $U=T_1$.

This proves that $T_1 \cap R$ is a maximal proper $N_R(T_1)$-invariant subgroup of $T_1$ containing $T_1\cap R$, which implies that $R$ is a maximal subgroup of $G$.

\vs

For the remainder of the proof, we may assume $R$ is maximal and we need to show that the conditions in part (i) or (ii) of the corollary are satisfied. As explained in the paragraph preceding the corollary, we have $O_r(G) = \Phi(G) = 1$ and thus Theorem \ref{t:main2} applies. If part (i) of Theorem \ref{t:main2} holds, then we are in case (i) of the corollary, so we just need to consider parts (ii) and (iii) of Theorem \ref{t:main2}.

First consider part (ii) of Theorem \ref{t:main2}, so $G$ is almost simple and $\mathcal{M}(R)=\{R\}$, which means that we can read off $(G,r,H)$ from Table \ref{tab:NGR}. In this way, we deduce that $T={\rm L}_2(q)$, $r=2$ and either
\begin{itemize}\addtolength{\itemsep}{0.2\baselineskip}
\item[{\rm (a)}] $H$ is of type ${\rm GL}_1(q)\wr S_2$ and either $q=9$ and $G \not\leqs {\rm P\Sigma L}_2(9)$, or $q=p \geqs 17$ is a Fermat prime; or
\item[{\rm (b)}] $H$ is of type ${\rm GL}_1(q^2)$ and either $G = {\rm PGL}_2(7)$ or $q=p \geqs 31$ is a Mersenne prime.
\end{itemize}
It follows that $G$ is isomorphic to one of the groups in \eqref{e:J}. In particular, part (ii) of the corollary holds with $k=1$.

Finally, let us assume part (iii) of Theorem \ref{t:main2} is satisfied, so $G=ER$ and $E = E(G)$ is a direct product of $k \geqs 2$ nonabelian simple groups. Let $N$ be a minimal normal subgroup of $G$ and note that $R<NR \leqs G$. By the maximality of $R$, we have $G=NR$ and thus $E = N = T_1\times T_2\times\cdots\times T_k$, where each $T_i\cong T$ for some nonabelian simple group $T$.
Then $Q=N_R(T_1)C_G(T_1)/C_G(T_1)\cong N_R(T_1)/C_R(T_1)$ is a Sylow $r$-subgroup of the almost simple group $J=N_G(T_1)/C_G(T_1)$, and $Q$ is contained in a unique maximal subgroup of $J$, say $M/C_G(T_1)$, where $C_G(T_1)N_R(T_1)\leqs M\leqs N_G(T_1)$.   Since $N_G(T_1)=EN_R(T_1)=C_G(T_1)N_R(T_1)T_1$, we have $M=C_G(T_1)N_R(T_1)(M\cap T_1)$. Clearly $R\cap T_1= N_R(T_1)\cap T_1\leqs M\cap T_1\leqs T_1$ and $M\cap T_1$ is $N_R(T_1)$-invariant. Also note that 
$M\cap T_1\neq T_1$ since $M\neq N_G(T_1)$.

Since $R$ is a maximal subgroup of $G=ER$, Lemma \ref{l:normal} implies that  $E\cap R$ is a maximal $R$-invariant subgroup of $E$ containing $E\cap R$.  As $R$ permutes the $T_i$ transitively and $R\cap T_1$ is nontrivial (since $r$ divides $|T_1|$), Lemma \ref{l:simpleprod} (and its proof) implies that $E \cap R\cap T_1=R\cap T_1$ is a maximal $N_R(T_1)$-invariant subgroup of $T_1$ containing $R\cap T_1$. It follows that $M \cap T_1=R\cap T_1$ and hence $M=C_G(T_1)N_R(T_1)(R\cap T_1)=C_G(T_1)N_R(T_1)$. This implies that $M/C_G(T_1)=N_R(T_1)C_G(T_1)/C_G(T_1) = Q$ is a maximal subgroup of $J$, hence $J$ has a maximal Sylow $r$-subgroup and by arguing as above we conclude that $r=2$ and $J$ is one of the almost simple groups in \eqref{e:J}. The proof of the corollary is complete.
\end{proof}

\section{The tables}\label{s:tables}

In this final section we present the tables referred to in the statements of  Theorems \ref{t:prime2} and \ref{t:main1}, as well as Corollaries \ref{c:NGR} and \ref{c:Or}.

Before presenting the tables, we first fix our notation. In all cases, $G$ is an almost simple group with socle $T$ such that $G/T$ is an $r$-group, where $r$ is a prime divisor of $|G|$. In addition, $R$ is a Sylow $r$-subgroup of $G$ and we set $R_0 = R \cap T$. In the cases where $T$ is a group of Lie type, we write $q = p^f$ with $p$ a prime and we use $r_i$ to denote a primitive prime divisor of $q^i-1$. 

\begin{rem}\label{r:prime2}
In the final row of Table \ref{tab:prime2}, we write $T.\la \varphi \ra = T.C_{2f}$ for the subgroup of ${\rm Aut}(T)$ generated by the inner, field and graph-field automorphisms of $T = {\rm P\O}_{n}^{-}(q)$. The same notation is also used in Table \ref{tab:Or}.
\end{rem}

\begin{rem}\label{r:main1}
Let us record some comments on the cases arising in Tables \ref{tab:main1} and \ref{tab:main2}.
\begin{itemize}\addtolength{\itemsep}{0.2\baselineskip}
\item[{\rm (a)}] We have excluded certain groups due to the existence of isomorphisms. Specifically, we exclude
\[
{\rm L}_2(4), \; {\rm L}_3(2), \; {\rm PSp}_4(2)', \; G_2(2)', \; {}^2G_2(3)'
\]
since we have isomorphisms
\[
{\rm L}_2(4) \cong {\rm L}_2(5), \; {\rm L}_3(2) \cong {\rm L}_2(7), \; {\rm PSp}_4(2)' \cong {\rm L}_2(9),
\]
\[
G_2(2)' \cong {\rm U}_3(3), \; {}^2G_2(3)' \cong {\rm L}_2(8)
\]
(see \cite[Proposition 2.9.1]{KL}). For example, if $G = {\rm L}_2(4)$ and $r$ is odd, then we have $|\mathcal{M}(R)| = 1$ if and only if $r = r_2 = 5$.

\item[{\rm (b)}] In both tables, $T$ is a simple group of Lie type over $\mathbb{F}_q$ (with $q=p^f$ and $p$ is a prime) and we use the notation
$\a(m,\e)$ to denote the following condition:
\[
\mbox{$(q^{1/k})^m \not\equiv \e \imod{r}$ for all $k \in \pi(f) \setminus \{r\}$,}
\]
where $m$ is a positive integer, $\e \in \{\pm 1\}$ and $\pi(f)$ is the set of prime divisors of $f$. Similarly, we write $\b(m,\e)$ for the same condition, but restricted to the \emph{odd} primes in $\pi(f) \setminus \{r\}$. We also write $s = \square \imod{p}$ and $s = \boxtimes \imod{p}$ to denote that $s$ is a square or nonsquare modulo $p$, respectively. The same notation is also used in Tables \ref{tab:NGR} and \ref{tab:Or}.
\item[{\rm (c)}] In the third column of Table \ref{tab:main1}, $P_1$ denotes the stabiliser in $G$ of a totally singular $1$-space. And in the final row of Table \ref{tab:main2}, we write $d_p(r)$ for the order of $p$ modulo $r$. 
\end{itemize}
\end{rem}

\begin{rem}\label{r:NGR}
We record some comments on Table \ref{tab:NGR}.
\begin{itemize}\addtolength{\itemsep}{0.2\baselineskip}
\item[{\rm (a)}] In view of the isomorphisms $A_5 \cong {\rm L}_2(4) \cong {\rm L}_2(5)$ and $A_6 \cong {\rm L}_2(9)$, we do not record the cases $(T,r) = (A_5,2)$, $(A_5,5)$ and $(A_6,3)$ in Table \ref{tab:NGR}. We also exclude the case $(T,r) = ({\rm L}_2(5),2)$. In addition, the groups with socle ${\rm PSp}_4(2)'$, $G_2(2)'$ and ${}^2G_2(3)'$ are omitted due to the isomorphisms ${\rm PSp}_4(2)' \cong {\rm L}_2(9)$, $G_2(2)' \cong {\rm U}_3(3)$ and ${}^2G_2(3)' \cong {\rm L}_2(8)$.

\item[{\rm (b)}] In the first row of Table \ref{tab:NGR}, we write $\mathcal{P}$ for the set of primes $r$ that are of the form $(q^d-1)/(q-1)$ for some prime power $q$ and integer $d \geqs 2$. Note that the first $240$ prime numbers of this form are recorded in \cite[Table II]{BS}.
\end{itemize}
\end{rem}

{\scriptsize
\begin{table}
\renewcommand\thetable{A}
\[
\begin{array}{lll} \hline
T &  \mbox{Type of $H$} & \mbox{Conditions} \\ \hline
{\rm L}_2(q) & P_1 & p = 2 \\
& 2^{1+2}_{-}.{\rm O}_{2}^{-}(2) & q = 5 \\
& {\rm GL}_1(q) \wr S_2 & \mbox{$q \equiv 1 \imod{4}$, $q \geqs 9$ and either $f = 2^a >1$ and $G \not\leqs {\rm P\Sigma L}_2(q)$, or} \\
& & \mbox{$f=1$ and either $|R_0| \geqs 2^4$, or $|R_0| = 2^3$ and $G = {\rm PGL}_2(q)$} \\
& {\rm GL}_1(q^2) & \mbox{$q = p \equiv 3 \imod{4}$ and either $|R_0| \geqs 2^4$, or $|R_0| = 2^3$ and $G = {\rm PGL}_2(q)$} \\ 
{\rm L}_3(q) & P_{1,2} & \mbox{$p=2$, $G \not\leqs {\rm P\Gamma L}_3(q)$} \\
{\rm U}_3(q) & P_1 & p = 2 \\
{\rm PSp}_4(q) & P_{1,2} & \mbox{$p=2$, $q \geqs 4$, $G \not\leqs {\rm P\Gamma Sp}_4(q)$} \\
{\rm L}_n(q) & {\rm GL}_{n-1}(q) \times {\rm GL}_1(q) & \mbox{$n = 2^k+1 \geqs 3$, $q =p \equiv 3 \imod{4}$, $G \not\leqs {\rm P\Gamma L}_n(q)$} \\
{\rm U}_n(q) & {\rm GU}_{n-1}(q) \times {\rm GU}_1(q) & \mbox{$n = 2^k+1 \geqs 3$, $q \equiv 1 \imod{4}$ and $f = 2^a \geqs 1$, with $q \geqs 9$ if $n=3$} \\
\O_n(q) & {\rm O}_{n-1}^{+}(q) \times {\rm O}_1(q) & \mbox{$n = 2^k+1 \geqs 9$ and $f = 2^a \geqs 1$, with $q \equiv \pm 1 \imod{8}$ if $f=1$} \\
{\rm P\O}_{n}^{+}(q) & {\rm O}_{n-2}^{+}(q) \times {\rm O}_2^{+}(q) & \mbox{$n = 2^k+2 \geqs 10$, $q=p \equiv 3 \imod{4}$, $G \not\leqs {\rm PO}_{n}^{+}(q)$} \\
{\rm P\O}_{n}^{-}(q) & {\rm O}_{n-2}^{+}(q) \times {\rm O}_2^{-}(q) & \mbox{$n = 2^k+2 \geqs 10$, $q \equiv 1 \imod{4}$, $f = 2^a \geqs 1$, $G \not\leqs T.\la \varphi \ra$} \\ \hline
\end{array}
\]
\caption{The cases $(G,H)$ in Theorem \ref{t:prime2} with $T$ classical}
\label{tab:prime2}
\end{table}}

{\scriptsize
\begin{table}
\renewcommand\thetable{B}
\[
\begin{array}{llll} \hline
T & r & \mbox{Type of $H$} & \mbox{Conditions} \\ \hline
{\rm L}_2(q) & p & P_1 & \\ 
& r_2 & {\rm GL}_1(q^2) & \mbox{$f \leqs 2$ and either $r >5$ or $|R| > r$} \\ 
& & & \mbox{$f>2$ and either $\a(1,-1)$, or} \\
& & & \mbox{$(r,p) = (3,2)$ and $q^{1/k} \not\equiv -1 \imod{3}$ for all $k \in \pi(f) \setminus \{3,f\}$} \\
& & & \\
{\rm L}_n(q) & r_1 & {\rm GL}_1(q) \wr S_n & \mbox{$n=r=3$ and either $f=1$ and $q \equiv 1 \imod{9}$, or $f = 3^a>1$} \\
& & & \mbox{$n=r \geqs 5$, $f$ odd, $\a(1,1)$} \\ 
& & {\rm GU}_n(q^{1/2}) & \mbox{$(n,q,r) = (3,4,3)$, $G = {\rm PGL}_3(4) = T.3$} \\
& r_n & {\rm GL}_{n/t}(q^t) & \mbox{$n = t^a$, $t \geqs 3$ prime, $\a(n,1)$ and either $f >1$ is odd, or} \\
& & & \mbox{$f=1$ and either $|R|>r$, or $r \ne 2n+1$, or $-r = \boxtimes \imod{p}$} \\
& & & \\
{\rm U}_3(q) & p & P_1 & \\
& r_2 & {\rm GU}_1(q) \wr S_3 & \mbox{$r=3$ and either $f=1$ and $q \equiv -1 \imod{9}$, or $f = 3^a>1$} \\
& r_6 & {\rm L}_2(7) & (q,r) = (3,7) \\ 
& & {\rm GU}_{1}(q^3) & \mbox{$\b(3,-1)$ and either $f>1$, or $f=1$ and either $|R| > r$ or $r > 7$} \\
& & & \\
{\rm U}_n(q) & r_2 & {\rm GU}_1(q) \wr S_n & \mbox{$n=r \geqs 5$, $\b(1,-1)$} \\
& r_{2n-2} & {\rm GU}_{n-1}(q) \times {\rm GU}_1(q) & \mbox{$n \geqs 4$ even, $\b(n-1,-1)$ and either $f>1$, or} \\
& & & \mbox{$f=1$ and either $|R|>r$, or $r \ne 2n-1$, or $-r = \square \imod{p}$} \\
& r_{2n} & {\rm L}_2(11) & (n,q,r) = (5,2,11) \\
& & {\rm GU}_{n/t}(q^t) & \mbox{$n = t^a$, $t \geqs 3$ prime, $\b(n,-1)$ and either $f >1$, or} \\
& & & \mbox{$f=1$ and either $|R|>r$, or $r \ne 2n+1$, or $-r = \square \imod{p}$} \\ 
& & & \\
{\rm PSp}_n(q) & r_2 & {\rm O}_n^{-}(q) & \mbox{$(n,q,r) = (6,2,3)$} \\ 
& r_n & {\rm Sp}_{n/2}(q^2) & 
\mbox{$n = 2^a \geqs 4$, $q$ odd, $\a(n,1)$ and either $f>2$, or} \\
& & & \mbox{$f=2$ and either $|R|>r$, or $r \ne 2n+1$, or $r = \square \imod{p}$, or} \\ 
& & & \mbox{$f=1$ and either $|R|>r$, or $r > 2n+1$, or $r=2n+1$ and $r = \boxtimes \imod{p}$} \\
& & & \\
\O_n(q) & r_2 & {\rm O}_{n-1}^{-}(q) \times {\rm O}_1(q) & \mbox{$n=2r+1 \geqs 7$, $\a(1,-1)$ and either $r>3$, or $f>1$, or $q \equiv -1 \imod{9}$} \\ 
& r_i & {\rm O}_{n-1}^{-}(q) \times {\rm O}_1(q) & \mbox{$4 \leqs i \leqs n-3$ even, $\a(i,1)$, $n-1 = r^ai$ with $a \geqs 1$} \\
& r_{n-1} & {\rm O}_{n-1}^{-}(q) \times {\rm O}_1(q) & \mbox{$n \geqs 9$, $\a(n-1,1)$ and either $r>2n-1$, or $|R|>r$, or $r =2n-1$ and either} \\
& & & \mbox{$f=2$ and $r = \square \imod{p}$, or $f=1$ and $r = \boxtimes \imod{p}$} \\
& & & \\
{\rm P\O}_n^{-}(q) & r_n & {\rm GU}_{n/2}(q) & \mbox{$n=2a$, $a \geqs 5$ prime, $\b(n/2,-1)$ and either $|R|>r$ or $r>n+1$} \\
& & {\rm O}_{n/2}^{-}(q^2) & \mbox{$n=2^a \geqs 8$, $\b(n/2,-1)$ and either $|R|>r$ or $r>n+1$} \\ \hline
\end{array}
\]
\caption{The triples $(G,r,H)$ in Theorem \ref{t:main1} with $T$ classical}
\label{tab:main1}
\end{table}}

{\scriptsize
\begin{table}
\renewcommand\thetable{C}
\[
\begin{array}{llll} \hline
T & r & \mbox{Type of $H$} & \mbox{Conditions} \\ \hline
{}^2B_2(q) & r_4 & q \pm \sqrt{2q} +1 & \mbox{$(q^{1/k})^2 \not\equiv -1 \imod{r}$ for all $k \in \pi(f) \setminus \{r,f\}$} \\
& & & \\
{}^2G_2(q) & 3 & [q^3]{:}(q-1) & \\
& r_6 & q \pm \sqrt{3q} +1 & \a(3,-1) \\
& & & \\
G_2(q) & r_{(3-\e)/2} & A_2^{\e}(q) & \mbox{$r=3$ and $f = 3^a \geqs 1$, with $q \equiv \e \imod{9}$ if $p \geqs 5$ and $f=1$} \\  
& r_{3(3-\e)/2} & A_2^{\e}(q) & \mbox{$\a(3,\e)$ and either $p=2$ and $(\e,q) \ne (-1,4)$, or} \\
& & & \mbox{$p \geqs 5$ and either $f>3$, or $|R|>r$, or $r>13$, or} \\
& & & \mbox{$f=3$ and either $r=13$ or $p \equiv \pm 1, \pm 3 \imod{9}$, or} \\
& & & \mbox{$f=2$ and $p = \square \imod{13}$, or} \\
& & & \mbox{$f=1$, $r=13$ and $p = \boxtimes \imod{13}$} \\
& & & \\
{}^3D_4(q) & r_{(3-\e)/2} & A_2^{\e}(q) \times (q^2+ \e q+1) & \mbox{$r=3$, $f = 3^a \geqs 1$} \\
& r_{12} & q^4-q^2+1 & \mbox{$(q^{1/k})^6 \not\equiv -1 \imod{r}$ for all $k \in \pi(f) \setminus \{3,r\}$} \\
& & & \\
{}^2F_4(q)' & r_{12} & q^2 \pm \sqrt{2q^3}+q \pm \sqrt{2q}+1 & \mbox{$f \geqs 3$, $\a(6,-1)$} \\
& & & \\
F_4(q) & r_8 & \O_9(q) & \mbox{$p \geqs 3$, $\a(4,-1)$ and either $f>2$, or $r>17$, or $|R|>r$, or} \\
& & & \mbox{$f=2$ and either $p=3$ or $p = \square \imod{17}$, or} \\
& & & \mbox{$f=1$ and $p = \boxtimes \imod{17}$} \\
& r_{12} & {}^3\!D_4(q) & \mbox{$p \geqs 3$, $\a(6,-1)$ and either $f \ne 3$, or $r>13$, or $|R|>r$, or} \\
& & & \mbox{$f=3$ and $p \equiv \pm 1 \imod{7}$} \\
& & & \\
E_6(q) & r_9 & A_2(q^3) & \mbox{$\a(9,1)$ and either $f>2$, or $r>19$, or $|R|>r$, or} \\
& & & \mbox{$f = 2$ and $p = \square \imod{5}$} \\
& & & \mbox{$f=1$ and either $p = \boxtimes \imod{5}$ or $p = \boxtimes \imod{19}$} \\
& & & \\
{}^2E_6(q) & r_{18} & A_2^{-}(q^3) & \mbox{$\b(9,1)$ and either $f>1$, or $r>19$, or $|R|>r$, or} \\
& & & \mbox{$f=1$ and either $p = \boxtimes \imod{5}$ or $p = \square \imod{19}$} \\
& & & \\
E_7(q) & r_{18} & {}^2\!E_6(q) \times (q+1) & \mbox{$\a(18,1)$ and either $f>2$, or $r>37$, or $|R|>r$, or} \\
& & & \mbox{$f=2$ and $p  = \square \imod{37}$, or} \\
& & & \mbox{$f=1$ and either $(q,r)=(2,19)$, or $r=37$ and $p = \boxtimes \imod{37}$} \\ 
& & & \\
E_8(q) & r_{15(3-\e)/2} & q^8-\e q^7+\e q^5-q^4+\e q^3-\e q+1 & \mbox{$\a(30,1)$ and either $r>61$, or $|R|>r$, or $|R|=r=61$ and} \\
& & & \mbox{either $f>2$, or $f=2$, $i=15$ and $d_p(r) \in \{15,30\}$} \\
 \hline
\end{array}
\]
\caption{The triples $(G,r,H)$ in Theorem \ref{t:main1} with $T$ exceptional}
\label{tab:main2}
\end{table}}

{\scriptsize
\begin{table}
\renewcommand\thetable{D}
\[
\begin{array}{lllclll} \hline
G & r & H & & & & \\ \hline
{\rm M}_{11} & 11 & {\rm L}_2(11) & & {\rm HN} & 19 & {\rm U}_{3}(8){:}3_1 \\
{\rm M}_{22} & 11 & {\rm L}_2(11) & & {\rm J}_4 & 29 & 29{:}28 \\
{\rm M}_{23} & 23 & 23{:}11 & & & 43 & 43{:}14 \\ 
{\rm He} & 17 & {\rm Sp}_{4}(4){:}2 & & {\rm Ly} & 37 & 37{:}18 \\
{\rm Ru} & 29 & {\rm L}_2(29) & & & 67 & 67{:}22 \\
{\rm Co}_{2} & 23 & {\rm M}_{23} & & \mathbb{B} & 47 & 47{:}23 \\
{\rm Co}_{3} & 23 & {\rm M}_{23} & & \mathbb{M} & 47 & 2.\mathbb{B} \\
{\rm J}_1 & 19 & 19{:}6 & & & 59 & {\rm L}_2(59) \\
{\rm J}_3 & 3 & 3^2.3^{1+2}{:}8 & & & 71 & {\rm L}_2(71) \\
{\rm Fi}_{24}' & 29 & 29{:}14 & & & & \\ \hline
\end{array}
\]
\caption{The triples $(G,r,H)$ in Theorem \ref{t:main1} with $T$ sporadic}
\label{tab:spor}
\end{table}}

{\scriptsize
\begin{table}
\renewcommand\thetable{E}
\[
\begin{array}{llll} \hline
T &  r & \mbox{Type of $N_G(R_0)$} & \mbox{Conditions} \\ \hline
A_r & r & {\rm AGL}_1(r) \cap G & \mbox{$r \geqs 13$, $r \ne 23$, $r \not\in \mathcal{P}$} \\
{\rm L}_2(q) & p & P_1 & \\
& 2 & {\rm GL}_1(q) \wr S_2 & \mbox{$q=9$ and $G \not\leqs {\rm P\Sigma L}_2(q)$, or $q=p = 2^k+1 \geqs 17$} \\
& 2& {\rm GL}_1(q^2) & \mbox{$G = {\rm PGL}_2(7)$ or $q=p =2^k-1 \geq 31$}  \\
& r_2 & {\rm GL}_1(q^2) & \mbox{$f \leqs 2$ and either $r >5$ or $|R| > r$} \\
& & & \mbox{$f>2$ and either $\a(1,-1)$, or} \\
& & & \mbox{$(r,p) = (3,2)$ and 
$q^{1/k} \not\equiv -1 \imod{3}$ for all $k \in \pi(f) \setminus \{3,f\}$}  \\
{\rm L}_3(q) & 2 & P_{1,2} & \mbox{$p =2$, $G \not\leqs {\rm P\Gamma L}_3(q)$}  \\
 & 3 & {\rm GU}_3(q^{1/2}) & G = {\rm PGL}_3(4)  \\
{\rm U}_3(q) & p & P_1 &  \\ 
& 3 & {\rm GU}_1(q) \wr S_3 & q = 8  \\
 & r_{6} & {\rm GU}_1(q^3) & \mbox{$\b(3,-1)$ and either $f >1$, or $|R| >r$, or $r >7$} \\
  {\rm PSp}_4(q) & 2 &   P_{1,2} & \mbox{$p=2$, $q \geq 4$, $G \not\leqs {\rm P\Gamma Sp}_4(q)$} \\
{\rm L}_n(q) & r_n & {\rm GL}_1(q^n) & \mbox{$n \geqs 3$ prime, $\a(n,1)$ and either $f >1$ is odd, or}  \\
 & & & \mbox{$f=1$ and either $|R|>r$, or $r \ne 2n+1$, or $-r = \boxtimes \imod{p}$} \\
{\rm U}_n(q) & r_{2n} & {\rm GU}_1(q^n) & \mbox{$n \geqs 3$ prime, $\b(n,-1)$ and either $f >1$, or}  \\
 & & & \mbox{$f=1$ and either $|R|>r$, or $r \ne 2n+1$, or $-r = \square \imod{p}$}  \\
{}^2B_2(q) & 2 & q^{1+1}{:}(q-1) &  \\
& r_4 & q \pm \sqrt{2q}+1 &   \mbox{$(q^{1/k})^2 \not\equiv -1 \imod{r}$ for all $k \in \pi(f) \setminus \{r,f\}$} \\
{}^2G_2(q) & 3 & [q^3]{:}(q-1) &  \\
&r_6 &  q \pm \sqrt{3q} +1 & \a(3,-1)  \\
{}^3D_4(q) & r_{12} & q^4-q^2+1 &   \mbox{$(q^{1/k})^6 \not\equiv -1 \imod{r}$ for all $k \in \pi(f) \setminus \{3,r\}$} \\
{}^2F_4(q) & r_{12} & q^2 \pm \sqrt{2q^3}+q \pm \sqrt{2q}+1 & \mbox{$f \geqs 3$, $\a(6,-1)$} \\
E_8(q) & r_{15(3-\e)/2} & q^8 - \e q^7 + \e q^5 - q^4 + \e q^3 - \e q +1 & 
 \mbox{$\a(30,1)$ and either $r>61$, or $|R|>r$, or $|R|=r=61$ and}  \\
& & & \mbox{either $f>2$, or $f=2$, $i=15$ and $d_p(r) \in \{15,30\}$} \\ 
{\rm M}_{23} & 23 & 23{:}11 &  \\ 
{\rm J}_{1} & 19 & 19{:}6 & \\ 
{\rm J}_{3} & 3 & 3^2.3^{1+2}{:}8 &  \\ 
{\rm J}_{4} & 29 & 29{:}28 &  \\ 
 & 43 & 43{:}14 & \\ 
{\rm Ly} & 37 & 37{:}18 &  \\
& 67 & 67{:}22 & \\
{\rm Fi}_{24}' & 29 & 29{:}14 &  \\ 
\mathbb{B} & 47 & 47{:}23 &  \\ \hline
\end{array}
\]
\caption{The pairs $(G,r)$ with $\mathcal{M}(R) = \{N_G(R_0)\}$ in Corollary  \ref{c:NGR}}
\label{tab:NGR}
\end{table}}

{\scriptsize
\begin{table}
\renewcommand\thetable{F}
\[
\begin{array}{llll} \hline
T &  r & \mbox{Type of $H$} & \mbox{Conditions} \\ \hline
A_9 & 3 & S_3 \wr S_3 & \\
{\rm L}_2(q) & 2 & {\rm GL}_1(q) \wr S_2 & \mbox{$q \equiv 1 \imod{4}$, $q \geqs 13$ and either $f = 2^a >1$ and $G \not\leqs {\rm P\Sigma L}_2(q)$,}  \\
& & & \mbox{or $q=p$ not Fermat and $|R| \geqs 2^4$} \\
& 2 & {\rm GL}_1(q^2) & \mbox{$q = p \equiv 3 \imod{4}$ not Mersenne and $|R| \geqs 2^4$} \\ 
{\rm L}_n(q) &  2 &  {\rm GL}_{n-1}(q) \times  {\rm GL}_1(q) & \mbox{$n = 2^k+1 \geqs 3$, $q =p \equiv 3 \imod{4}$, $G = T.2 \not\leqs {\rm P\Gamma L}_n(q)$} \\
              & r_1 & {\rm GL}_1(q) \wr S_n & \mbox{$n=r=3$ and either $f=1$ and $q \equiv 1 \imod{9}$, or $f = 3^a>1$} \\
& & & \mbox{$n=r \geqs 5$, $f$ odd, $\a(1,1)$} \\ 
{\rm U}_n(q)& 2 &{\rm GU}_{n-1}(q) \times {\rm GU}_1(q)& \mbox{$n = 2^k+1$, $q \equiv 1 \imod{4}$ and $f = 2^a \geqs 1$, with $q \geq 9$ if $n = 3$} \\ 

& r_2 & {\rm GU}_1(q) \wr S_n & \mbox{$n=r=3$ and either $f=1$ and $q \equiv -1 \imod{9}$, or $f = 3^a>1$  and $q \ne 8$} \\
& & & \mbox{$n=r \geqs 5$, $\b(1,-1)$} \\

 \O_n(q) & 2 & {\rm O}_{n-1}^{+}(q) \times {\rm O}_1(q) & \mbox{$n = 2^k+1 \geqs 9$ and $f = 2^a \geqs 1$, with $q \equiv \pm 1 \imod{8}$ if $f=1$}  \\
 {\rm P\O}_{n}^{+}(q) & 2 &{\rm O}_{n-2}^{+}(q) \times {\rm O}_2^{+}(q) & \mbox{$n = 2^k+2 \geqs 10$, $q=p \equiv 3 \imod{4}$, $G \not\leqs {\rm PO}_{n}^{+}(q)$}  \\
  {\rm P\O}_{n}^{-}(q) & 2 & {\rm O}_{n-2}^{+}(q) \times {\rm O}_2^{-}(q) & \mbox{$n = 2^k+2 \geqs 10$, $q \equiv 1 \imod{4}$, $f = 2^a \geqs 1$, $G \not\leqs T.\la \varphi \ra$}\\
  G_2(q) & r_{(3-\e)/2} & {\rm SL}_3^{\e}(q) & \mbox{$r=3$ and $f = 3^a \geqs 1$, with $q \equiv \e \imod{9}$ if $p \geqs 5$ and $f=1$} \\
{}^3D_4(q) & r_{(3-\e)/2} & A_2^{\e}(q) \times (q^2+ \e q+1) & \mbox{$r=3$, $f = 3^a \geqs 1$} \\ \hline
\end{array}
\]
\caption{The triples $(G,r,H)$ with $O_r(H) \ne 1$ in Corollary \ref{c:Or}(ii)}
\label{tab:Or}
\end{table}}

\end{document}